\documentclass[12pt]{amsart}
\usepackage{amsmath}
\usepackage{amssymb,bbm}
\usepackage{mathabx}
\usepackage{IEEEtrantools}
\usepackage{mathtools}
\usepackage{mathrsfs}
\usepackage{url}
\usepackage{hyperref}
\usepackage{xcolor}
\usepackage{color}
\usepackage{tikz-cd}
\usepackage{multirow}
\usepackage{cases}
\usetikzlibrary{arrows,decorations.markings}
\usetikzlibrary{fit,matrix}

\def \qed { \mbox{}\hfill
	$\Box$\vspace{1ex}}

\newtheorem{definition}{Definition}[section]
\newtheorem{theorem}[definition]{Theorem}
\newtheorem{conjecture}[definition]{Conjecture}
\newtheorem{lemma}[definition]{Lemma}

\newtheorem{remark}[definition]{Remark}

\title[]{Quantum unitary group $U_{q, \Theta}(3)$ for complex deformation parameters}
\author{{\sc Manabendra Giri}}
\address[Manabendra Giri]{Department of Mathematics, Indian Institute of Technology Bombay, Powai, Mumbai}
\email{manab@math.iitb.ac.in} 
\author{{\sc Debabrata Jana}}
\address[Debabrata Jana]{Department of Mathematics, Ashoka University, Haryana,  India}
\email{debabrata.jana05@gmail.com}

\begin{document}
\maketitle

\begin{abstract} 
In this article, we consider a particular Hayashi $R$-matrix that satisfies the Yang-Baxter equation $R_{12}R_{13}R_{23}=R_{23}R_{13}R_{12}$. Using the  FRT-bialgebra technique and Woronowicz's method of  construction, we construct a concrete compact quantum group $U_{q,\Theta}(3)$ for non zero real $q$ and modulus one complex deformation parameters $\theta_{ij}$. We then study in detail the irreducible $*$-representations of the underlying $C^*$-algebra $C(U_{q,\Theta}(3))$, using the representations of the three dimensional noncommutative torus. Also, a monomial basis for the dense Hopf *-algebra  $\mathbb{C}[U_{q,\Theta}(3)]$ is obtained.	
\end{abstract}

{\bf AMS Subject Classification No.:} 
20G42, 
46L67, 
47L55  
\\
{\bf Keywords.} Compact quantum group, $(q, \Theta)$-deformation, $*$-representations.
	
\section{Introduction}\label{sec1}
In the theory of compact quantum group (CQG) of Woronowicz \cite{wor87, wor87cmp}, the first non-trivial and the most studied example is the $SU_q(2)$ for $q\in\mathbb{R}\setminus\{0\}$. It is widely investigated in the literature through different perspective. In \cite{kasprzak2016braided}, Woronowicz et al. defined a family of $q$-deformations of $SU(2)$ for $q\in\mathbb{C}\setminus\{0\}$. This agrees with the compact quantum group $SU_q(2)$ when $q$ is real but for $q\in\mathbb{C}\setminus\mathbb{R},\,SU_q(2)$ is not a CQG, rather a braided quantum group in a suitable tensor category. In \cite{meyer2016quantum}, it is shown that the quantum analogue of the semidirect product construction for groups turns the braided quantum group $SU_q(2)$ into a genuine CQG. This CQG is the coopposite of the compact quantum group $U_q(2)$ defined in \cite{zhangzhao02}, which we refer as the quantum unitary group. In \cite{jana2024some}, it is shown that for complex deformation parameter $q$, $U_q(2)$ is `in some sense' the first concrete example of a CQG, since $SU_q(2)$ is no longer a CQG when $q$ is complex.

Let \begin{align*}
	\mathbbm{1}=\begin{bmatrix}1&1&\cdots& 1\\1&1&\cdots&1\\\vdots&\vdots&\ddots&\vdots\\1&1&\cdots&1\end{bmatrix}&\qquad\qquad\text{and}& \Theta=\begin{bmatrix}1&\theta_{1,2}&\cdots& \theta_{1,n}\\\theta_{2,1}&1&\cdots&\theta_{2,n}\\\vdots&\vdots&\ddots&\vdots\\\theta_{n,1}&\theta_{n,2}&\cdots&1\end{bmatrix}
\end{align*} such that $\lvert\theta_{i,j}\rvert=1$ and $\overline{\theta_{j,i}}=\theta_{i,j}$ for $i\neq j$ . In \cite{koelink91}, Koelink studied the quantum unitary group $U_{q,\mathbbm{1}}(n)$, for positive $q$.
which is a compact quantum group in Woronowicz's framework. Whereas Zhang and Zhao studied the quantum unitary group $U_{q,\Theta}(2)$ in \cite{zhangzhao02}, which is a genuine compact quantum group for $q \in (0,\infty)$.
Zhang and Zhao constructed $U_{q,\Theta}(2)$ following Woronowicz's method of construction in \cite{wor87} by considering the FRT bialgebra construction technique in \cite{faddeev1988quantization} for a particular Hayashi’s $R$-matrix in \cite{hayashi90} which satisfies the Yang-Baxter equation. This construction agrees with Koelink’s construction when $\theta_{1,2}=\theta_{2,1}=1$, and with the construction of Connes and DuboisViolette in \cite{convio02} when $q=1$. Motivated from Connes and DuboisViolette we consider all complex parameter of Hayashi $R$-matrix of type $A$ are of unit modulus. Therefore, it is natural to ask the following questions:

\begin{itemize}
\item Given a Hayashi $R$-matrix $(n>2)$ of type $A$ or $C$, is it possible to construct a nontrivial compact quantum group following Woronowicz's method of construction?
    \item Find all irreducible $*$-representations of the underlying $C^*$-algebra of the compact quantum group.
    \item What are possible braided analogues and it's semidirect product quantum group?
\end{itemize}

In this article, we are interested to address the first question for $n=3$ and construct the quantum unitary group $C(U_{q,\Theta}(3))$.
Moreover, Zhang and Zhao classified all irreducible representations of $U_{q,\Theta}(2)$ using the representations of the irrational rotation algebra. We shall use the representations of the three dimensional noncommutative torus in order to obtain all irreducible representations of $C(U_{q,\Theta}(3))$.  For $0<q<1$ and $\Theta=\mathbbm{1}$, Bragiel studied the irreducible $*$-representation of $C^*$-algebra $C(SU_{q,\Theta}(3))$ in \cite{bragiel1989twisted}  and at the label of $C^*$-algebras  $C(U_{q,\Theta}(3)) \cong C(SU_{q}(3))\otimes C(\mathbb{T})$. Given an $R$-matrix satisfies Yang-Baxter equation there are two methods to construct a Hopf algebra namely $SL$-like construction and $GL$-like construction. The strategy given in \cite{hayashi90} that first find group like element in the FRT-bialgebra then show that quotient has an antipode. Whereas in $SL$-like construction one needs the group like element to be central but in $GL$-like construction the group like element not necessarily central and then show localization bialgebra has an antipode. In $GL$-like construction, where the quantum determinant is central has been studied by Anna Kula in \cite{kula2015woronowicz} for $n=3$, following Woronowicz's method of construction in \cite{wor87cmp}. The same case is also investigated by Guo et al. and Jiang in \cite{qian1992new} and \cite{lining2004multi}, following the approach of \cite{hayashi90} together with Woronowicz's construction. In our setting, the quantum determinant is generally not central; however, for certain parameter values, it becomes central and agrees with the construction in \cite{kula2015woronowicz}, \cite{lining2004multi}.  Our main intention to prepare this article is to build a ground for studying its noncommutative geometric and probabilistic aspects. The equivariant spectral triple for the quantum unitary group $U_q (2)$ is studied by Guin and Saurabh in \cite{guin2023equivariant}, whereas the probabilistic aspect of the quantum unitary group for complex deformation parameters is not studied yet. \\
\smallskip

Brief description of our work is the following. 
 In Sec. $2$, we recall FRT-bialgebra construction, quantum minor, quantum determinant, the Hopf $*$-algebra $\mathbb{C}[U_{q,\Theta}(n)]$. One can find following approaches in section 9.1,\, 9.2 of the book \cite{bookklisch97} and section 5,\,6 of article \cite{hayashi90}. In Sec. $3$, we present a bases for $\mathcal{A}_{q,\Theta}(n)$ and $\mathbb{C}[U_{q,\Theta}(n)]$. In sec. $4$, we address the $C^*$- algebra $C(U_{q,\Theta}(n))$.  In Sec. $5$, using Hayashi $R$-matrix and following Woronowicz's method of construction we construct the compact quantum group  $U_{q,\Theta}(3)$ and describe the monomial basis of it's dense $*$-algebra. In Sec. $6$, we shall discuss the irreducible $*$-representations of underlying $C^*$-algebra of the quantum unitary group $U_{q,\Theta}(3)$, using the representations of the three dimensional noncommutative torus.
We also recall few underlying $C^*$-algebras $C(U_{q,\mathbbm{1}}(n))$, $C(U_{1,\Theta}(n))$, $C(U_{q,\Theta}(2))$, Disk algebra and Higher dimensional non commutative torus in appendix. 
\medskip

\textbf{Notation:} Here we fix some notation which we shall be using throughout the article.

\begin{itemize}
\item The standard orthonormal basis of $\ell^2(\mathbb{N})$ (or $\ell^2(\mathbb{Z})$) is denoted by $\{e_n:n\in \mathbb{N}\}$ (or $n\in \mathbb{Z}$ depending on context).
\item $N$ denotes the number operator $e_n\mapsto ne_n$ acting on $\ell^2(\mathbb{N})$ or $\ell^2(\mathbb{Z})$
\item $S$ denotes the backward shift operator $e_n\mapsto e_{n-1}$ on $\ell^{2}(\mathbb{N})$ (or the backward shift operator on $\ell^{2}(\mathbb{Z})$ depending on context).
\end{itemize}

\section{The bialgebra $\mathcal{A}_{q,\Theta}$ and Hopf *-algebra $\mathbb{C}[U_{q,\Theta}(n)]$}\label{sec:Hopf-ast-algebra}
It is well known\cite{bookklisch97} that for a linear space $H$ with a basis $\{e_1,e_2,\cdots,e_n\}$ and a linear map $R:H\otimes H\longrightarrow H\otimes H$ such that
\begin{align*}
    R:e_i\otimes e_j\to \sum_{k,l=1}^{n}R_{kl,ij}e_k\otimes e_l,
\end{align*}
then there exists a bialgebra $A(R)$ (Faddeev-Reshetikhin-Takhtajan construction ) (See appendix \ref{sec:FRT}) as follows.
Here $A(R)$ is a bialgebra with unit, counit $\epsilon$ and coproduct $\Delta$ generated by the generators $V_{i,j}$ satisfying the relations:
\begin{align*}
    \sum_{k,l=1}^{n}R_{ji,kl} V_{k,r}V_{l,s}&=\sum_{k,l=1}^{n}R_{lk,rs} V_{i,k}V_{j,l}\\
    \Delta(V_{i,j})&=V_{i,k}\otimes V_{k,j}\\
    \epsilon(V_{i,j})&=\delta_{i,j}
\end{align*}

Let $w_1,w_2,\cdots w_n\in S^1$.

Consider the algebra $\mathcal{A}(R)^{ext}$ generated by the generators $\{\mathscr{D}^{-1},V_{i,j}~:~1\leq i,j\leq n\}$ satisfying the relations 
\begin{align}
\sum_{k,l=1}^{n}R_{ji,kl} V_{k,r}V_{l,s}&=\sum_{k,l=1}^{n}R_{lk,rs} V_{i,k}V_{j,l},\\
w_k V_{i,k}\mathscr{D}^{-1}&=w_i \mathscr{D}^{-1}V_{i,k}
\end{align}

Therefore we see that $A(R)$ is a subalgebra of $\mathcal{A}(R)^{ext}$ and $\mathcal{A}(R)^{ext}=\oplus_{m=0}^{\infty}(\mathscr{D}^{-1})^mA(R)$ as a vector space.
\begin{theorem}
There is a unique bialgebra structure on the algebra $\mathcal{A}(R)^{ext}$ such that
\begin{align}
    \Delta(V_{i,j})=\sum_k V_{i,k}\otimes V_{k,j}&&
    \epsilon(V_{i,j})=\delta_{i,j}\\
    \Delta(\mathscr{D}^{-1})=\mathscr{D}^{-1}\otimes \mathscr{D}^{-1}&&
    \epsilon(\mathscr{D}^{-1})=1.
\end{align}
\end{theorem}
\begin{proof}
As $\mathbb{C}\langle V_{i,j},\mathscr{D}^{-1}\rangle$ is the free algebra generated by $\{V_{i,j}~:~i\leq i,j\leq n\}\cup \{\mathscr{D}^{-1}\}$, there are unique algebra homomorphisms $\Delta:\mathbb{C}\langle V_{i,j},\mathscr{D}^{-1}\rangle\longrightarrow\mathbb{C}\langle V_{i,j},\mathscr{D}^{-1}\rangle\otimes \mathbb{C}\langle V_{i,j},\mathscr{D}^{-1}\rangle$ and $\epsilon:\mathbb{C}\langle V_{i,j},\mathscr{D}^{-1}\rangle\longrightarrow\mathbb{C}$ such that 
\begin{align*}
    \Delta(V_{i,j})=\sum_k V_{i,k}\otimes V_{k,j},&&
    \epsilon(V_{i,j})=\delta_{i,j}\\
    \Delta(\mathscr{D}^{-1})=\mathscr{D}^{-1}\otimes \mathscr{D}^{-1},&&
    \epsilon(\mathscr{D}^{-1})=1
\end{align*}
holds.
Let $I(R_0)$ be the ideal generated by $\{I_{ij,rs}~, J_{i,j}~:~1\leq i,j,r,s\leq n\}$ where
\begin{align*}
    I_{ij,rs}=\sum_{k,l=1}^{n}\Big(R_{ji,kl} V_{k,r}V_{l,s}-R_{lk,rs} V_{i,k}V_{j,l}\Big),&&
J_{i,j}=\Big(w_j V_{i,j}\mathscr{D}^{-1}-w_i \mathscr{D}^{-1}V_{i,j}\Big)
\end{align*}
Since $\Delta(I_{ij,rs})=\sum_{kl} I_{ij,kl}\otimes V_{k,r}V_{l,s}+V_{i,k}V_{j,l}\otimes I_{kl,rs}$, $\Delta(J_{i,j})=\sum_k (J_{i,k}\otimes V_{k,j}\mathscr{D}^{-1}+\mathscr{D}^{-1}V_{i,k}\otimes J_{k,j}) $ ,$\epsilon(I_{ij,rs})=0$ and $\epsilon(J_{i,j})=0$; we get $\Delta(I(R_0))\subseteq I(R_0)\otimes \mathbb{C}\langle V_{i,j},\mathscr{D}^{-1}\rangle + \mathbb{C}\langle V_{i,j},\mathscr{D}^{-1}\rangle\otimes I(R_0)$ and $\epsilon(I(R_0))=\{0\}$. That is $I(R_0)$ is a biideal, so the quotient algebra $\mathcal{A}(R)^{ext}=\mathbb{C}\langle V_{i,j},\mathscr{D}^{-1}\rangle/I(R_0)$.
\end{proof}
\subsection{Algebra $\mathcal{A}_{q,\Theta}(n)$, $\widehat{\mathcal{A}}_{q,\Theta}(n)$ and $\mathbb{C}[U_{q,\Theta}(n)]$ :}
Let $q$ be a positive real number. We consider the special R-matrix in \cite{hayashi90} which satisfies Yang-Baxter equation $R_{12}R_{13}R_{23}=R_{23}R_{13}R_{12}$, known as Hayashi's R-matrix.
\begin{align}\label{R-matrix}
    R=\sum_{i=1}^n  qE_{i,i}\otimes E_{i,i}+\sum_{i\neq j}\theta_{j,i}E_{i,i}\otimes E_{j,j}+\sum_{i>j}(q-\frac{1}{q})E_{i,j}\otimes E_{j,i}
\end{align}
where $E_{i,j}$ denotes matrix units and complex numbers  $\theta_{i,j}$ such that $\theta_{i,i}=1$, $\overline{\theta_{i,j}}=\frac{1}{\theta_{i,j}}=\theta_{j,i}$.

For the above $R$ matrix we have the matrix entries $R_{ji,kl}$ such that  $R_{ji,kl} = q^{\delta_{i,j}}\theta_{i,j}\delta_{i,l}\delta_{j,k}+(q-\frac{1}{q})\delta_{i,k}\delta_{j,l}H(j-i)$ where $H$ is Heaviside symbol, that is $H(r)=1$ if $r>0$ and $H(r)=0$ if $r\leq 0$.

Therefore we have $R_{ii,ii}=q$ for all $i$, $R_{ij,ij}=\theta_{j,i}$ for all $i\neq j$, $R_{ij,ji}=(q-\frac{1}{q})$ for all $i> j$ and $R_{ij,k,l}=0$ for other cases.


For the above particular $R$-matrix, consider $w_i=\theta_{1,i}\theta_{2,i}\cdots \theta_{i-1,i}\theta_{i+1,i}\cdots \theta_{n,i}$ for all $i=1,2,...,n$ where $\theta_{i,j}\in S^1$.

For that particular $R$-matrix, we get the algebra $\mathcal{A}_{q,\Theta}(n)$:=$\mathcal{A}(R)$ generated by the generators $\{V_{i,j}~:~1\leq i,j\leq n\}$ satisfying the relations 
\begin{align}
V_{i,j}V_{i,l}&=q\theta_{l,j}V_{i,l}V_{i,j}&&\text{ if }j<l&&\forall i\\
V_{i,j}V_{k,j}&=q\theta_{i,k}V_{k,j}V_{i,j}&&\text{ if }i<k&&\forall j\\
\theta_{l,j}V_{i,l}V_{k,j}&=\theta_{i,k}V_{k,j}V_{i,l}&&\text{ if }j<l,~i<k&&\\
\theta_{j,l}V_{i,j}V_{k,l}-\theta_{i,k}V_{k,l}V_{i,j}&=(q-\frac{1}{q})V_{i,l}V_{k,j}&&\text{ if }j<l,~i<k&&
\end{align}
and the algebra $\widehat{\mathcal{A}}_{q,\Theta}(n)$:= $\mathcal{A}(R)^{ext}$ generated by the generators $\{\mathscr{D}^{-1}\}\cup\{V_{i,j}~:~1\leq i,j\leq n\}$ satisfying the relations 
\begin{align}
V_{i,j}V_{i,l}&=q\theta_{l,j}V_{i,l}V_{i,j}&&\text{ if }j<l&&\forall i\\
V_{i,j}V_{k,j}&=q\theta_{i,k}V_{k,j}V_{i,j}&&\text{ if }i<k&&\forall j\\
\theta_{l,j}V_{i,l}V_{k,j}&=\theta_{i,k}V_{k,j}V_{i,l}&&\text{ if }j<l,~i<k&&\\
\theta_{j,l}V_{i,j}V_{k,l}-\theta_{i,k}V_{k,l}V_{i,j}&=(q-\frac{1}{q})V_{i,l}V_{k,j}&&\text{ if }j<l,~i<k&&\\
V_{i,k}\mathscr{D}^{-1}&=\Big[\frac{\theta_{k,1}\theta_{k,2}\cdots \theta_{k,k-1}\theta_{k,k+1}\cdots \theta_{k,n}}{\theta_{i,1}\theta_{i,2}\cdots \theta_{i,i-1}\theta_{i,i+1}\cdots \theta_{i,n}}\big]\mathscr{D}^{-1}V_{i,k}
\end{align}
\subsection{Quantum determinant and quantum minor of $\mathcal{A}_{q,\Theta}(n)$}
Consider the following element for $i<k,~\&~j<l$,
\begin{align*}
 [D^{i,k}_{j,l}]&=V_{i,j}V_{k,l}-q\theta_{i,k}V_{k,j}V_{i,l}\\
 &=V_{i,j}V_{k,l}-q\theta_{l,j}V_{i,l}V_{k,j}\\
 &=\theta_{l,j}\theta_{i,k}[V_{k,l}V_{i,j}-\frac{1}{q\theta_{l,j}}V_{k,j}V_{i,l}]\\
 &=\theta_{l,j}\theta_{i,k}[V_{k,l}V_{i,j}-\frac{1}{q\theta_{i,k}}V_{i,l}V_{k,j}]
\end{align*}

Also for two sets $I=\{i_1<i_2<\cdots <i_r\}$ and $J=\{j_1<j_2<\cdots <j_r\}$
we can define

\begin{align*}
[D^I_J]&=\sum_{\sigma} \Big(\underset{\begin{matrix}x<y\\\sigma(x)>\sigma(y)\end{matrix}}{\prod}(-q~\theta_{i_{\sigma(y)},i_{\sigma(x)}})\Big)V_{i_{\sigma(1)},j_1}V_{i_{\sigma(2)},j_2}\cdots V_{i_{\sigma(r)},j_r}\\~\\
[\Bar{D}^I_J]&=\sum_{\sigma} \Big(\underset{\begin{matrix}x<y\\\sigma(x)>\sigma(y)\end{matrix}}{\prod}(-q\theta_{j_{\sigma(y)},j_{\sigma(x)}})\Big)V_{i_1,j_{\sigma(1)}}V_{i_2,j_{\sigma(2)}}\cdots V_{i_r,j_{\sigma(r)}}
\end{align*}

For $``r=2"$, we have $[D^{\{i,k\}}_{\{j,l\}}]=[\Bar{D}^{\{i,k\}}_{\{j,l\}}]=[D^{i,k}_{j,l}]$.

For $x<y$ but $i_{\sigma(x)}>i_{\sigma(y)}$, we have $\theta_{j_y,j_x}V_{i_{\sigma(y)},j_y}V_{i_{\sigma(x)},j_x}=\theta_{i_{\sigma(y)},i_{\sigma(x)}}V_{i_{\sigma(x)},j_x}V_{i_{\sigma(y)},j_y} $ and otherwise we don't need to swap. Therefore we have $[D^I_J]=[\Bar{D}^I_J]$.

Let $M_{i,k}=[D^{\{1,2,\cdots,i-1,i+1,\cdots,n\}}_{\{1,2,\cdots,k-1,k+1,\cdots,n\}}]$ for all $i,k$ and $Det_{q}=[D^{\{1,2,\cdots,n\}}_{\{1,2,\cdots,n\}}]$

\subsection{Coaction of $\mathcal{A}_{q,\Theta}(n)$ on the quantum exterior algebras}

We have left Exterior algebra $\Lambda_l=\{e_i:e_i^2=0,~q\theta_{j,i}e_ie_j+e_je_i=0\text{ for }i<j\}$. Also left coaction is given by $\delta_l:\Lambda_l\to \mathcal{A}_{q,\Theta}(n)\otimes \Lambda_l$ such that $\delta_l(e_i)=\sum_k V_{i,k}\otimes e_k$.

We have right Exterior algebra $\Lambda_r=\{e_i:e_i^2=0,~q\theta_{i,j}e_ie_j+e_je_i=0\text{ for }i<j\}$. Also right coaction is given by $\delta_r:\Lambda_r\to  \Lambda_r\otimes \mathcal{A}_{q,\Theta}(n)$ such that $\delta_r(e_i)=\sum_k  e_k\otimes V_{k,i}$. 

Therefore we have  
\begin{align}
\delta_l(e_ie_k)&=\sum_{j<l}[D^{i,k}_{j,l}]\otimes e_je_l&&\forall i<k\\
\delta_r(e_je_l)&=\sum_{i<k}e_ie_k\otimes [D^{i,k}_{j,l}]&&\forall j<l.
\end{align}

The vector $f_i$ in $\Lambda_l$ is given by $e_1e_2\cdots e_{i-1}e_{i+1}\cdots e_n$. Then we have $\delta_l(f_i)=\sum_k M_{i,k}\otimes f_k$. One can check that $Det_{q}$ is an element of $\mathcal{A}_{q,\Theta}(n)$ satisfying $\delta_l(e_1e_2\cdots e_n)=Det_{q}\otimes e_1e_2\cdots e_n$. Therefore using $\delta_{l}(e_1f_1)=Det_{q}\otimes e_1f_1$ we can got 
\begin{itemize}
    \item From $\delta_l(f_ie_j)$ we got 
    $\sum_k\Big[\frac{\theta_{n,k}\theta_{n-1,k}\cdots \theta_{k+1,k}}{\theta_{n,j}\theta_{n-1,j}\cdots \theta_{j+1,j}}\big](-q)^{j-k}M_{i,k}V_{j,k}=\delta_{i,j}Det_{q}$
    \item From $\delta_l(e_jf_i)$ we got 
    $\sum_k\Big[\frac{\theta_{k,1}\theta_{k,2}\cdots \theta_{k,k-1}}{\theta_{j,1}\theta_{j,2}\cdots \theta_{j,j-1}}\big](-q)^{k-j}V_{j,k}M_{i,k}=\delta_{i,j}Det_{q}$
\end{itemize}

The vector $f_i$ in $\Lambda_r$ is given by $e_1e_2\cdots e_{i-1}e_{i+1}\cdots e_n$. Then we have $\delta_r(f_i)=\sum_k f_k\otimes M_{k,i}$. One can check that $Det_{q}$ is an element of $\mathcal{A}_{q,\Theta}(n)$ satisfying $\delta_r(e_1e_2\cdots e_n)= e_1e_2\cdots e_n\otimes Det_{q}$. And using $\delta_{r}(e_1f_1)= e_1f_1\otimes Det_{q}$ we can get 
\begin{itemize}
    \item From $\delta_r(f_ie_j)$ we get 
    $\sum_k\Big[\frac{\theta_{k,n}\theta_{k,n-1}\cdots \theta_{k,k+1}}{\theta_{j,n}\theta_{j,n-1}\cdots \theta_{j,j+1}}\big](-q)^{j-k}M_{k,i}V_{k,j}=\delta_{i,j}Det_{q}$
    \item From $\delta_r(e_jf_i)$ we get 
    $\sum_k\Big[\frac{\theta_{1,k}\theta_{2,k}\cdots \theta_{k-1,k}}{\theta_{1,j}\theta_{2,j}\cdots \theta_{j-1,j}}\big](-q)^{k-j}V_{k,j}M_{k,i}=\delta_{i,j}Det_{q}$
\end{itemize}

Let $\kappa_k=\theta_{k+1,k}\theta_{k+2,k}\cdots \theta_{n,k}$, $\wp_k=\theta_{1,k}\theta_{2,k}\cdots \theta_{k-1,k}$ and $w_k=\theta_{1,k}\theta_{2,k}\cdots \theta_{k-1,k}\theta_{k+1,k}\cdots \theta_{n,k}$ for all $k=1,2,...,n$. Therefore we have $w_k=\wp_k\kappa_k$. So the above four relations are written as following:
\begin{align}
	\sum_k\Big[\frac{\kappa_k}{\kappa_j}\big](-q)^{j-k}M_{i,k}V_{j,k}&=\delta_{i,j}Det_{q}\label{q-minor-1}\\
	\sum_k\Big[\frac{\wp_j}{\wp_k}\big](-q)^{k-j}V_{j,k}M_{i,k}&=\delta_{i,j}Det_{q}\label{q-minor-2}\\
	\sum_k\Big[\frac{\kappa_j}{\kappa_k}\big](-q)^{j-k}M_{k,i}V_{k,j}&=\delta_{i,j}Det_{q}\label{q-minor-3}\\
	\sum_k\Big[\frac{\wp_k}{\wp_j}\big](-q)^{k-j}V_{k,j}M_{k,i}&=\delta_{i,j}Det_{q}\label{q-minor-4}
\end{align}

\subsection{Hopf *-algebra $\mathbb{C}[U_{q,\Theta}(n)]$ }
\begin{theorem}
    $\mathbb{C}[U_{q,\Theta}(n)]$ is the algebra generated by the generators $\{\mathscr{D}^{-1},V_{i,j}:1\leq i,j\leq n\}$ satisfying the relations 
\begin{align}\label{eqn: hopf* rel}\nonumber
V_{i,j}V_{i,l}&=q\theta_{l,j}V_{i,l}V_{i,j}&\text{ if }j<l&\forall i\\\nonumber
V_{i,j}V_{k,j}&=q\theta_{i,k}V_{k,j}V_{i,j}&\text{ if }i<k&\forall j\\
\theta_{l,j}V_{i,l}V_{k,j}&=\theta_{i,k}V_{k,j}V_{i,l}&\text{ if }j<l,\,i<k&\\\nonumber
\theta_{j,l}V_{i,j}V_{k,l}-\theta_{i,k}V_{k,l}V_{i,j}&=(q-\frac{1}{q})V_{i,l}V_{k,j}&\text{ if }j<l,\,i<k&
\\\nonumber
V_{i,k}\mathscr{D}^{-1}&=\Big[\frac{\theta_{k,1}\theta_{k,2}\cdots \theta_{k,k-1}\theta_{k,k+1}\cdots \theta_{k,n}}{\theta_{i,1}\theta_{i,2}\cdots \theta_{i,i-1}\theta_{i,i+1}\cdots \theta_{i,n}}\big]\mathscr{D}^{-1}V_{i,k}\\
1&=Det_q\mathscr{D}^{-1}
\end{align}
,where $Det_q=
\sum_{\sigma} \Big(\underset{\begin{matrix}x<y\\
\sigma(x)>\sigma(y)\end{matrix}}{\prod}(-q\,\theta_{{\sigma(y)},{\sigma(x)}})\Big)V_{{\sigma(1)},1}V_{{\sigma(2)},2}\cdots V_{{\sigma(n)},n}$\\
with unique commultiplication $\Delta$ and counit $\epsilon$ such that
\begin{align*}
    \Delta(V_{i,j})=\sum_k V_{i,k}\otimes V_{k,j},&&
    \epsilon(V_{i,j})=\delta_{i,j},\\
    \Delta(\mathscr{D}^{-1})=\mathscr{D}^{-1}\otimes \mathscr{D}^{-1},&&
    \epsilon(\mathscr{D}^{-1})=1.
\end{align*}
\end{theorem}
\proof
 As we can see $Det_q\cdot \mathscr{D}^{-1}=\mathscr{D}^{-1}\cdot Det_q$ in $\widehat{\mathcal{A}}_{q,\Theta}(n)$, Consider $\mathbb{C}[U_{q,\Theta}(n)]=\widehat{\mathcal{A}}_{q,\Theta}(n)/\langle Det_q\cdot\mathscr{D}^{-1}-1\rangle$. Let $I(R_1)$ is the ideal in $\widehat{\mathcal{A}}_{q,\Theta}(n)$ generated by $Det_q\cdot\mathscr{D}^{-1}-1$. Also $\Delta(Det_q\cdot\mathscr{D}^{-1}-1)=(Det_q\cdot\mathscr{D}^{-1}-1)\otimes Det_q\cdot\mathscr{D}^{-1}+1\otimes (Det_q\cdot\mathscr{D}^{-1}-1)$ and $\epsilon(Det_q\cdot\mathscr{D}^{-1}-1)=0$
So $I(R_1)$ is the biideal of $\widehat{\mathcal{A}}_{q,\Theta}(n)$.

In short, the algebra $\mathbb{C}[U_{q,\Theta}(n)]$ is defined as the extension of $\mathcal{A}_{q,\Theta}(n)$ with the element $\mathscr{D}^{-1}$ subject to the relation $Det_q\cdot \mathscr{D}^{-1}=\mathscr{D}^{-1}\cdot Det_q=1$.
The comultiplication and counit extend uniquely to $\mathbb{C}[U_{q,\Theta}(n)]$ if we put $\Delta(\mathscr{D}^{-1})=\mathscr{D}^{-1}\otimes \mathscr{D}^{-1}$ and $\epsilon(\mathscr{D}^{-1})=1$.
\qed

\begin{theorem}
    There are unique Hopf algebra structures on the algebra $\mathbb{C}[U_{q,\Theta}(n)]$. The antipode $S$ of this Hopf algebra is given by $S(V_{i,k})=\Big[\frac{\kappa_k}{\kappa_i}\Big](-q)^{i-k}\mathscr{D}^{-1}M_{k,i}$ for all $i,k$ and $S(\mathscr{D}^{-1})=Det_q$ where $\kappa_k=\theta_{k+1,k}\theta_{k+2,k}\cdots \theta_{n,k}$. Moreover we have $S^2(V_{i,j})=q^{2(i-j)}V_{i,j}$.
\end{theorem}
\begin{proof}
	 Let $\kappa_k=\theta_{k+1,k}\theta_{k+2,k}\cdots \theta_{n,k}$,  $\wp_k=\theta_{1,k}\theta_{2,k}\cdots \theta_{k-1,k}$ and $w_k=w_k=\wp_k\kappa_k$ for all $k=1,2,...,n$.
	
    From the relations between $M_{i,j}$ and $Det_q$ given from \ref{q-minor-1} to  \ref{q-minor-4} we have
    \begin{align}
        (-q)^{i-j}\Big[\frac{\kappa_i}{\kappa_j}\big]\sum_k\Big[\frac{\kappa_k}{\kappa_i}\big](-q)^{k-i}M_{i,k}(-q)^{2(j-k)}V_{j,k}&=\delta_{i,j}Det_{q}\label{antipode-1}\\
        (-q)^{i-j}\frac{\kappa_i}{\kappa_j}\sum_k \Big[\frac{w_j}{w_k}\Big]V_{j,k}Det_q\cdot (-q)^{k-i}\Big[\frac{\kappa_k}{\kappa_i}\big]\mathscr{D}^{-1}M_{i,k}&=\delta_{i,j}Det_{q}\label{antipode-2}\\
        (-q)^{j-i}\Big[\frac{\kappa_j}{\kappa_i}\big]
        \sum_k (-q)^{i-k}\Big[\frac{\kappa_i}{\kappa_k}\big]M_{k,i}V_{k,j}&=\delta_{i,j}Det_{q}\label{antipode-3}\\(-q)^{j-i}\frac{\kappa_j}{\kappa_i}\sum_k\Big[\frac{w_k}{w_j}\Big](-q)^{2(k-j)}V_{k,j} Det_q\cdot(-q)^{i-k}\Big[\frac{\kappa_i}{\kappa_k}\big]\mathscr{D}^{-1}M_{k,i}&=\delta_{i,j}Det_{q}\label{antipode-4}
    \end{align}

    Consider the matrix $W$ whose $(i,j)$-th entry is $W_{i,j}=(-q)^{i-j}\frac{\kappa_i}{\kappa_j}\mathscr{D}^{-1}M_{j,i}$ and the matrix $V$ whose $(i,j)$-th entry is $V_{i,j}$.

Then the above relations~\ref{antipode-1}-\ref{antipode-4} can be written as
\begin{align}
 \sum_kW_{k,i}\cdot (-q)^{2(j-k)}V_{j,k}&=\delta_{i,j}\label{antipode-5}\\\sum_k V_{j,k}W_{k,i}&=\delta_{i,j}\label{antipode-6}\\
 \sum_k W_{i,k}V_{k,j}&=\delta_{i,j}\label{antipode-7}\\\sum_k(-q)^{2(k-j)}V_{k,j}\cdot  W_{i,k}&=\delta_{i,j}\label{antipode-8}
  \end{align}

    Therefore we have $WV=VW=I$ and $W^{tr}DV^{tr}D^{-1}=I=DV^{tr}D^{-1}W^{tr}$ where $D=diag((-q)^{-2},(-q)^{-4},\cdots,(-q)^{-2n})$ is the scalar matrix. 

    For the matrix $U$, Let $U_1=U\otimes I$ and $U_2=I\otimes U$. Then the above four generating relations give us $RV_1V_2=V_2V_1R$. This relations give us $W_1W_2R=RW_2W_1$. 

Therefore there exist an anti-multiplicative linear map $S:\mathcal{A}_{q,\Theta}(n)\longrightarrow \mathbb{C}[U_{q,\Theta}(n)]$  such that $S(1)=1$. So $S(V_{i,k})=W_{i,k}=(-q)^{i-k}\frac{\kappa_i}{\kappa_k}\mathscr{D}^{-1}M_{k,i}=\Big[\frac{\theta_{n,i}\theta_{n-1,i}\cdots \theta_{i+1,i}}{\theta_{n,k}\theta_{n-1,k}\cdots \theta_{k+1,k}}\Big](-q)^{i-k}\mathscr{D}^{-1}M_{k,i}$.

From equation~\ref{antipode-6} and \ref{antipode-7}, we have $\sum_k~V_{j,k}S(V_{k,i})=\delta_{i,j}$ and $\sum_k~S(V_{i,k})V_{k,j}=\delta_{i,j}$. So it satisfy the condition that $m(S\otimes Id)\Delta(a)=\epsilon(a)1=m(Id\otimes S)\Delta(a)$ for all $a\in \mathcal{A}_{q,\Theta}(n) $. As $Det_q\in \mathcal{A}_{q,\Theta}(n)$ and $\Delta(Det_q)=Det_q\otimes Det_q$, we have $S(Det_q)Det_q=Det_qS(Det_q)=\epsilon(Det_q)=1$. So $S(Det_q)=\mathscr{D}^{-1}$. Now if we want to extend $S$ from $\mathcal{A}_{q,\Theta}(n)$ to $\mathbb{C}[U_{q,\Theta}(n)]$, it Should satisfy $\mathscr{D}^{-1}S(\mathscr{D}^{-1})=S(\mathscr{D}^{-1}Det_q)=S(1)=1$ and similarly $S(\mathscr{D}^{-1})\mathscr{D}^{-1}=1$. Hence we have the only choice $S(\mathscr{D}^{-1})=Det_q$. 
    Also we have \begin{align*}
        S(V_{i,k}\cdot\mathscr{D}^{-1})&=Det_q \cdot S(V_{i,k})\\
        &=Det_q\frac{\kappa_i}{\kappa_k}(-q)^{i-k}\mathscr{D}^{-1}M_{k,i}\\
        &=\frac{w_1w_2\cdots w_{k-1}w_{k+1}\cdots w_n}{w_{\sigma(1)}w_{\sigma(2)}\cdots w_{\sigma(k-1)}w_{\sigma(k+1)}\cdots w_{\sigma(n)}}\frac{\kappa_i}{\kappa_k}(-q)^{i-k}\mathscr{D}^{-1}M_{k,i}\\
        &=\frac{w_1w_2\cdots w_{k-1}w_{k+1}\cdots w_n}{w_1w_2\cdots w_{i-1}w_{i+1}\cdots w_n}\frac{\kappa_i}{\kappa_k}(-q)^{i-k}\mathscr{D}^{-1}M_{k,i} \\
        &=\frac{w_i}{w_k}S(V_{i,k})\cdot Det_q\\
        &=\frac{w_i}{w_k}S(\mathscr{D}^{-1}\cdot V_{i,k})
    \end{align*}
\end{proof}
So, $S$ extends to an anti-homomorphism by setting $S(\mathscr{D}^{-1})=Det_q$. Since the relations $m(S\otimes Id)\Delta(a)=\epsilon(a)1=m(Id\otimes S)\Delta(a)$ hold for the generators of $\mathbb{C}[U_{q,\Theta}(n)]$, the mapping $S$ is an antipode and $\mathbb{C}[U_{q,\Theta}(n)]$ is a Hopf algebra.

Therefore we have $	\Delta(S(V_{i,j}))=(S\otimes S)\circ\text{flip}\circ\Delta(V_{i,j})=\sum_k S(V_{k,j})\otimes S(V_{i,k})$ and $\epsilon(S(V_{i,j}))=\epsilon(V_{i,j})=\delta_{i,j}$. Therefore we have $\sum_k S(W_{k,j})W_{i,k}=\delta_{i,j}=\sum_k W_{k,i}S(W_{j,k})$

From equation~\ref{antipode-5} and \ref{antipode-8}, we have $S^2(V_{j,k})=S(W_{j,k})=(-q)^{2(j-k)}V_{j,k}$. Also $S^2(Det_q)=Det_q$ and $S^2(\mathscr{D}^{-1})=\mathscr{D}^{-1}$ \qed

Let $\omega:\mathcal{A}_{q,\Theta}(n)\longrightarrow \mathcal{A}_{q,\Theta}(n)$ be conjugate-linear multiplicative map such that $\omega(V_{i,j})=V_{j,i}$. Then we have $\omega([D^I_J])=[\Bar{D}^J_I]$. As a consequence we have, $\omega(M_{i,k})=M_{k,i}$ and $\omega(Det_q)=Det_q$. From the relation, $Det_q\cdot \mathscr{D}^{-1}=\mathscr{D}^{-1}\cdot Det_q=1$, we can extend $\omega$ by defining $\omega(\mathscr{D}^{-1})=\mathscr{D}^{-1}$. Therefore we know $\omega(V_{i,j})=V_{j,i}$, $\omega(Det_q)=Det_q$, $\omega(M_{i,k})=M_{k,i}$ and $\omega(\mathscr{D}^{-1})=\mathscr{D}^{-1}$.

We define the $*$ structure on $\mathbb{C}[U_{q,\Theta}(n)]$, such that $*=S\circ\omega$. Therfore we have $V_{i,k}^*=S(V_{k,i})=W_{k,i}=\Big[\frac{\theta_{n,k}\theta_{n-1,k}\cdots \theta_{k+1,k}}{\theta_{n,i}\theta_{n-1,i}\cdots \theta_{i+1,i}}\big](-q)^{k-i}\mathscr{D}^{-1}M_{i,k}$. Now $*$ is conjugate-linear anti-multiplicative map. 

Now we can see that
\begin{align*}
    S\circ \omega \circ S\circ \omega(V_{i,j})&=S\circ \omega \circ S(V_{j,i})\\ &=S\circ\omega(\Big[\frac{\theta_{n,i}\theta_{n-1,i}\cdots \theta_{i+1,i}}{\theta_{n,j}\theta_{n-1,j}\cdots \theta_{j+1,j}}\Big](-q)^{j-i}\mathscr{D}^{-1}M_{i,j})
    \\&=S(\Big[\frac{\theta_{n,j}\theta_{n-1,j}\cdots \theta_{j+1,j}}{\theta_{n,i}\theta_{n-1,i}\cdots \theta_{i+1,i}}\Big](-q)^{j-i}~\mathscr{D}^{-1}M_{j,i})
    \\&=(-q)^{2(j-i)}S(S(V_{i,j}))=V_{i,j}
\end{align*}
Therefore we have $\ast\circ\ast=Id$ and $S\circ\ast\circ S\circ\ast=Id$ on generators.

\section{Basis of  $\mathbb{C}[U_{q,\Theta}(n)]$}\label{sec:monomial-basis}
In this section, we present bases for $\mathcal{A}_{q,\Theta}(n)$ and $\mathbb{C}[U_{q,\Theta}(n)]$. We first define a reduction system and certain sets. Using arguments similar to those in \cite{koelink91}, one can derive bases for these algebras via the reduction system.

{\bf Ordering on monomials :}
Consider the ordering on generators $X=\{V_{i,j}~:~1\leq i,j\leq n\}$ defined by $V_{i,j}\leq_0 V_{k,l}$ if and only if $i+j< k+l$ or $(i<k)\&(i+j=k+l)$. We extend the ordering $\leq_0$ to monomials od $\mathbb{C}\langle X\rangle $ by first ordering by the degree and for monomials with same degree by the lexicographic ordering. 

{\bf Reduction system :} 
In order to write the reduction system in a simple manner we introduce the operator $L$ defined by 
\begin{align}
	\left\{\begin{matrix}
		L(V_{i,l}V_{i,j})&=\frac{1}{q\theta_{l,j}}V_{i,j}V_{i,l}&&\text{ if }&j<l~~\forall i,\\
		L(V_{k,j}V_{i,j})&=\frac{1}{q\theta_{i,k}}V_{i,j}V_{k,j}&&\text{ if }&i<k~~\forall j,\\
		L(V_{i,l}V_{k,j})&=\frac{\theta_{i,k}}{\theta_{l,j}}V_{k,j}V_{i,l}&&\text{ if }&j<l,~i<k,~\text{ and }i+l>k+j,\\
		L(V_{k,j}V_{i,l})&=\frac{\theta_{l,j}}{\theta_{i,k}}V_{i,l}V_{k,j}&&\text{ if }&j<l,~i<k,~\text{ and }i+l\leq k+j,\\
		L(V_{k,l}V_{i,j})&=\frac{\theta_{j,l}}{\theta_{i,k}}V_{i,j}V_{k,l}&+\frac{1}{\theta_{i,k}}(q-\frac{1}{q})V_{i,l}V_{k,j}&\text{ if }&j<l,~i<k,~\text{ and }i+l\leq k+j,\\
		L(V_{k,l}V_{i,j})&=\frac{\theta_{j,l}}{\theta_{i,k}}V_{i,j}V_{k,l}&+\frac{1}{\theta_{l,j}}(q-\frac{1}{q})V_{k,j}V_{i,l}&\text{ if }&j<l,~i<k,~\text{ and }i+l> k+j,\\
		L(V_{i,j}V_{k,l})&=~~~~V_{i,j}V_{k,l}&&&\text{ otherwise. }
	\end{matrix}\right.
\end{align}
The reduction system is now given by $\{V_{i,j}V_{k,l},~L(V_{i,j}V_{k,l})\}$ for $V_{k,l}\leq_0 V_{i,j}$. The order $\leq_0$ is compatible with the reduction system. Now every element of $\mathbb{C}\langle X\rangle $ is reduction finite, since the number of monomials smaller than a given monomial is finite.
Therefore we have the following theorem by diamond lemma:
\begin{theorem}
	For the ordering $\leq_0$ on the elements $V_{i,j}$ there is a basis for $\mathcal{A}_{q,\Theta}(n)$ consisting of \[ \{V_{i_1,j_1}^{r_1}\cdots V_{i_m,j_m}^{r_m} ~:~m=n^2,~r_i\geq 0,~V_{i_1,j_1}\leq_0 V_{i_2,j_2}\leq_0\cdots \leq_0 V_{i_m,j_m}\}\]\label{basis-theorem}
\end{theorem}

Consider the folowing sets
\begin{align}
	B_s&=\{V_{1,s}^{m_1}V_{2,s-1}^{m_2}\cdots V_{s,1}^{m_s}:m_1,m_2,\cdots m_s\geq 0\}&\text{ for }&1\leq s\leq n,\\
	B_s&=\{V_{s+1-n,n}^{m_1}V_{s-n+2,n-1}^{m_2}\cdots V_{n,s+1-n}^{m_{2n+1-s}}:m_1,m_2,\cdots m_{2n+1-s}\geq 0\}&\text{ for }&n+1\leq s\leq 2n-1,\\
	B_D&=\{(Det_q)^k:k\geq 0\},\\
	B&=\{b_1b_2\cdots b_{2n-1}~:~b_s\in B_s \text{ for } 1\leq s\leq 2n-1\},
	\end{align}
	\begin{align}
		\hat{B}_n&=\{V_{1,n}^{m_1}V_{2,n-1}^{m_2}\cdots V_{n,1}^{m_n}:m_1,m_2,\cdots m_s\geq 0~\text{ and } \underset{1\leq i\leq n}{~min~}~m_i=0\},
		\\
		\hat{B}&=\{b_1b_2\cdots b_{n-1}\hat{b}_nb_{n+1}\cdots b_{2n-1}(Det_q)^k~:~b_s\in B_s \text{ for } s\neq n,~k\geq 0\text{ and }\hat{b}_n\in \hat{B}_n\}\text{ and }
		\label{basis-set-1}\\
		\tilde{B}&=\{b_1b_2\cdots b_{n-1}\hat{b_n}b_{n+1}\cdots b_{2n-1}(Det_q)^k~:~b_s\in B_s \text{ for } s\neq n,~k\in\mathbb{Z}\text{ and }\hat{b}_n\in \hat{B}_n\}\label{basis-set-2}.
	\end{align}

Let $\rho_s=span~ B_s$ for $1\leq s\leq 2n-1$, $\hat{\rho}_n=span~\hat{B}_n$ and $\rho_D=span~ B_D$.
So $\mathcal{A}_{q,\Theta}(n)=\rho_1\otimes_{\mathbb{C}}\rho_2 \otimes_{\mathbb{C}}\cdots \otimes_{\mathbb{C}}\rho_{2n-1} $ and $B$ is the basis of $\mathcal{A}_{q,\Theta}(n)$ mentioned in theorem-\ref{basis-theorem}.

 Let $\sigma_1(i)=n+1-i$ for all $i$. As we know \begin{align*}Det_q&=\sum_{\sigma} \Big(\underset{\begin{matrix}x<y\\\sigma(x)>\sigma(y)\end{matrix}}{\prod}(-q~\theta_{{\sigma(y)},{\sigma(x)}})\Big)V_{{\sigma(1)},1}V_{{\sigma(2)},2}\cdots V_{{\sigma(n)},n}\\
 &=\sum_{\sigma\neq \sigma_1} \Big(\underset{\begin{matrix}x<y\\\sigma(x)>\sigma(y)\end{matrix}}{\prod}(-q~\theta_{{\sigma(y)},{\sigma(x)}})\Big)V_{{\sigma(1)},1}V_{{\sigma(2)},2}\cdots V_{{\sigma(n)},n}\\&+
 (-q)^\frac{n(n-1)}{2}\Big(\prod_{r<s}\theta_{r,s}\Big) V_{n,1}V_{n-2,2}\cdots V_{1,n}
 \end{align*}
 Therefore we can express 
  $V_{n,1}V_{n-2,2}\cdots V_{1,n}$  as the following
  \begin{align}(-\frac{1}{q})^\frac{n(n-1)}{2}\Big(\prod_{r>s}\theta_{r,s}\Big)[Det_q ~-\sum_{\sigma\neq \sigma_1} \Big(\underset{\begin{matrix}x<y\\\sigma(x)>\sigma(y)\end{matrix}}{\prod}(-q~\theta_{{\sigma(y)},{\sigma(x)}})\Big)V_{{\sigma(1)},1}V_{{\sigma(2)},2}\cdots V_{{\sigma(n)},n}]\label{exp-for-h}\end{align}
  
  Now for any element $h=V_{1,n}^{m_1}V_{2,n-1}^{m_2}\cdots V_{n,1}^{m_n}\in B_n$, we define $\eta(h)=\underset{1\leq i\leq n}{~min~}~m_i$. If $\eta(h)\geq 1$, then we put $h'=V_{1,n}^{m_1-1}V_{2,n-1}^{m_2-1}\cdots V_{n,1}^{m_n-1}$. Note that we can reorder the summand in the expression \ref{exp-for-h} such that all factors from $\rho_i$ precede all factors from $\rho_j$ for $i<j$. Then we can replace $\rho_n$ by tensor product of $\hat{\rho}_n$ and $\rho_D$.
  Therefore we have  
  \begin{align*}
  	\mathcal{A}_{q,\Theta}(n)&=\rho_1\otimes_{\mathbb{C}}\rho_2 \otimes_{\mathbb{C}}\cdots\otimes_{\mathbb{C}}\rho_{n-1} \otimes_{\mathbb{C}}  {\rho_n}\otimes_{\mathbb{C}} \rho_{n+1}\otimes_{\mathbb{C}} \cdots  \otimes_{\mathbb{C}}\rho_{2n-1}\\
  	&=\rho_1\otimes_{\mathbb{C}}\rho_2 \otimes_{\mathbb{C}}\cdots\otimes_{\mathbb{C}}\rho_{n-1} \otimes_{\mathbb{C}} \hat {\rho_n}\otimes_{\mathbb{C}} \rho_{n+1}\otimes_{\mathbb{C}} \cdots  \otimes_{\mathbb{C}}\rho_{2n-1}\otimes_{\mathbb{C}} \rho_D
  \end{align*} and $\hat{B}$ is the basis of $\mathcal{A}_{q,\Theta}(n)$.
\begin{theorem}[{Basis}] The sets $\hat{B}$ and $\tilde{B}$ form bases of $\mathcal{A}_{q,\Theta}(n)$ and $\mathbb{C}[U_{q,\Theta}(n)]$, respectively, where they are defined in \eqref{basis-set-1}--\eqref{basis-set-2}.
\end{theorem}

For $n=3$, we have the following remark.
\begin{remark}\label{monomial-marker}
Consider the following 3 sets
\begin{align*}
	\tilde{B}_1&=\{V_{1,1}^iV_{1,2}^jV_{2,1}^kV_{1,3}^lV_{3,1}^nV_{2,3}^pV_{3,2}^rV_{3,3}^s(Det_q)^u~~:~i,j,k,l,m,n,p,r,s\geq 0;u\in\mathbb{Z}\},\\
	\tilde{B}_2&=\{V_{1,1}^iV_{1,2}^jV_{2,1}^kV_{2,2}^mV_{3,1}^nV_{2,3}^pV_{3,2}^rV_{3,3}^s(Det_q)^u~~:~i,j,k,l,m,n,p,r,s\geq 0;u\in\mathbb{Z}\}~\text{ and }\\
	\tilde{B}_3&=\{V_{1,1}^iV_{1,2}^jV_{2,1}^kV_{2,2}^mV_{1,3}^lV_{2,3}^pV_{3,2}^rV_{3,3}^s(Det_q)^u~~:~i,j,k,l,m,n,p,r,s\geq 0;u\in\mathbb{Z}\}.
\end{align*}
Therefore from the above theorem we see that $\tilde{B}_1\sqcup \tilde{B}_2\sqcup \tilde{B}_3 $ is basis of $\mathbb{C}[U_{q,\Theta}(n)]$. So for any element $x\in \mathbb{C}[U_{q,\Theta}(3)]$, it can be expressed as
\begin{align*}
\sum_\alpha c_\alpha~	&V_{1,1}^{i_\alpha} V_{1,2}^{j_\alpha}V_{2,1}^{k_\alpha}V_{2,2}^{m_\alpha}V_{3,1}^{n_\alpha}V_{2,3}^{p_\alpha }V_{3,2}^{r_\alpha }V_{3,3}^{s_\alpha }(Det_q)^{u_\alpha }\\
+\sum_\beta d_\beta~ &V_{1,1}^{i_\beta} V_{1,2}^{j_\beta}V_{2,1}^{k_\beta}V_{1,3}^{l_\beta}V_{3,1}^{n_\beta}V_{2,3}^{p_\beta }V_{3,2}^{r_\beta }V_{3,3}^{s_\beta }(Det_q)^{u_\beta }\\
+\sum_\kappa g_\kappa~ &V_{1,1}^{i_\kappa} V_{1,2}^{j_\kappa}V_{2,1}^{k_\kappa}V_{1,3}^{l_\kappa}V_{2,2}^{m_\kappa}V_{2,3}^{p_\kappa }V_{3,2}^{r_\kappa }V_{3,3}^{s_\kappa }(Det_q)^{u_\kappa }
\end{align*}
Therefore for $\mathbb{C}[U_{q,\Theta}(3)]$, we say $\alpha$-monomial to mean elements from $\tilde{B}_1$, $\beta$-monomial to mean elements from $\tilde{B}_2$ and $\kappa$-monomial to mean elements from $\tilde{B}_3$.
\end{remark}

\section{ The $C^*$- algebra $C(U_{q,\Theta}(n))$}\label{sec:C-ast-algebra}
\begin{theorem}
    There exist a universal $C^*$-algebra $C(U_{q,\Theta}(n))$ generated by $\{\mathscr{D}^{-1}\}\cup\{V_{i,j}~:~1\leq i,j\leq n\}$ satisfying the following relations
    \begin{align}
V_{i,j}V_{i,l}&=q\theta_{l,j}V_{i,l}V_{i,j}&&\text{ if }j<l&\forall i\\
V_{i,j}V_{k,j}&=q\theta_{i,k}V_{k,j}V_{i,j}&&\text{ if }i<k&\forall j\\
\theta_{l,j}V_{i,l}V_{k,j}&=\theta_{i,k}V_{k,j}V_{i,l}&&\text{ if }j<l,~i<k&\\
\theta_{j,l}V_{i,j}V_{k,l}-\theta_{i,k}V_{k,l}V_{i,j}&=(q-\frac{1}{q})V_{i,l}V_{k,j}&&\text{ if }j<l,~i<k&\\1&=
Det_q\mathscr{D}^{-1}
\\V_{i,k}\mathscr{D}^{-1}&=\Big[\frac{\theta_{k,1}\theta_{k,2}\cdots \theta_{k,k-1}\theta_{k,k+1}\cdots \theta_{k,n}}{\theta_{i,1}\theta_{i,2}\cdots \theta_{i,i-1}\theta_{i,i+1}\cdots \theta_{i,n}}\big]\mathscr{D}^{-1}V_{i,k}\\
V_{k,i}^*&=\Big[\frac{\theta_{k,n}\theta_{k,n-1}\cdots \theta_{k,k+1}}{\theta_{i,n}\theta_{i,n-1}\cdots \theta_{i,i+1}}\big](-q)^{i-k}\mathscr{D}^{-1}M_{k,i}
\end{align}
where $Det_q=\sum_{\sigma} \Big(\underset{\begin{matrix}x<y\\\sigma(x)>\sigma(y)\end{matrix}}{\prod}(-q~\theta_{{\sigma(y)},{\sigma(x)}})\Big)V_{{\sigma(1)},1}V_{{\sigma(2)},2}\cdots V_{{\sigma(n)},n}$ and\newline $M_{k,i}=\sum_{\sigma~:~\sigma(i)=k} \Big(\underset{\begin{matrix}x<y\\\sigma(x)>\sigma(y)\end{matrix}}{\prod}(-q~\theta_{{\sigma(y)},{\sigma(x)}})\Big)V_{{\sigma(1)},1}V_{{\sigma(2)},2}\cdots V_{{\sigma(i-1)},i-1} V_{{\sigma(i+1)},i+1}\cdots  V_{{\sigma(n)},n}$
\end{theorem}
\begin{proof}

We have the following relations
$\sum_k V_{i,k}V_{j,k}^*=\delta_{i,j}$ and $\sum_kV_{k,i}^*V_{k,j}=\delta_{i,j}$ and $\mathscr{D}^{-1}$ is unitary. Therefore we have $||V_{i,k}||\leq 1$. So polynomials are norm bounded. Now $V_{i,k}=\delta_{i,k}$ gives us one admissible representations.
\end{proof}

Therefore we have $*$-algebra homomorphism $\Finv:\mathbb{C}[U_{q,\Theta}(n)]\longrightarrow C(U_{q,\Theta}(n))$. So one can ask whether $\Finv$ is injective or not. 

 The case $n=2$ was established by Zhang and Zhao. We address the $n=3$ case in this article, whereas the case $n > 3$ remains open.

We therefore propose the following conjecture:

\begin{conjecture}
There exists a representation $\Finv:\mathbb{C}[U_{q,\Theta}(n)]\longrightarrow C(U_{q,\Theta}(n))$ such that the set $\{\Finv(b) : b \in \tilde{B}\}$ is linearly independent.
\end{conjecture}

\begin{theorem}
	There are $C^*$-algebra homomorphisms $\Delta:C(U_{q,\Theta}(n))\rightarrow C(U_{q,\Theta}(n))\otimes C(U_{q,\Theta}(n))$ and $\epsilon:C(U_{q,\Theta}(n))\rightarrow \mathbb{C}$ and anti algebra homomorphism $S:\mathbb{C}[U_{q,\Theta}(n)]\rightarrow \mathbb{C}[U_{q,\Theta}(n)]$ such that
	\begin{align}
		  \Delta(V_{i,j})=\sum_k V_{i,k}\otimes V_{k,j}&&\Delta(\mathscr{D}^{-1})=\mathscr{D}^{-1}\otimes \mathscr{D}^{-1}\\
		  \epsilon(V_{i,j})=\delta_{i,j}&&\epsilon(\mathscr{D}^{-1})=1\\
		  S(V_{i,j})=\delta_{i,j}&&S(\mathscr{D}^{-1})=Det_q\\
		  S\circ * \circ S\circ *=Id\\
		  \sum_k S(V_{i,k}) V_{k,j}=\delta_{i,j}&&\sum_k V_{i,k} S(V_{k,j})=\delta_{i,j}
	\end{align}
\end{theorem}
\qed

To end this section  we will state some commutation relations in $\mathbb{C}[U_{q,\Theta}(n)]$ concerning $V_{i,j}$ and $V_{i,j}^*$. So we have $[\sum_{k,l} R_{ik,ls} V^*_{kj}V_{lr}]= [\sum_{k,l} R_{kj,rl}V_{i,k}V_{s,l}^*]$ for all $i,j,r,s$ (See appendix \ref{sec:star}).

\begin{align*}
    \sum_{k,l} R_{ik,ls} V^*_{kj}V_{lr}&=\left\{\begin{matrix}
       R_{ii,ss} V^*_{ij}V_{sr}+R_{is,is} V^*_{sj}V_{ir}&\text{ if }i\neq s\\R_{ii,ii} V^*_{ij}V_{ir}+\sum_{k=1;k\neq i}^n R_{ik,ki} V^*_{kj}V_{kr} &\text{ if }i=s
    \end{matrix}\right.\\~\\
    &=\left\{\begin{matrix}
       \theta_{s,i} V^*_{sj}V_{ir}&\text{ if }i\neq s\\q V^*_{ij}V_{ir}+\sum_{k=1}^{i-1} (q-\frac{1}{q}) V^*_{kj}V_{kr} &\text{ if }i=s \\ 
       \frac{1}{q} V^*_{ij}V_{ir}+(q-\frac{1}{q})\delta_{j,r}-\sum_{k=i+1}^{n} (q-\frac{1}{q}) V^*_{kj}V_{kr} &\text{ if }i=s
    \end{matrix}\right.
\end{align*}

\begin{align*}
    \sum_{k,l} R_{kj,rl}V_{i,k}V_{s,l}^*&=\left\{\begin{matrix}
       R_{jj,rr} V_{i,j}V_{s,r}^*+R_{rj,rj} V_{i,r}V_{s,j}^*&\text{ if }j\neq r\\R_{jj,jj} V_{i,j}V_{s,j}^*+\sum_{k=1,k\neq j}^n R_{kj,jk}V_{i,k}V_{s,k}^* &\text{ if }j=r
    \end{matrix}\right.\\~\\
    &=\left\{\begin{matrix}
       \theta_{j,r} V_{i,r}V_{s,j}^*&\text{ if }j\neq r\\
       q V_{i,j}V_{s,j}^*+\sum_{k=j+1}^n (q-\frac{1}{q})V_{i,k}V_{s,k}^* &\text{ if }j=r \\ 
       \frac{1}{q} V_{i,j}V_{s,j}^*+(q-\frac{1}{q})\delta_{i,s}-\sum_{k=1}^{j-1} (q-\frac{1}{q})V_{i,k}V_{s,k}^* &\text{ if }j=r 
    \end{matrix}\right.
\end{align*}
Therefore we have four relations. Using $ \sum_{k=1}^{n} V^*_{kj}V_{kr}=\delta_{j,r}$ and $\sum_{k=1}^n V_{i,k}V_{s,k}^*=\delta_{i,s}$, first three relations among those four are given below:
\begin{align}
    \theta_{s,i} V^*_{sj}V_{ir}&=\theta_{j,r} V_{i,r}V_{s,j}^* &\text{ if }i\neq s,~j\neq r\\
   \theta_{j,r} V_{i,r}V_{i,j}^*&= \Big[q V^*_{ij}V_{ir}+\sum_{k=1}^{i-1} (q-\frac{1}{q}) V^*_{kj}V_{kr}\Big]~=\Big[\frac{1}{q}V^*_{ij}V_{ir}-\sum_{k=i+1}^{n} (q-\frac{1}{q}) V^*_{kj}V_{kr}\Big]&\text{ if }j\neq r,\\
    \theta_{s,i} V^*_{sj}V_{ij}&= \Big[q V_{i,j}V_{s,j}^*+\sum_{k=j+1}^n (q-\frac{1}{q})V_{i,k}V_{s,k}^* \Big]~=\Big[\frac{1}{q} V_{i,j}V_{s,j}^*-\sum_{k=1}^{j-1} (q-\frac{1}{q})V_{i,k}V_{s,k}^*\Big] &\text{ if }i\neq s,
\end{align}
and using $ \sum_{k=1}^{n} V^*_{kj}V_{kr}=\delta_{j,r}$ and $\sum_{k=1}^n V_{i,k}V_{s,k}^*=\delta_{i,s}$ fourth relation can be written as follows:
\begin{align}
    q V^*_{ij}V_{ij}+\sum_{k=1}^{i-1} (q-\frac{1}{q}) V^*_{kj}V_{kj}&=q V_{i,j}V_{i,j}^*+\sum_{k=j+1}^n (q-\frac{1}{q})V_{i,k}V_{i,k}^*\\
    \frac{1}{q} V^*_{ij}V_{ij}+(q-\frac{1}{q})-\sum_{k=i+1}^{n} (q-\frac{1}{q}) V^*_{kj}V_{kj}&=q V_{i,j}V_{i,j}^*+\sum_{k=j+1}^n (q-\frac{1}{q})V_{i,k}V_{i,k}^*\\
    q V^*_{ij}V_{ij}+\sum_{k=1}^{i-1} (q-\frac{1}{q}) V^*_{kj}V_{kj}&=\frac{1}{q} V_{i,j}V_{i,j}^*+(q-\frac{1}{q})-\sum_{k=1}^{j-1} (q-\frac{1}{q})V_{i,k}V_{i,k}^*\\
     V^*_{ij}V_{ij}-\sum_{k=i+1}^n (q^2-1) V^*_{kj}V_{kj}&= V_{i,j}V_{i,j}^*-\sum_{k=1}^{j-1} (q^2-1)V_{i,k}V_{i,k}^*
\end{align}

\section{Construction of the quantum unitary group $C(U_{q,\Theta}(3))$}
We follow woronowicz method of construction in \cite{wor87} to construct the quantum unitary group $U_{q,\Theta}(3)$. we denote $a=V_{1,1}$, $b=V_{1,2}$, $c=V_{1,3}$, $d=V_{2,1}$, $e=V_{2,2}$, $f=V_{2,3}$, $g=V_{3,1}$, $h=V_{3,2}$, $k=V_{3,3}$ and $D =Det_q$. Also let $\gamma=\theta_{2,1}$, $\lambda=\theta_{3,1}$ and $\varsigma=\theta_{3,2}$ and $\overline{\gamma}=\theta_{1,2}$, $\overline{\lambda}=\theta_{1,3}$ and $\overline{\varsigma}=\theta_{2,3}$.

\subsection{All Relations among generators of the $\ast$-algebra}\label{subsec:3-relations}

From equation \eqref{eqn: hopf* rel} and the unitary condition in \cite{guo02},  we get  the generators $a,b,c,d,e,f,g,h,k,D$ of $*$-algebra $\mathbb{C}[U_{q,\Theta}(n)]$ satisfying the following relations 

\begin{enumerate}
	\item All row, coloumn and cross diagonal wise twisted commutation relations among generators are given below:
\begin{align}\label{eqn:rowcol rel}
    \begin{matrix}
a b &= q\,\gamma\, b a,\\
a c &= q\,\lambda\, c a,\\
a d &= q\,\overline{\gamma}\, d a,\\
a g &= q\,\overline{\lambda}\, g a, 
    \end{matrix}
    & \qquad\qquad\quad
    \begin{matrix}
b c &= q\,\varsigma\, c b, \\
\gamma\, b d &= \overline{\gamma}\, d b,\\
b e &= q\,\overline{\gamma}\, e b, \\
\gamma\, b g &= \overline{\lambda}\, g b,\\
b h &= q\,\overline{\lambda}\, h b, 
    \end{matrix}
    &
    \begin{matrix}    
\lambda\, c d &= \overline{\gamma}\, d c, \\
\varsigma\, c e &= \overline{\gamma}\, e c, \\
c f &= q\,\overline{\gamma}\, f c, \\
\lambda\, c g &= \overline{\lambda}\, g c, \\
\varsigma\, c h &= \overline{\lambda}\, h c, \\
c k &= q\,\overline{\lambda}\, k c, 
    \end{matrix}\nonumber\\\,\\
    \begin{matrix}
    d e &= q\,\gamma\, e d, \\
d f &= q\,\lambda\, f d, \\
d g &= q\,\overline{\varsigma}\, g d,
    \end{matrix}
    & \qquad\qquad\quad
     \begin{matrix}
        e f &= q\,\varsigma\, f e,\\
\gamma\, e g &= \overline{\varsigma}\, g e, \\
        e h &= q\,\overline{\varsigma}\, h e,
    \end{matrix}
    &
    \begin{matrix} 
\lambda\, f g &= \overline{\varsigma}\, g f, \\
\varsigma\, f h &= \overline{\varsigma}\, h f, \\
f k &= q\,\overline{\varsigma}\, k f, 
    \end{matrix}.\nonumber\\\nonumber\\
    \begin{matrix}
g h &= q\,\gamma\, h g, \\
g k &= q\,\lambda\, k g,
    \end{matrix}
    & \qquad \qquad\quad
h k = q\,\varsigma\, k h.
    &\nonumber
\end{align}

\item  Quantum 2x2 minor relations are given below:
 \begin{align}\label{eqn: qminor rel}
 \begin{matrix}
\overline{\gamma}\, a e - \overline{\gamma}\, e a
&= (q - q^{-1})\, b d, \\
\overline{\lambda}\, a f - \overline{\gamma}\, f a
&= (q - q^{-1})\, c d, \\
\overline{\gamma}\, a h - \overline{\lambda}\, h a
&= (q - q^{-1})\, b g, \\\nonumber
\overline{\lambda}\, a k - \overline{\lambda}\, k a
&= (q - q^{-1})\, c g, 
 \end{matrix}
 && \begin{matrix}
\overline{\varsigma}\, b f - \overline{\gamma}\, f b
&= (q - q^{-1})\, c e, \\
\overline{\varsigma}\, b k - \overline{\lambda}\, k b
&= (q - q^{-1})\, c h, 
 \end{matrix}
\\\,\\\begin{matrix}
\overline{\gamma}\, d h - \overline{\varsigma}\, h d
&= (q - q^{-1})\, e g, \\\nonumber
\overline{\lambda}\, d k - \overline{\varsigma}\, k d
&= (q - q^{-1})\, f g,  \end{matrix}&&
\overline{\varsigma}\, e k - \overline{\varsigma}\, k e
= (q - q^{-1})\, f h.
\end{align}
\item Relations involving quantum determinant and generators are given below:
\begin{align}\label{eqn:qdet rel}
	\begin{matrix}
		aD&=Da,\\dD&=\frac{\gamma\overline{\varsigma}}{\overline{\gamma}\overline{\lambda}} Dd,\\gD&=\frac{\lambda\varsigma}{\overline{\gamma}\overline{\lambda}} Dg,
	\end{matrix}&&
	\begin{matrix}
	bD&=\frac{\overline{\gamma}\overline{\lambda}}{\gamma\overline{\varsigma}} Db,\\eD&=De,\\hD&=\frac{\lambda\varsigma}{\gamma\overline{\varsigma}} Dh,
	\end{matrix}&&
	\begin{matrix}
	cD&=\frac{\overline{\gamma}\overline{\lambda}}{\lambda\varsigma}Dc,\\fD&=\frac{\gamma\overline{\varsigma}}{\lambda\varsigma} Df,\\Dk&=kD.
	\end{matrix}
\end{align}
\item Relations due to unitary condition:
\begin{align}\label{unitary eqn}\nonumber
aa^*+bb^*+cc^*=1,\quad dd^*+ee^*+ff^*=1,\quad gg^*+hh^*+kk^*=1,\\\nonumber
a^*a+d^*d+g^*g=1,\quad b^*b+e^*e+h^*h=1,\quad c^*c+f^*f+k^*k=1.\\
ad^*+be^*+cf^*=0,\quad ag^*+bh^*+ck^*=0,\quad dg^*+eh^*+fk^*=0,\\ \nonumber
a^*b+d^*e+g^*h=0,\quad a^*c+d^*f+g^*k=0,\quad b^*c+e^*f+h^*k=0.\\\nonumber
\end{align}
\item Expression of $*$ of generators:
\begin{align}\label{eqn: star of generator}
	\begin{matrix*}[l]
		a^*&=D^*[ek-q\overline{\varsigma}~hf]\\
		&=D^*[ek-q\varsigma~fh]
	\end{matrix*}&&
	\begin{matrix*}[l]
		a^*&=D^*[ke-\frac{1}{q}\overline{\varsigma}~hf]\\   &=D^*[ke-\frac{1}{q}\varsigma~fh]
	\end{matrix*}\nonumber\\\nonumber\\\nonumber
	\begin{matrix*}[l]
	b^*&=(-q\overline{\lambda}\varsigma\overline{\gamma})~D^*[dk-q\overline{\varsigma}~gf]\\
	&=(-q\overline{\lambda}\varsigma\overline{\gamma})~D^*[dk-q\lambda~fg]
	\end{matrix*}&&
	\begin{matrix*}[l]
	b^*&=(-q\overline{\gamma})~D^*[kd-\frac{1}{q}\overline{\lambda}~gf]\\
	&=(-q\overline{\gamma})~D^*[kd-\frac{1}{q}\varsigma~fg]
	\end{matrix*}\\\nonumber\\\nonumber
	\begin{matrix*}[l]
		c^*&=(q^2\overline{\lambda}\overline{\gamma})~D^*[dh-q\overline{\varsigma}~ge]\\
		&=(q^2\overline{\lambda}\overline{\gamma})~D^*[dh-q\gamma~eg]
	\end{matrix*}&&
	\begin{matrix*}[l]
		c^*&=(q^2\overline{\lambda}\overline{\varsigma}) ~D^*[hd-\frac{1}{q}\overline{\gamma}~ge]\\
		&=(q^2\overline{\lambda}\overline{\varsigma})~D^*[hd-\frac{1}{q}\varsigma~eg]
	\end{matrix*}\\\,\nonumber\\\nonumber
	\begin{matrix}d^*&=(-\frac{1}{q}\lambda \overline{\varsigma}\gamma)~D^*[bk-q\overline{\lambda}~hc]\\
		&=(-\frac{1}{q}\lambda \overline{\varsigma}\gamma)~D^*[bk-q\varsigma~ch]
	\end{matrix}&&
	\begin{matrix}d^*	&=(-\frac{1}{q}\gamma)~D^*[kb-\frac{1}{q}\overline{\varsigma}~hc]\\
		&=(-\frac{1}{q}\gamma)~D^*[kb-\frac{1}{q}\lambda~ch]
	\end{matrix}\\\,\nonumber\\
	\begin{matrix}
		e^*&=D^*[ak-q\overline{\lambda}~gc]\\
		&=~D^*[ak-q\lambda~cg]
	\end{matrix}&&
	\begin{matrix}e^*&=D^*[ka-\frac{1}{q}\overline{\lambda}~gc]\\
		&=D^*[ka-\frac{1}{q}\lambda~cg]
	\end{matrix}\\\,\nonumber\\\nonumber
	\begin{matrix}f^*&=(-q\overline{\varsigma})~D^*[ah-q\overline{\lambda}~gb]\\
		&=(-q\overline{\varsigma})~D^*[ah-q{\gamma}~bg]
	\end{matrix}&&
	\begin{matrix}f^*	&=(-q\overline{\lambda }\overline{\varsigma}\gamma)~D^*[ha-\frac{1}{q}\overline{\gamma}~gb]\\
		&=(-q\overline{\lambda }\overline{\varsigma}\gamma)~D^*[ha-\frac{1}{q}\lambda~bg]
	\end{matrix}\\\,\nonumber\\\nonumber
	\begin{matrix}
		g^*&=(\frac{1}{q^2}\gamma \lambda)~D^*[bf-q\overline{\gamma}~ec]\\
		&=(\frac{1}{q^2}\gamma \lambda)~D^*[bf-q\varsigma~ce]
	\end{matrix}&&
	\begin{matrix}
		g^*&=(\frac{1}{q^2}{\varsigma}\lambda)D^*[fb-\frac{1}{q}\overline{\varsigma}~ec]\\
		&=(\frac{1}{q^2}{\varsigma}\lambda)~D^*[fb-\frac{1}{q}\gamma~ce]
	\end{matrix}\\\,\nonumber\\\nonumber
	\begin{matrix}h^*&=(-\frac{1}{q}\varsigma)~D^*[af-q\overline{\gamma}~dc]\\
		&=(-\frac{1}{q}\varsigma)~D^*[af-q\lambda~cd]
	\end{matrix}&&
	\begin{matrix}
		h^*&=(-\frac{1}{q}\overline{\gamma}\lambda\varsigma)~D^*[fa-\frac{1}{q}\overline{\lambda}~dc]\\
		&=(-\frac{1}{q}\overline{\gamma}\lambda\varsigma)~D^*[fa-\frac{1}{q}\gamma~cd]
	\end{matrix}\\\,\nonumber\\\nonumber
	\begin{matrix}
		k^*&=D^*[ae-q\overline{\gamma}~db]\\
		&=D^*[ae-q\gamma~bd]
	\end{matrix}&&
	\begin{matrix}
		k^*&=D^*[ea-\frac{1}{q}\overline{\gamma}~db]\\
		&=D^*[ea-\frac{1}{q}\gamma~bd]
	\end{matrix}
\end{align}

\item  { Relations involving $a^*$:}\label{eqn: a* rel}
\begin{align}
&\begin{matrix*}[l]
q~(a^*a-aa^*)=(q-\frac{1}{q})[bb^*+cc^*]\\
(a^*a-aa^*)=(q^2-1)[d^*d+g^*g]
\end{matrix*}&&
\begin{matrix*}[l]
q~a^*b=\overline{\gamma}ba^*\\
a^*b=q\overline{\gamma}~ba^*+(q^2-1)[d^*e+g^*h]
\end{matrix*}\nonumber\\\nonumber\\
&\begin{matrix*}[l]
	q~a^*c=\overline{\lambda}ca^*\\
	a^*c=q\overline{\lambda}~ca^*+(q^2-1)[d^*f+g^*k]
\end{matrix*}
&& \begin{matrix*}[l]
	\overline{\gamma}~a^*d=qda^*+(q-\frac{1}{q})(eb^*+fc^*)\\
	q\overline{\gamma}~a^*d=da^*
\end{matrix*}\\\nonumber\\
&\begin{matrix*}[l]
	a^*e=ea^*\\
\overline{\gamma}a^*f=\overline{\lambda}fa^*
\end{matrix*}&&
\begin{matrix*}[l]
	\overline{\lambda}~a^*g=qga^*+(q-\frac{1}{q})(hb^*+kc^*)\\
	q\overline{\lambda}~a^*g=ga^*
\end{matrix*}\nonumber\\\nonumber\\\nonumber
&\overline{\lambda}a^*h=\overline{\gamma}ha^*
&&a^*k=ka^*
\end{align}
\item  { Relations involving $b^*$:}\label{eqn:bb^*}
\begin{align}
&\begin{matrix*}[l]
		q~b^*a=\gamma~ ab^*\\
		b^*a-(q^2-1)(e^*d+h^*g)=q\gamma~ ab^*
	\end{matrix*}&&
	\begin{matrix*}[l]	
		q~b^*b=q~bb^*+(q-\frac{1}{q})cc^*\\
		b^*b-(q^2-1)(e^*e+h^*h)=bb^*-(q^2-1)aa^*	
	\end{matrix*}\nonumber\\\nonumber\\
&\begin{matrix*}[l]
		q~b^*c=\overline{\varsigma}~ cb^*\\
		b^*c-(q^2-1)(e^*f+h^*k)=\overline{\varsigma}~ cb^*
	\end{matrix*}&&
	\overline{\gamma}~b^*d=\gamma ~db^*\nonumber\\\,\\
&\begin{matrix*}[l]
		\overline{\gamma}~b^*e=q ~eb^*+(q-\frac{1}{q})fc^*\\
		q\overline{\gamma}~b^*e= ~eb^*-(q^2-1)da^*
	\end{matrix*}&&
	\begin{matrix*}[l]	
	\overline{\gamma}b^*f=\overline{\varsigma }fb^*\\
	\overline{\lambda}~b^*g=\gamma ~gb^*\end{matrix*}\nonumber\\\nonumber\\\nonumber
&\begin{matrix*}[l]	
		\overline{\lambda}~b^*h=q ~hb^*+(q-\frac{1}{q})kc^*\\
		q\overline{\lambda}~b^*h= ~hb^*-(q^2-1)ga^*	
	\end{matrix*}&& \overline{\lambda}~b^*k=\overline{\varsigma} ~kb^*
\end{align}

\item  {Relations involving $c^*$:}\label{eqn:c* rel}
\begin{align}
	&\begin{matrix*}[l]	
		qc^*a=\lambda~ ac^*\\
		c^*a-(q^2-1)(f^*d+k^*g)=q\lambda~ac^*
	\end{matrix*}&&
	\begin{matrix*}[l]
		qc^*b=\varsigma~ bc^*\\
		c^*b-(q^2-1)(f^*e+k^*h)=q\varsigma~bc^*		
	\end{matrix*}\nonumber\\\nonumber\\
	&\begin{matrix*}[l]
		cc^*=c^*c\\
		\overline{\gamma}~c^*d=\lambda ~dc^*\\
		\overline{\gamma}~c^*e=\varsigma ~ec^*		
	\end{matrix*}&&
	\begin{matrix*}[l]	
		\overline{\gamma}~c^*f=q ~fc^*\\
		q\overline{\gamma}~c^*f=fc^*-(q^2-1)(da^*+eb^*)	
	\end{matrix*}\\\nonumber\\\nonumber
	&\begin{matrix*}	
		\overline{\lambda}~c^*g=\lambda ~gc^*\\	
		\overline{\lambda}~c^*h=\varsigma ~hc^*
	\end{matrix*}&&
	\begin{matrix*}	
		\overline{\lambda}~c^*k=q ~kc^*\\
		q\overline{\lambda}~c^*k=kc^*-(q^2-1)(ga^*+hb^*)	
	\end{matrix*}
\end{align}

\item  { Relations involving $d^*$:}\label{eqn:d* rel}
\begin{align}
	&\begin{matrix*}[l]
			 q\gamma~d^*a=ad^*\\
			\gamma~d^*a=q~ad^*+(q-\frac{1}{q})(be^*+cf^*)	
	\end{matrix*}&&
	\begin{matrix*}[l]
		\gamma\,d^*b=\overline{\gamma}bd^*\\ 
		\gamma~d^*c=\overline{\lambda}~cd^*	
	\end{matrix*}\nonumber\\\nonumber\\
	&\begin{matrix*}[l]	
		qd^*d+(q-\frac{1}{q})a^*a=qdd^*+(q-\frac{1}{q})(ee^*+ff^*)\\
		d^*d-(q^2-1)g^*g=dd^*	
	\end{matrix*}&&
	\begin{matrix*}[l]	
		qd^*e+(q-\frac{1}{q})a^*b=\overline{\gamma}ed^*\\
		d^*e+(q^2-1)g^*h=q\overline{\gamma}ed^*	
	\end{matrix*}\nonumber\\\,\\
	&\begin{matrix*}[l]
		qd^*f+(q-\frac{1}{q})a^*c=\overline{\lambda}fd^*\\
		d^*f+(q^2-1)g^*k=q\overline{\lambda}fd^*		
	\end{matrix*}&&
	\begin{matrix*}[l]
		q\overline{\varsigma}~d^*g=gd^*\\
		\overline{\varsigma}~d^*g=q~gd^*+(q-\frac{1}{q})(he^*+kf^*)		
	\end{matrix*}\nonumber\\\nonumber\\\nonumber
	&
	\overline{\varsigma}~d^*h=\overline{\gamma}~hd^*
	&&
	\overline{\varsigma}~d^*k=\overline{\lambda}~kd^*
\end{align}
\item  { Relations involving $e^*$:}\label{eqn: e* rel}
\begin{align}
&e^*a=ae^*&&
	\begin{matrix*}[l]		
	\gamma~e^*b=qbe^*+(q-\frac{1}{q})cf^*\\
	q\gamma~e^*b=be^*-(q^2-1)ad^*
	\end{matrix*}\nonumber\\\nonumber\\
&\gamma~e^*c=\overline{\varsigma}~ce^*&&
	\begin{matrix*}[l]		
	qe^*d+(q-\frac{1}{q})b^*a=\gamma~ de^*\\
	e^*d-(q^2-1)h^*g=q\gamma~ de^*
	\end{matrix*}\nonumber\\~\\
&\begin{matrix*}[l]	
		qe^*e+(q-\frac{1}{q})b^*b=qee^*+(q-\frac{1}{q})ff^*\\
		e^*e-(q^2-1)h^*h=ee^*-(q^2-1)dd^*	
	\end{matrix*}&&
	\begin{matrix*}[l]
		qe^*f+(q-\frac{1}{q})b^*c=\overline{\varsigma}~fe^*\\
		e^*f-(q^2-1)h^*k=q\overline{\varsigma}~fe^*		
\end{matrix*}\nonumber\\\nonumber\\\nonumber
&\begin{matrix*}[l]
		\overline{\varsigma}~e^*g=\gamma~ ge^*	\\
		e^*k=ke^*		
\end{matrix*}&&
\begin{matrix*}[l]
		\overline{\varsigma}~e^*h=qhe^*+(q-\frac{1}{q})kf^*\\
		q\overline{\varsigma}~e^*h=he^*-(q^2-1)gd^*	
\end{matrix*}
\end{align}
\item { Relations involving $f^*$:}\label{eqn: f* rel}
\begin{align}
	&\begin{matrix*}[l]	
		\gamma f^*a=\lambda~ a f^*\\
		\gamma f^*b=\varsigma~ b f^*
	\end{matrix*}&&
	\begin{matrix*}[l]	
		\gamma f^*c=q~ c f^*\\
		q\gamma f^*c= c f^*-(q^2-1)(ad^*+be^*)	
	\end{matrix*}\nonumber\\\nonumber\\
	&\begin{matrix*}[l]	
		q~f^*d+(q-\frac{1}{q})c^*a=\lambda~ d f^*\\
		f^*d-(q^2-1)k^*g=q\lambda~ d f^*	
	\end{matrix*}&&
	\begin{matrix*}[l]	
		q~f^*e+(q-\frac{1}{q})c^*b=\varsigma~e f^*\\
		f^*e-(q^2-1)k^*h=q\varsigma~ e f^*	
	\end{matrix*}\nonumber\\\,\\
	&\overline{\varsigma}~ f^*g=\lambda~ g f^*
	&&\overline{\varsigma}~ f^*h=\varsigma~ h f^*\nonumber\\\nonumber\\\nonumber
	&\begin{matrix*}[l]	q~f^*f+(q-\frac{1}{q})c^*c=q~f f^*\\
		f^*f-(q^2-1)k^*k=f f^*-(q^2-1)(dd^*+ee^*)	
	\end{matrix*}&&
	\begin{matrix*}[l]
		\overline{\varsigma}~f^*k=q~k f^*\\
		q\overline{\varsigma}~f^*k=k f^*-(q^2-1)(gd^*+he^*)		
	\end{matrix*}
\end{align}
\item { Relations involving $g^*$:}\label{eqn: g* rel}
\begin{align}
	&\lambda g^*a=\frac{1}{q}ag^*,&\qquad&\lambda g^*b=\overline{\gamma}bg^*,\qquad\nonumber\\
	&\lambda g^*c=\overline{\lambda}cg^*,&&\varsigma  g^*d=\frac{1}{q}dg^*,\qquad\\\nonumber
	&\varsigma  g^*e=\overline{\gamma}eg^*,&\qquad&
	\varsigma  g^*f=\overline{\lambda}fg^*\\\nonumber
&	g^*g=gg^*&\qquad&\frac{1}{q}g^*h=\overline{\gamma}hg^*\\&\frac{1}{q}g^*k=\overline{\lambda}kg^*&&\qquad\nonumber
\end{align}
\item { Relations involving $h^*$:}\label{eqn: h* rel}
\begin{align}
	&\begin{matrix*}[l]
		\lambda h^*a=\gamma ah^*		
	\end{matrix*}&&
	\begin{matrix*}[l]
		\lambda h^*b=qbh^*+(q-\frac{1}{q})ck^*\\
		\lambda h^*b=\frac{1}{q}bh^*-(q-\frac{1}{q})ag^*
	\end{matrix*}\nonumber\\\nonumber\\
	&\begin{matrix*}[l]	
		\lambda h^*c=\overline{\varsigma}ch^*\\	
		\varsigma  h^*d=\gamma dh^*\\	
	\end{matrix*}&&
	\begin{matrix*}[l]	
		\varsigma~h^*e=q~eh^*+(q-\frac{1}{q})fk^*\\
		\varsigma~h^*e=\frac{1}{q}eh^*-(q-\frac{1}{q})dg^*	
	\end{matrix*}\\\nonumber\\
	&\begin{matrix*}[l]	
		\varsigma~h^*f=\overline{\varsigma}~fh^*\\
		h^*g=q\gamma~gh^*\\
		h^*k=q\overline{\varsigma}~kh^*	
	\end{matrix*}&&
	\begin{matrix*}[l]	
		q~h^*h+(q-\frac{1}{q})(b^*b+e^*e)=q~hh^*+(q-\frac{1}{q})kk^*\\h^*h=~hh^*-(q^2-1)gg^*	\nonumber
	\end{matrix*}
\end{align}
\item { Relations involving $k^*$:}\label{eqn: k* rel}
\begin{align}
	&\begin{matrix*}[l]	
		k^*a=ak^*\\
		\lambda~k^*b=\varsigma~bk^*\\
		\lambda k^*c=q~ck^*
	\end{matrix*}&&\nonumber
	\begin{matrix*}[l]
		\varsigma~k^*d=\lambda~dk^*\\
		~k^*e=~ek^*\\
		\varsigma k^*f= q~fk^*
	\end{matrix*}\nonumber\\~\\
	&\begin{matrix*}[l]	
		k^*g=q\lambda~gk^*\\
		k^*h=q\varsigma~hk^*	
	\end{matrix*}&&
	\begin{matrix*}[l]	
		q~k^*k+(q-\frac{1}{q})(c^*c+f^*f)=q~kk^*\\
		~k^*k=kk^*-(q^2-1)(gg^*+hh^*)	\nonumber
	\end{matrix*}
\end{align}
\end{enumerate}
\medskip

\begin{lemma}
The followings two matrices are unitary.
\begin{align}\label{eqn: unitary mat}
\begin{pmatrix}
		a & b& (q^2\gamma\lambda) \,(h^*d^*-q\overline{\gamma}\,g^*e^*)\\d& e &(-q\varsigma)\,(h^*a^*-q\overline{\gamma}\,g^*b^*)\\g&h&(e^*a^*-q\overline{\gamma}\,d^*b^*)
	\end{pmatrix},&&
	\begin{pmatrix}
	a & b& (q^2\gamma\lambda) \,(h^*d^*-q\overline{\gamma}\,g^*e^*)D\\d& e &(-q\varsigma)\,(h^*a^*-q\overline{\gamma}\,g^*b^*)D\\g&h&(e^*a^*-q\overline{\gamma}\,d^*b^*)D
    \end{pmatrix}
\end{align}
\end{lemma}
\begin{proof}
From all the commutation relation and $\ast$-relation in \eqref{subsec:3-relations} we derive the following relation which are consistent.

\begin{align}\label{eqn: main rel}\nonumber
&a b = q\,\gamma\, b a,\quad\, a d = q\,\overline{\gamma}\, d a,\quad a e = ea+ (q - q^{-1})\,\gamma b d, \quad a g = q\,\overline{\lambda}\, g a,&\\\nonumber
&a h = \gamma\overline{\lambda}\, h a+ (q - q^{-1})\gamma\, b g,\quad \gamma\, b d = \overline{\gamma}\, d b,\quad b e = q\,\overline{\gamma}\, e b,\quad \gamma\, b g = \overline{\lambda}\, g b,&\\\nonumber
&b h = q\,\overline{\lambda}\, h b,\,\, d e = q\,\gamma\, e d,\,\, d g = q\,\overline{\varsigma}\, g d,\,\, d h =\gamma \overline{\varsigma}\,\, h d+ (q - q^{-1})\,\gamma e g,&\\\nonumber
&\gamma\, e g = \overline{\varsigma} g e,\qquad e h = q\,\overline{\varsigma} h e,\qquad gh=q\gamma hg,\qquad ba^*=q\gamma a^*b,\, &\\\nonumber
&bb^*-q^2b^*b=(1-q^2)(1-aa^*),\,bd^*=\gamma^2 d^*b,\, be^*+(1-q^2)ad^*=q\gamma e^* b,\,&\\\nonumber
&bg^*=\gamma\lambda g^*b,\,bh^*+(1-q^2)ag^*=q\lambda\,h^* b, \,da^*=q\overline{\gamma}a^*d, \, dd^*=d^*d+(1-q^2)g^*g,&\\
&q\gamma de^*=e^*d+(1-q^2)h^*g,\quad dg^*=q\varsigma\,g^*d,\quad dh^*=\overline{\gamma}\varsigma\,h^*d,\quad ea^*=a^*e, &\\\nonumber
&ee^*+(1-q^2)dd^*=e^*e+(1-q^2)h^*h,\, eg^*=\gamma\varsigma g^*e,\, eh^*+(1-q^2)dg^*=q\varsigma h^*e,&\\\nonumber
&ga^*=q\overline{\lambda}a^*g,\,\, gg^*=g^*g,\,q\gamma gh^*=h^*g,\,\,ha^*=\gamma\overline{\lambda}a^*h,\,\, hh^*+(1-q^2)gg^*=h^*h, &\\ \nonumber
&aa^*-q^2a^*a=1-q^2,\,aa^*+bb^*+q^4(h^*d^*-q\overline{\gamma}g^*e^*)(dh-q\gamma\,eg)=1,\, &\\ \nonumber
& a^*a+d^*d+g^*g=1,\,dd^*+ee^*+q^2(h^*a^*-q\overline{\gamma}g^*b^*)(ah-q\gamma\,bg)=1,\,    \\\nonumber
&  b^*b+e^*e+h^*h=1,\, ad^*+be^*-q^3\gamma\lambda\overline{\varsigma}(h^*d^*-q\overline{\gamma}g^*e^*)(ah-q\gamma\,bg)=0,\\\nonumber
& a^*b+d^*e+g^*h=0,\,ag^*+bh^*+q^2\gamma\lambda(h^*d^*-q\overline{\gamma}g^*e^*)(ae-q\gamma\,bd)=0, \\\nonumber
&aD=Da,\,\,dD=\frac{\gamma\overline{\varsigma}}{\overline{\gamma}\overline{\lambda}} Dd,\,\, gD=\frac{\lambda\varsigma}{\overline{\gamma}\overline{\lambda}} Dg,\,\, bD=\frac{\overline{\gamma}\overline{\lambda}}{\gamma\overline{\varsigma}}Db,\,\,
eD=De,\,\, hD=\frac{\lambda\varsigma}{\gamma\overline{\varsigma}} Dh,&\\\nonumber
&D^*D=DD^*=1.
\end{align}
Using above relations, one can verify by a straightforward calculation that the given matrices are unitary. 
\end{proof}
\medskip

Let $\pi$ be any admissible $\ast$-representation of $\mathbb{C}[U_{q,\Theta}(3)]$ on some Hilbert space $\mathcal{H}$. Then by \eqref{unitary eqn}, we have 
\begin{align}
    &||\pi(a)||\le1,\quad ||\pi(b)||\le1,\quad||\pi(d)||\le1,
    \quad||\pi(e)||\le1,\\\nonumber
    &||\pi(g)||\le1,\quad||\pi(h)||\le1,\quad||\pi(D)||\le1
\end{align}
This implies the $C^*$-seminorm on $||x||=\text{sup}||\pi(x)||$, $\pi$ runs over all admissible representations,
is well defined. Let $N$ be the two sided ideal consists of element of seminorm zero then completion of the quotient algebra $\mathbb{C}[U_{q,\Theta}(3)]$  with respect to the induced norm is $C(U_{q,\Theta}(3))$. To prove the two sided ideal consists exactly the relation sets which the generator satisfy, it is enough to prove following theorem.
\smallskip

\begin{theorem}
The set of elements of the form \[\langle k,m,n,i,j,l,r,s\rangle:=\begin{cases}
    a_kd^md^{*n}g^ig^{*j}e_lb^rb^{*s}& \text{ if } i\ge0,j\ge0\\
    a_kd^md^{*n}g^ih^{-j}e_lb^rb^{*s} & \text{ if } i\ge0,j<0,\\
    a_kd^md^{*n}g^{*j}h^{*-i}e_lb^rb^{*s} & \text{ if } i<0,j\ge0\\
    a_kd^md^{*n}h^{-j}h^{*-i}e_lb^rb^{*s} & \text{ if } i<0,j<0,
\end{cases}\]
where $k,i,j,l \in \mathbb{Z}$ and $m,n,r,s\in \mathbb{N}$ and 
\[a_k:=\begin{cases}
a^k& \text{ if } k\ge0\\
a^{*-k}& \text{ if } k<0
\end{cases} \qquad\text{ and } \qquad e_l:=\begin{cases}
e^l& \text{ if } l\ge0\\
e^{*-l}& \text{ if } l<0
\end{cases}\]
forms a basis of $\mathbb{C}[U_{q,\Theta}(n)]$. 
\end{theorem}

\begin{proof}
 From the commutation relation of generators and their $\ast$ one can show that the set of elements of the form $\langle k,m,n,i,j,l,r,s\rangle$ spans the entire space $\mathbb{C}[U_{q,\mathbbm{1}}(n)]$. Thus, it is enough to prove that they are linearly independent. Let us consider the non trivial finite linear combination of elements
\[\sum_{\sigma\in \{0,1\}^4}\sum_{\substack{mnrs=0 \\ (k\,l\,i\,j)=\sigma}}^\text{finite}c^{\sigma}_{k,m,n,i,j,l,r,s}\langle k,m,n,i,j,l,r,s\rangle=0 \]
where $c^{\sigma}_{k,m,n,i,j,l,r,s}\in \mathbb{C}$ almost all (but not all) equal to $0$.

To prove the linear combination is non zero, we use the following representations of the $C^*$-algebra $C(U_{q,\Theta}(3))$, defined as follows.

For $q<1$ let $\mathcal{H}\equiv  \ell^2(\mathbb{N})^{\otimes 3}\otimes\ell^2(\mathbb{Z})^{\otimes 3}$
be the Hilbert space such that
\begin{itemize} \label{faithful repn}
\item $\pi(a)=S\sqrt{1-q^{2N}}\otimes I \otimes  I \otimes  I\otimes  I \otimes  I$

\item$\pi(b)=q(q\gamma)^N\otimes \gamma^{2N}\otimes S\sqrt{1-q^{2N}}\otimes  I\otimes  I \otimes  I$

\item $\pi(c)=q^2(q\lambda)^N\otimes (\gamma\lambda)^N\otimes (q\varsigma)^N\otimes U\otimes  I \otimes  I$

\item $\pi(d)=(q\overline{\gamma})^N\otimes S\sqrt{1-q^{2N}}\otimes I \otimes  I\otimes I \otimes  I$

\item $\pi(e)=-q\overline{\lambda}\otimes (q\gamma)^N\otimes (q\overline{\gamma})^N \otimes\overline{\lambda}^{2N}U^*(\gamma\lambda^2\varsigma)^N\otimes U^*(\overline{\gamma}\overline{\lambda}^2 \overline{\varsigma})^N \otimes U\\
-\gamma\sqrt{1-q^{2N}}S^*\otimes \gamma^{2N} S\sqrt{1-q^{2N}}\otimes S\sqrt{1-q^{2N}}\otimes I \otimes  I \otimes  I$
\item $\pi(f)=  q\overline{\gamma}\varsigma(\overline{\gamma}\lambda)^N\otimes (q\lambda)^N\otimes (\gamma\lambda)^N \sqrt{1-q^{2N}}S^*(\overline{\gamma}^2\overline{\lambda}\varsigma)^N\otimes (\gamma\varsigma)^N\otimes U^*(\overline{\gamma}\overline{\lambda}^2 \overline{\varsigma})^N \otimes U\\
-q\lambda\sqrt{1-q^{2N}}S^*(\overline{\gamma}\lambda)^N\otimes(\gamma\lambda)^N S\sqrt{1-q^{2N}}\otimes (q\varsigma)^N\otimes U\otimes I\otimes I$

\item $\pi(g)=(q\overline{\lambda})^N\otimes (q\overline{\varsigma})^N\otimes (\overline{\gamma}\overline{\lambda})^N\otimes \lambda^{2N}\otimes U\otimes I$

\item $\pi(h)= \overline{\lambda}\overline{\varsigma} (\gamma\overline{\lambda})^N\otimes\sqrt{1-q^{2N}}S^*(\gamma\overline{\varsigma})^N\otimes (q\overline{\gamma}^2\overline{\lambda})^N \otimes U^*(\gamma\lambda^2\varsigma)^N\otimes (\overline{\gamma}\overline{\lambda}^2 \overline{\varsigma})^N \otimes U\\
-\gamma\sqrt{1-q^{2N}}S^*(\gamma\overline{\lambda})^N\otimes(q\gamma^2\overline{\varsigma})^N\otimes S\sqrt{1-q^{2N}}(\overline{\gamma}\overline{\lambda})^N\otimes\lambda^{2N}\otimes U\otimes I $
\item $\pi(k)= -q\lambda\sqrt{1-q^{2N}}S^* \otimes (q\gamma\lambda\overline{\varsigma})^N\otimes (q\overline{\gamma}\overline{\lambda}\varsigma)^N \otimes U\lambda^{2N}\otimes U\otimes I\\
-\overline{\gamma}\otimes \sqrt{1-q^{2N}}S^*(\lambda\overline{\varsigma})^N\otimes\sqrt{1-q^{2N}}S^*(\overline{\gamma}^2\overline{\lambda}\varsigma)^N \otimes (\gamma\lambda^2\varsigma)^N\otimes (\overline{\gamma}\overline{\lambda}^2 \overline{\varsigma})^N \otimes U$

\item $\pi(D)=  I\otimes (\gamma^2\lambda\overline{\varsigma})^N\otimes (\overline{\gamma}^2\overline{\lambda}\varsigma)^N \otimes (\gamma\lambda^2\varsigma)^N\otimes (\overline{\gamma}\overline{\lambda}^2 \overline{\varsigma})^N \otimes U$
\end{itemize}
and for $q>1$ we has similar representation on $\mathcal{H}\equiv  \ell^2(\mathbb{N})^{\otimes 3}\otimes\ell^2(\mathbb{Z})^{\otimes 3}$ such that
\begin{itemize} \label{faithful repn}
\item $\pi(a)=\sqrt{1-q^{-2N}}S^*\otimes I \otimes  I \otimes  I\otimes  I \otimes  I$

\item$\pi(b)=(q\gamma)^{-N}\otimes \overline{\gamma}^{2N}\otimes \sqrt{1-q^{-2N}}S^*\otimes  I\otimes  I \otimes  I$

\item $\pi(c)=(q\lambda)^{-N}\otimes (\overline{\gamma}\overline{\lambda})^N\otimes (q\varsigma)^{-N}\otimes U\otimes  I \otimes  I$

\item $\pi(d)=q^{-1}(q\overline{\gamma})^{-N}\otimes \sqrt{1-q^{-2N}}S^*\otimes I \otimes  I\otimes I \otimes  I$

\item $\pi(e)=-q^{-1}\overline{\lambda}\otimes (q\gamma)^{-N}\otimes (q\overline{\gamma})^{-N} \otimes\overline{\lambda}^{2N}U^*(\gamma\lambda^2\varsigma)^N\otimes U^*(\overline{\gamma}\overline{\lambda}^2 \overline{\varsigma})^N \otimes U\\
-\gamma S\sqrt{1-q^{-2N}}\otimes \overline{\gamma}^{2N} \sqrt{1-q^{-2N}}S^*\otimes \sqrt{1-q^{-2N}}S^*\otimes I \otimes  I \otimes  I$
\item $\pi(f)=  \overline{\gamma}\varsigma(\overline{\gamma}\lambda)^{-N}\otimes (q\lambda)^{-N}\otimes (\overline{\gamma}\overline{\lambda})^N S\sqrt{1-q^{-2N}} (\gamma^2\lambda\overline{\varsigma})^N\otimes (\gamma\varsigma)^N\otimes U^*(\overline{\gamma}\overline{\lambda}^2 \overline{\varsigma})^N \otimes U\\
-\lambda S\sqrt{1-q^{-2N}}(\overline{\gamma}\lambda)^{-N}\otimes(\overline{\gamma}\overline{\lambda})^N \sqrt{1-q^{-2N}}S^*\otimes (q\varsigma)^{-N}\otimes U\otimes I\otimes I$

\item $\pi(g)=q^{-2}(q\overline{\lambda})^{-N}\otimes (q\overline{\varsigma})^{-N}\otimes (\gamma\lambda)^N\otimes \lambda^{2N}\otimes U\otimes I$

\item $\pi(h)= q^{-1}\overline{\lambda}\overline{\varsigma} (\gamma\overline{\lambda})^{-N}\otimes S\sqrt{1-q^{-2N}}(\overline{\gamma}\varsigma)^N\otimes (q\overline{\gamma}^2\overline{\lambda})^{-N} \otimes U^*(\gamma\lambda^2\varsigma)^N\otimes (\overline{\gamma}\overline{\lambda}^2 \overline{\varsigma})^N \otimes U\\
-q^{-1}\gamma S\sqrt{1-q^{-2N}}(\gamma\overline{\lambda})^{-N}\otimes(q\gamma^2\overline{\varsigma})^{-N}\otimes \sqrt{1-q^{-2N}}S^*(\gamma\lambda)^N\otimes\lambda^{2N}\otimes U\otimes I $
\item $\pi(k)= -q^{-1}\lambda S\sqrt{1-q^{-2N}} \otimes (q\gamma\lambda\overline{\varsigma})^{-N}\otimes (q\overline{\gamma}\overline{\lambda}\varsigma)^{-N} \otimes U\lambda^{2N}\otimes U\otimes I\\ 
-\overline{\gamma}\otimes S\sqrt{1-q^{-2N}}(\overline{\lambda}\varsigma)^N\otimes S\sqrt{1-q^{-2N}}(\gamma^2\lambda\overline{\varsigma})^N \otimes (\gamma\lambda^2\varsigma)^N\otimes (\overline{\gamma}\overline{\lambda}^2 \overline{\varsigma})^N \otimes U$

\item $\pi(D)=  I\otimes (\overline{\gamma}^2\overline{\lambda}\varsigma)^N\otimes (\gamma^2\lambda\overline{\varsigma})^N \otimes (\gamma\lambda^2\varsigma)^N\otimes (\overline{\gamma}\overline{\lambda}^2 \overline{\varsigma})^N \otimes U$

\end{itemize}

For $q=1$, we consider the Hilbert space 
$\mathcal{H}\equiv  \ell^2(\mathbb{Z})^{\otimes 3}\otimes\ell^2(\mathbb{Z})^{\otimes 3}$
 such that 

\begin{itemize} \label{faithful repn1}
\item $\pi(a)=\text{cos}\phi_1\,S\otimes I \otimes  I \otimes  I\otimes  I \otimes  I$

\item$\pi(b)=\text{sin}\phi_1\text{cos}\phi_3(\gamma)^N\otimes \gamma^{2N}\otimes S\otimes  I\otimes  I \otimes  I$

\item $\pi(c)=\text{sin}\phi_1\text{sin}\phi_3(\lambda)^N\otimes (\gamma\lambda)^N\otimes ( \varsigma)^N\otimes U\otimes  I \otimes  I$

\item $\pi(d)=\text{sin}\phi_1\text{cos}\phi_2(\overline{\gamma})^N\otimes S\otimes I \otimes  I\otimes I \otimes  I$

\item $\pi(e)=-\text{sin}\phi_2\text{sin}\phi_3\,\overline{\lambda}\otimes (\gamma)^N\otimes (\overline{\gamma})^N \otimes\overline{\lambda}^{2N}U^*(\gamma\lambda^2\varsigma)^N\otimes U^*(\overline{\gamma}\overline{\lambda}^2 \overline{\varsigma})^N \otimes U\\
-\text{cos}\phi_1\text{cos}\phi_2\text{cos}\phi_3\gamma \,S^*\otimes \gamma^{2N} S\otimes S\otimes I \otimes  I \otimes  I$
\item $\pi(f)=\text{sin}\phi_2\text{cos}\phi_3\, \overline{\gamma}\varsigma(\overline{\gamma}\lambda)^N\otimes ( \lambda)^N\otimes (\gamma\lambda)^N S^*(\overline{\gamma}^2\overline{\lambda}\varsigma)^N\otimes (\gamma\varsigma)^N\otimes U^*(\overline{\gamma}\overline{\lambda}^2 \overline{\varsigma})^N \otimes U\\
-\text{cos}\phi_1\text{cos}\phi_2\text{sin}\phi_3\lambda S^*(\overline{\gamma}\lambda)^N\otimes(\gamma\lambda)^N S \otimes (\varsigma)^N\otimes U\otimes I\otimes I$

\item $\pi(g)=\text{sin}\phi_1\text{sin}\phi_2(\overline{\lambda})^N\otimes (\overline{\varsigma})^N\otimes (\overline{\gamma}\overline{\lambda})^N\otimes \lambda^{2N}\otimes U\otimes I$

\item $\pi(h)=\text{cos}\phi_2\text{sin}\phi_3\,\overline{\lambda}\overline{\varsigma} (\gamma\overline{\lambda})^N\otimes S^*(\gamma\overline{\varsigma})^N\otimes (\overline{\gamma}^2\overline{\lambda})^N \otimes U^*(\gamma\lambda^2\varsigma)^N\otimes (\overline{\gamma}\overline{\lambda}^2 \overline{\varsigma})^N \otimes U\\
-\text{cos}\phi_1\text{sin}\phi_2\text{cos}\phi_3\gamma S^*(\gamma\overline{\lambda})^N\otimes(\gamma^2\overline{\varsigma})^N\otimes S (\overline{\gamma}\overline{\lambda})^N\otimes\lambda^{2N}\otimes U\otimes I $
\item $\pi(k)= -\text{cos}\phi_1\text{sin}\phi_2\text{sin}\phi_3\,\lambda S^* \otimes (\gamma\lambda\overline{\varsigma})^N\otimes (\overline{\gamma}\overline{\lambda}\varsigma)^N \otimes U\lambda^{2N}\otimes U\otimes I\\
-\text{cos}\phi_2\text{cos}\phi_3\overline{\gamma}\otimes  S^*(\lambda\overline{\varsigma})^N\otimes S^*(\overline{\gamma}^2\overline{\lambda}\varsigma)^N \otimes (\gamma\lambda^2\varsigma)^N\otimes (\overline{\gamma}\overline{\lambda}^2 \overline{\varsigma})^N \otimes U$

\item $\pi(D)=  I\otimes (\gamma^2\lambda\overline{\varsigma})^N\otimes (\overline{\gamma}^2\overline{\lambda}\varsigma)^N \otimes (\gamma\lambda^2\varsigma)^N\otimes (\overline{\gamma}\overline{\lambda}^2 \overline{\varsigma})^N \otimes U$

\end{itemize}
The proof for all three cases exactly same as \cite{wor87}. We leave it for the reader.
\end{proof}
\begin{theorem}
	There is $C^*$-algebra homomorphisms $\Delta:C(U_{q,\Theta}(3))\rightarrow C(U_{q,\Theta}(3))\otimes C(U_{q,\Theta}(3))$, $\epsilon:C(U_{q,\Theta}(3))\rightarrow \mathbb{C}$ and anti algebra homomorphism $S:\mathbb{C}[U_{q,\Theta}(3)]\rightarrow \mathbb{C}[U_{q,\Theta}(3)]$ such that
\begin{align}\nonumber
&\Delta(a) = a\otimes a+b\otimes d+(q^2\gamma\lambda) \,(h^*d^*-q\overline{\gamma}\,g^*e^*)D\otimes g,\quad\,\,\\\nonumber
&\Delta(b)=a\otimes a+b\otimes e+ (q^2\gamma\lambda) \,(h^*d^*-q\overline{\gamma}\,g^*e^*)D\otimes h,\quad\,\,\,\\
\,&\Delta(d) = d\otimes a+e\otimes d+(-q\varsigma)\,(h^*a^*-q\overline{\gamma}\,g^*b^*)D\otimes g,\\\nonumber
&\Delta(e) = d\otimes b+e\otimes e+(-q\varsigma)\,(h^*a^*-q\overline{\gamma}\,g^*b^*)D\otimes h,\,\\\nonumber
&\Delta(g) = g\otimes a+h\otimes d+(e^*a^*-q\overline{\gamma}\,d^*b^*)D\otimes g,\qquad\qquad\qquad\\\nonumber
&\Delta(h) = g\otimes b+h\otimes e+(e^*a^*-q\overline{\gamma}\,d^*b^*)D\otimes h,\qquad\qquad\quad\,
\end{align}
 \begin{align}
 \varepsilon(a)=1,\quad \varepsilon(d)=0,\quad \varepsilon(g)=0,\quad \varepsilon(b)=0,\\\nonumber
 \quad \varepsilon(e)=1,\quad \varepsilon(h)=0,\quad\epsilon(D)=1.\qquad
 \end{align}
 \begin{align}\nonumber
 S(a)=a^*,\qquad S(b)=d^*,\qquad S(d)=b^*,
 \qquad S(e)=e^*,\\\nonumber
  S(g)=(q^2\overline{\lambda}^2\overline{\varsigma}\overline{\gamma})\,[dh-q\overline{\varsigma}\,ge],
 \quad S(h)=(-q\overline{\lambda }\overline{\varsigma}^2\gamma)\,[ah-q\overline{\lambda}\,gb]D,\\
 S(a^*)=a,\qquad S(b^*)=d,\qquad S(d^*)=b,
 \qquad S(e^*)=e,\\\nonumber
 S(g^*)=(q^2\lambda^2\varsigma^2) \,(d^*h^*-\frac{1}{q}\overline{\varsigma}\,g^*e^*)D,\qquad\qquad\\\quad S(h^*)=(-q\lambda^2\varsigma^2\overline{\gamma}^2)\,(a^*h^*-\frac{1}{q}\gamma\,b^*g^*)D,\quad
 S(D)=D^*.\nonumber
 \end{align}
\end{theorem}
\qed

\section{Irreducible Representations of $C(U_{q,\Theta}(3))$}\label{sec:irred-rep}
In this section, we investigate the irreducible representations of the $C^*$-algebra $C(U_{q,\Theta}(3))$. Let $\pi$ be a irreducible representation of $C(U_{q,\Theta}(3))$ on some Hilbert space $\mathcal{H}$. Let $a=\pi(V_{1,1})$, $b=\pi(V_{1,2})$, $c=\pi(V_{1,3})$, $d=\pi(V_{2,1})$, $e=\pi(V_{2,2})$, $f=\pi(V_{2,3})$, $g=\pi(V_{3,1})$, $h=\pi(V_{3,2})$, $k=\pi(V_{3,3})$ and $D =\pi(Det_q)$. Also let $\gamma=\theta_{2,1}$, $\lambda=\theta_{3,1}$ and $\varsigma=\theta_{3,2}$. Then $\overline{\gamma}=\theta_{1,2}$, $\overline{\lambda}=\theta_{1,3}$ and $\overline{\varsigma}=\theta_{2,3}$. 

One can write  $\mathcal{H}=Ker(g)\oplus \overline{range(g^*)}$. Since $g$ is normal, we have $Ker(g)=Ker(g^*)$ and $\overline{range(g)}=\overline{range(g^*)}$. Thus, $\mathcal{H}=Ker(g)\oplus \overline{range(g)}$

As $g$ and $c$ are normal, Let $\mathcal{H}_1=Ker(g)\cap Ker(c)$, $\mathcal{H}_2=Ker(g)\cap Ker(c)^\perp$, $\mathcal{H}_3=Ker(g)^\perp\cap Ker(c)$ and $\mathcal{H}_4=Ker(g)^\perp\cap Ker(c)\perp$. All of them are invariant subspaces.
Let us first prove some lemmas.

\begin{lemma}
    If $g=c=0$ then either $d= 0$ or $h=0$.
\end{lemma}
\begin{proof}
    $g=0$ and $c^*=0$ implies $dh=hd=0$. Let $d\neq 0$. As $\overline{range(d)}$ is invariant, thus by irreducibility we have $\mathcal{H}=\overline{range(d)}$. Now, $\xi\in range(d)$ implies $\xi=d~\xi'$ for some $\xi'$. Thus, $h~\xi=hd~\xi'=0$. Hence, $h=0$.
\end{proof}

\begin{lemma}
    If $g\neq 0$ and $c\neq 0$ then $gc\neq 0$.
\end{lemma}

\begin{proof}
    Let $g\neq0$ but $gc=0$.  As $\overline{range(g)}$ is invariant, thus by irreducibility we have $\mathcal{H}=\overline{range(g)}$. Now, $\xi\in range(g) $ implies $\xi=g~\xi'$ for some $\xi'$. Thus, $c~\xi=cg~\xi'=\overline{\lambda}^2~gc~\xi'=0$. This contradicts the fact that $c\ne0$.
\end{proof}

Therefore we have the following cases:
\begin{itemize}
    \item Case-1: $g=0$, $d=h=0$. So we have $c=0$.
    
    \item Case-2: $g=0$, $h=0$, $Ker(d)=\{0\}$. So we have $c=0$.
    \item Case-3: $g=0$, $d=0$, $Ker(h)=\{0\}$. So we have $c=0$.
    \item Case-4: $Ker(g)=\{0\}$, $c=0$.
    \item Case-5: $Ker(c)=\{0\}$, $g=0$.
    \item Case-6: $Ker(c)=\{0\}=Ker(g)$. So $gc\neq 0$
\end{itemize}
\smallskip

The case-$1$ is equivalent to  the one dimensional representations $\chi_{(\mu_1,\mu_2,\mu_3)}:C(U_{q,\Theta}(3))\longrightarrow \mathbb{C}$, are given as follows
\begin{align*}
	\begin{matrix*}[l]
		a&\mapsto \mu_1,\\
		k&\mapsto \frac{\mu_3}{\mu_2},
	\end{matrix*}&&
	\begin{matrix*}[l]
	 e&\mapsto \frac{\mu_2}{\mu_1}, \\
	 D&\mapsto \mu_3 .
	\end{matrix*}
\end{align*}
and other generators are zero for some unit modulous complex numbers $\mu_1,\mu_2,\mu_3$.

\subsection{Infinite dimensional representations of $C(U_{q,\Theta}(3))$:}
Here, we list some infinite dimensional representations of the $C^*$-algebra $C(U_{q,\Theta}(3))$ for $0<q<1$. 
\subsubsection{Case 2:}
Let $(\Game_1, \mathbb{K}_{{\lambda}\overline{\varsigma},\gamma^2,1})$ be an irreducible representation of noncommutative $3$-torus $\mathbb{T}_{{\lambda}\overline{\varsigma},\gamma^2,1}(3)$. Then the representation $(\psi_{1,\Game_1},\ell^2(\mathbb{N})\otimes \mathbb{K}_{{\lambda}\overline{\varsigma},\gamma^2,1})$ of the $C^*$-algebra $C(U_{q,\Theta}(3))$ is given as follows:

	\begin{tabular}{lllll}
		$a\mapsto S\sqrt{1-q^{2N}}\otimes I,$	&&  $b\mapsto -q^{N+1}{\gamma}^{N+1}\otimes \Game_1(W^*Y)$,&&\\
		$d\mapsto q^{N}\overline{\gamma}^N\otimes  \Game_1(W),$&&$e\mapsto \sqrt{1-q^{2N}}S^*\otimes \Game_1(Y),$&&  \\
		$k\mapsto I\otimes \Game_1(X)$	&& $D\mapsto  I\otimes \Game_1(XY)$, &&$\Gamma_0\mapsto I\otimes \Game_1(Y)$ .
	\end{tabular}
	\subsubsection{Case 3:}
	Let $(\Game_2, \mathbb{K}_{\gamma\overline{\lambda}, 1, \varsigma^2})$ be an irreducible representation of noncommutative $3$-torus $\mathbb{T}_{\gamma\overline{\lambda}, 1, \varsigma^2}(3)$. Then the representation $(\psi_{2,\Game_2},\ell^2(\mathbb{N})\otimes \mathbb{K}_{\gamma\overline{\lambda}, 1, \varsigma^2})$ of the $C^*$-algebra $C(U_{q,\Theta}(3))$ is given as follows:	
	
\begin{tabular}{lllll}
		$a\mapsto  I\otimes \Game_2(W),$	&&  $e\mapsto  S\sqrt{1-q^{2N}}\otimes   I$, && \\
		$f\mapsto  -q^{N+1}{\varsigma}^{N+1}\otimes \Game_2(X^*Y),$ && $h\mapsto  q^N\overline{\varsigma}^N\otimes \Game_2(X),$ &&  \\
		$k\mapsto  \sqrt{1-q^{2N}}S^*\otimes \Game_2(Y)$	&& $D\mapsto I\otimes \Game_2(YW) $ && $\Gamma_0=I\otimes \Game_2(Y)$ 
	\end{tabular}
	
\subsubsection{Case 4:}
Let $(\Game_3, \mathbb{K}_{a_1,a_2,a_3})$ be an irreducible representation of noncommutative $3$-torus $\mathbb{T}_{a_1,a_2,a_3}(3)$ where $a_1=\gamma^2\lambda\overline{\varsigma},$, $a_2=\overline{\gamma}\lambda\varsigma^2$ and $a_3=\varsigma\overline{\gamma}$. Then the representation $(\psi_{3,\Game_3},\ell^2(\mathbb{N})\otimes\ell^2(\mathbb{N})\otimes\mathbb{K}_{a_1,a_2,a_3})$ of the $C^*$-algebra $C(U_{q,\Theta}(3))$ is given as follows:
\begin{align*}
	a \mapsto &\; I \otimes S\sqrt{1-q^{2N}}\otimes I, \\
	b \mapsto &\; (\varsigma\overline{\lambda})^N \otimes (-q^{N+1})\gamma^{N} \otimes \Game_3(X), \\
	d \mapsto &\; \gamma\overline{\lambda }\,S\sqrt{1-q^{2N}} \lambda^N\overline{\varsigma}^N\otimes q^{N}\overline{\gamma}^{N} \otimes \Game_3(WX^*), \\
	e \mapsto &\; S\sqrt{1-q^{2N}} \otimes \sqrt{1-q^{2N}}S^* \otimes \Game_3(W), \\
	f \mapsto &\; -q^{N+1}\overline{\varsigma}^N \otimes \lambda^N\overline{\gamma}^{N} \otimes \Game_3(Y), \\
	g \mapsto &\; \lambda\gamma\, q^N\lambda^N \otimes q^N\overline{\lambda}^{N} \otimes \Game_3(WX^*Y^*), \\
	h \mapsto &\; \varsigma\, q^{N}{\varsigma}^{N} \otimes \sqrt{1-q^{2N}}S^* \gamma^{N}\overline{\lambda}^{N} \otimes \Game_3(Y^*W), \\
	k \mapsto &\; \sqrt{1-q^{2N}}S^*\otimes I\otimes I, \\
	D \mapsto &\;I\otimes I\otimes \Game_3(W).
\end{align*} 

\subsubsection{Case 5:}
Let $(\Game_4, \mathbb{K}_{a_1,a_2,a_3})$ be an irreducible representation of noncommutative $3$-torus $\mathbb{T}_{a_1,a_2,a_3}(3)$ where $a_1=\gamma\overline{\varsigma},$, $a_2=\gamma^2\overline{\varsigma}\lambda$ and $a_3=\overline{\gamma}\varsigma^2\lambda $. Then the representation $(\psi_{4,\Game_4},\ell^2(\mathbb{N})\otimes\ell^2(\mathbb{N})\otimes\mathbb{K}_{a_1,a_2,a_3})$ of the $C^*$-algebra $C(U_{q,\Theta}(3))$ is given as follows:

\begin{align*}
	a \mapsto &\; I\otimes S\sqrt{1-q^{2N}}\otimes I, \\
	b \mapsto &\; (\overline{\varsigma}\lambda\gamma )~S\sqrt{1-q^{2N}}(\overline{\lambda}\varsigma)^N \otimes (-q^{N+1}\gamma^{N} )\otimes \Game_4(W^*Y), \\
	c\mapsto &\lambda\varsigma~(-q^{N+1}\overline{\lambda}^N)\otimes (-q^{N+1}{\lambda}^N)\otimes \Game_4(W^*X^*Y)\\
	d \mapsto &(\lambda \overline{\varsigma})^N\otimes q^N\overline{\gamma}^N\otimes \Game_4(W)\\
	e \mapsto &\; S\sqrt{1-q^{2N}} \otimes \sqrt{1-q^{2N}}S^* \otimes \Game_4(Y), \\
	f \mapsto &\; \lambda\varsigma\overline{\gamma}~(-q^{N+1}\overline{\varsigma}^N )\otimes \sqrt{1-q^{2n}}S^*\lambda^N\overline{\gamma}^{N} \otimes \Game_4(X^*Y), \\
	h \mapsto &\; q^{N}{\varsigma}^{N} \otimes  \gamma^{N}\overline{\lambda}^{N} \otimes \Game_4(X), \\
	k \mapsto &\; \sqrt{1-q^{2N}}S^*\otimes I\otimes I, \\
	D \mapsto &\; I\otimes I\otimes \Game_4(Y).
\end{align*}

\subsubsection{Case 6:}
Let $(\Game_5, \mathbb{K}_{a_1,a_2,a_3})$ be an irreducible representation of noncommutative $3$-torus $\mathbb{T}_{a_1,a_2,a_3}(3)$ where $a_1=\gamma\lambda^2\varsigma$, $a_2=\overline{\gamma}\overline{\lambda}^2\overline{\varsigma}$ and $a_3=\overline{\lambda}^2$. Then the representation $(\psi_{5,\Game_5},\ell^2(\mathbb{N})\otimes \ell^2(\mathbb{N})\otimes\ell^2(\mathbb{N})\otimes\mathbb{K}_{a_1,a_2,a_3})$ of the $C^*$-algebra $C(U_{q,\Theta}(3))$ is given as follows:
\begin{align*}
a \mapsto &\; I\otimes S\sqrt{1-q^{2N}}\otimes I\otimes I, \\
b \mapsto &\; \gamma^{2N}\otimes q^{N+1}\gamma^N\otimes ~S\sqrt{1-q^{2N}}\otimes I, \\
c\mapsto &\;(\lambda\gamma)^N\otimes q^{N+1}\lambda^N\otimes q^{N+1}\varsigma^N\otimes \Game_5(X),\\
d \mapsto &S\sqrt{1-q^{2N}}\otimes q^N\overline{\gamma}^N\otimes I\otimes I,\\
g \mapsto &\; (q\overline{\varsigma})^N \otimes (q\overline{\lambda})^N \otimes (\overline{\lambda}\overline{\gamma})^N\otimes \Game_5(Y), \\
D \mapsto &\; (\gamma^2\lambda\overline{\varsigma})^N\otimes1\otimes (\overline{\gamma}^2\overline{\lambda}\varsigma)^N \otimes \Game_5(W)\\
e\mapsto &\; - (q\gamma)^N\otimes q\overline{\lambda} \otimes (q\overline{\gamma})^N\otimes  \Game_5(Y^*X^*W)\\
\,&\;- \gamma^{2N} S\sqrt{1-q^{2N}}\otimes \gamma\sqrt{1-q^{2N}}S^*\otimes S\sqrt{1-q^{2N}}\otimes I\\
f\mapsto &\;  (q\lambda)^N\otimes q\overline{\gamma}\varsigma(\overline{\gamma}\lambda)^N \otimes (\gamma\lambda)^N \sqrt{1-q^{2N}}S^*(\overline{\gamma}^2\overline{\lambda}\varsigma)^N\otimes \Game_5(Y^*W )\\
\,&\;-(\gamma\lambda)^N S\sqrt{1-q^{2N}}\otimes q\lambda\sqrt{1-q^{2N}}S^*(\overline{\gamma}\lambda)^N\otimes (q\varsigma)^N\otimes \Game_5(X)\\
h\mapsto &\; \sqrt{1-q^{2N}}S^*(\gamma\overline{\varsigma})^N\otimes\overline{\lambda}\overline{\varsigma} (\gamma\overline{\lambda})^N\otimes (q\overline{\gamma}^2\overline{\lambda})^N \otimes \Game_5(X^*W )\\
\,&\;-(q\gamma^2\overline{\varsigma})^N\otimes \gamma\sqrt{1-q^{2N}}S^*(\gamma\overline{\lambda})^N\otimes S\sqrt{1-q^{2N}}(\overline{\gamma}\overline{\lambda})^N\otimes \Game_5(Y),\\
k\mapsto&\; -(q\gamma\lambda\overline{\varsigma})^N\otimes q\lambda\sqrt{1-q^{2N}}S^* \otimes  (q\overline{\gamma}\overline{\lambda}\varsigma)^N \otimes \Game_5(XY)\\
\,&\; -\sqrt{1-q^{2N}}S^*(\lambda\overline{\varsigma})^N\otimes \overline{\gamma}I \otimes \sqrt{1-q^{2N}}S^*(\overline{\gamma}^2\overline{\lambda}\varsigma)^N \otimes \Game_5(W),
\end{align*}

\subsection{Irreducibility of above representations:}
Consider the collection of representations $\mathscr{R}=\{\chi_{(\mu_1,\mu_2,\mu_3)}:~|\mu_1|=|\mu_2|=|\mu_3|=1 \}\cup\{\psi_{1,\Game_1},~ \psi_{2,\Game_2},~ \psi_{3,\Game_3}, ~\psi_{4,\Game_4},~ \psi_{5,\Game_5}~\}.$ Then we have the following theorems.
\begin{theorem}\label{main-theorem}
	Let $\pi$ be irreducible representations of $C(U_{q,\Theta}(3))$. Then $\pi$ is equivalent to one of the representations in $\mathscr{R}$.
\end{theorem}
\begin{proof}
	Since the proof requires heavy computation, we address it in the section \ref{sec:main-theorem} for reader's convenience.
\end{proof}
\begin{remark}
	In the proof of the above theorem, $\mathcal{H}_0$ denotes the intersection of kernels of some generators and its $*$ which corresponds to the underlying Hilbert space of noncommutative $3$-torus.
\end{remark}
In each cases defined above, $\mathcal{H}_0$ is given below which is crucial for next theorem.
\begin{align*}
	\mathcal{H}_0=\left\{ \begin{matrix}
		Ker(a)&\text{ for case 2, }&\\
		Ker(e)&\text{ for case 3, }&\\
		Ker(a)\cap Ker(k^*)&\text{ for case 4 }&\text{and 5, }\\
		Ker(a)\cap Ker(d)\cap Ker(b)&\text{ for case 6, }&
	\end{matrix}\right.
\end{align*}
\begin{theorem}
	 Let $\pi$ be one of the representations in $\mathscr{R}$. Then $\pi$ is irreducible. Also, the representations in $\mathscr{R}$ are inequivalent.
\end{theorem}
\begin{proof}
	The inequivalence of different representations follows by comparing the spectrum of the image of generators.
	
	Now for the proof of irreducibility, Let $(\pi,\mathcal{H})$ be infinite dimensional representations in $\mathscr{R}$ and $T$ be the commutator of image of $\pi$.
	As $T$ commutes with all image of the generators, $\mathcal{H}_0$ is invariant under $T$. As $T$ commutes with the image of generators, $T|_{\mathcal{H}_0}$ commutes with the image of generators for noncommutative $3$-torus. So $T\xi=s\xi$ for some scalar $s$ and $\xi\in\mathcal{H}_0$. Since the Hilbert space $\mathcal{H}$ is closed linear span of $\{\pi(m)\xi:\xi\in\mathcal{H}_0,~m\in\mathbb{C}[U_{q,\Theta}(3)]\}$, then $T\pi(m)\xi=\pi(m)T\xi=s\pi(m)\xi$. i.e. $T=sI$. Hence, $(\pi,\mathcal{H})$ is irreducible.
\end{proof}

\section{Proof of the theorem \ref{main-theorem} }\label{sec:main-theorem}
 Let $W,X,Y$ denotes the generators of noncommutative $3$-torus. For details see \ref{subsec:nctorus}. 
\subsection{case-1 : $g=d=h=0$} 
Here $a,e,k$ and $D$ are commute with each other. Then $\pi:C(U_{q,\Theta}(3))\longrightarrow \mathbb{C}$ is a representation such that $a\mapsto \mu_1$, $e\mapsto \frac{\mu_2}{\mu_1}$, $k\mapsto \frac{\mu_3}{\mu_2}$ and $D\mapsto \mu_3$ and other generators are zero; i.e. $\pi\equiv\chi_{(\mu_1,\mu_2,\mu_3)}$.

\subsection{case-2 : $g=h=0$, $Ker(d)=\{0\}$} 
\begin{theorem}
	Let $(\pi,\mathcal{H})$ be an irreducible representation satisfying the condition of case 2. Let $\mathcal{H}_0=Ker(d^*d-1)$. Then $(d|_{\mathcal{H}_0},k|_{\mathcal{H}_0},\Gamma_0|_{\mathcal{H}_0})$ is the representation of noncommutative $3$-torus $\mathbb{T}_{a_1,a_2,1}(3)$ on $\mathcal{H}_0$ where $a_1=\lambda\overline{\varsigma}$, $a_2=\gamma^2$ and $\Gamma_0 =k^*D$.
\end{theorem}

\proof Here, we have $g=h=0$, $Ker(d)=\{0\}$.
Then $c=f=0$. Thus, one can obtain the following relations:
\begin{align*}
	\overline{\lambda} d k = \overline{\varsigma} kd, && d^*=(-\frac{1}{q}\overline{\varsigma}\lambda\gamma)~D^*bk \\
	\overline{\varsigma} b k = \overline{\lambda} kb, && a^*=D^*ek.
\end{align*}
Therefore, $k$, $D$ are unitary. Let $\Gamma_0=ae-q\gamma~bd$. Then $\Gamma_0k=D$, i.e., $\Gamma_0=Dk^*=k^*D$ is also unitary.

So $d^*=(-\frac{1}{q}\gamma)\Gamma_0^*b$ and $a^*=\Gamma_0^* e$. Now, $d$ is normal implies either $Ker(d)=\{0\}$ or $d=0$. If $d=0$, then representation is one dimensional and $a,e,k$ are all unitary, also commute with each other, which is same as case I.
Therefore, $Ker(d)=\{0\}$. As $a^*a+d^*d=1$, let $\mathcal{H}_0=Ker(d^*d-1)$. Then $\mathcal{H}_0\neq \{0\}$  as $a$ is non unitary and satisfies the relation of quantum disk algebra and $\mathcal{H}_0$ is invariant under $k$ and $\Gamma_0$. Also, $d|_{\mathcal{H}_0}$ is unitary. On $\mathcal{H}_0$, $k$, $\Gamma_0$ and $d$ are unitary and $k$, $\Gamma_0$ commutes. Thus, we have
 \begin{align*}
	d k = {\lambda}\overline{\varsigma} kd&&\text{ and }&&d\Gamma_0=\gamma^2~\Gamma_0 d.
\end{align*}
\qed
\begin{theorem}
Let $(\pi,\mathcal{H})$ be an irreducible representation satisfying the condition of case 2. Let $\mathcal{M}$ be a closed subspace of $Ker(d^*d-1)$ such that 
\begin{align*}
d(\mathcal{M})\subseteq \mathcal{M},&& k(\mathcal{M})\subseteq\mathcal{M},&&\Gamma_0(\mathcal{M})\subseteq \mathcal{M},\\d^*(\mathcal{M})\subseteq \mathcal{M},&& k^*(\mathcal{M})\subseteq\mathcal{M},&&\Gamma_0^*(\mathcal{M})\subseteq \mathcal{M},
\end{align*}
where $\Gamma_0 =k^*D$.
Then, $\mathcal{H}_{\mathcal{M}}:=\{(a^*)^n\xi~:~n\in\mathbb{N},\xi\in\mathcal{M}\}$ is an invariant subspace for $\pi$.
If $\mathcal{N}$ be the another subspace of $Ker(d^*d-1)$ satisfying the above conditions such that $\mathcal{M}\perp\mathcal{N}$, then $\mathcal{H}_{\mathcal{M}}$ and $\mathcal{H}_{\mathcal{N}}$ are orthogonal.
\end{theorem}

\proof 
Now, we have the following relations:
\begin{align*}
	a(a^*)^n&=(a^*)^na+(1-(q^2)^n)(a^*)^{n-1}d^*d \text{ if }n\geq 1,
	\end{align*}
	\begin{align*}
	ba^*=q\gamma a^*b,&&
	da^*=q\overline{\gamma} a^*d,\\
	\overline{\lambda} d k = \overline{\varsigma} kd,&&
	\overline{\varsigma} b k = \overline{\lambda} kb,\\
	\Gamma_0=[ae-q\overline{\gamma}~db]=Dk^*,&&e=a^*\Gamma_0,\\d^*=(-\frac{1}{q}\gamma)~D^*kb,
\end{align*}
and $a^*$ commute with $k$, $D$ and $\Gamma_0$. We also have $Ker(d^*d-1)=Ker(a^*a)=Ker(a)$. Therefore for any $\xi\in\mathcal{M}$, we have
\begin{align*}
	a(a^*)^n\xi&=\left\{\begin{matrix}(1-(q^2)^n)(a^*)^{n-1}\xi  &\text{ if }n\geq 1,\\o&\text{ otherwise }\end{matrix}\right.\\
	b(a^*)^n\xi&=-q^{n+1}\gamma^{n-1} (a^*)^nk^*Dd^*\xi,\\
	d(a^*)^n\xi&=q^n\overline{\gamma}^n (a^*)^nd\xi,\\
	e(a^*)^n\xi&=(a^*)^{n+1}Dk^*\xi,\\
	\Gamma_0(a^*)^n\xi&=(a^*)^nDk^*\xi,\\
	k(a^*)^n\xi&=(a^*)^nk\xi,
\end{align*}
So, $\mathcal{H}_{\mathcal{M}}$ is invariant subspace under $\pi$. Therefore, $\mathcal{H}=\mathcal{H}_{\mathcal{M}}$.

For the 2nd part, take $\xi\in\mathcal{M}$, $\zeta\in\mathcal{N}$. Then for $n\in\mathbb{N}$, one has
\begin{align*}
	\langle (a^*)^n\xi,(a^*)^n\zeta\rangle &=\langle (a^*)^{n-1}\xi,(aa^*)(a^*)^{n-1}\zeta\rangle\\
	&={\prod_{s=1}^{n}(1-q^{2s})}\langle \xi,\zeta\rangle\\ &=0
\end{align*}
Next, for $n,m\in\mathbb{N}$ with $m>n$, we have
\begin{align*}
	\langle (a^*)^n\xi,(a^*)^m\zeta\rangle &=\langle (a^*)^{n-1}\xi,(aa^*)(a^*)^{m-1}\zeta\rangle\\
	&={\prod_{s=m-n+1}^{m}(1-q^{2s})}\langle \xi,(a^*)^{m-n}\zeta\rangle\\ &=0
\end{align*}
A similar calculation gives $\langle (a^*)^n\xi,(a^*)^m\zeta\rangle=0$ for $m<n$. Thus, $\mathcal{H}_\mathcal{M}$ and $\mathcal{H}_\mathcal{N}$ are orthogonal.

\qed

\begin{theorem}
		Let $(\pi,\mathcal{H})$ be an irreducible representation satisfying the condition of case 2. Then $\mathcal{H}$ is unitarily equivalent to  $\ell^2(\mathbb{N})\otimes \mathbb{K}_{{\lambda}\overline{\varsigma},\gamma^2,1}$ such that
		
\begin{tabular}{lllll}
$a\mapsto S\sqrt{1-q^{2N}}\otimes I,$	&&  $b\mapsto -q^{N+1}{\gamma}^{N+1}\otimes W^*Y$,&&\\
$d\mapsto q^{N}\overline{\gamma}^N\otimes  W,$&&$e\mapsto \sqrt{1-q^{2N}}S^*\otimes Y,$&&  \\
$k\mapsto I\otimes X$	&& $D\mapsto  I\otimes XY$, &&$\Gamma_0\mapsto I\otimes Y$ .
\end{tabular}
\end{theorem}
\proof 
Take $\mathcal{M}$ be a subspace of $Ker(d^*d-1)$ with conditions given in previous theorem such that $d|_\mathcal{M}, k|_\mathcal{M},~\Gamma_0|_\mathcal{M}$ is the image of generators of irreducible representation of noncommutative 3-torus. Then, $\mathcal{H}_\mathcal{M}$ is invariant subspace of $\mathcal{H}$ and hence by irreducibility of $\pi$, one has $\mathcal{H}_\mathcal{M}=\mathcal{H}$. Thus, the map \[U:e_i\otimes x\longmapsto \frac{1}{\sqrt{\prod_{s=1}^i(1-q^{2s})}} (a^*)^i(x)\] from $\ell^2(\mathbb{N})\otimes\mathcal{M}$ to $ \mathcal{H}$ extends to a unitary and gives the required unitary equivalence.

On $\mathcal{M}$, $k$, $\Gamma_0$ and $d$ are unitary; $k$ and $\Gamma_0$ commute; $ d k = {\lambda}\overline{\varsigma} kd$ and $d\Gamma_0=\gamma^2~\Gamma_0 d$. Then, $\mathcal{M}\cong \mathbb{K}_{{\lambda}\overline{\varsigma},\gamma^2,1}$ (for details see subsection~\ref{subsec:nctorus}) such that $d|_\mathcal{M}\mapsto W$, $\Gamma_0|_\mathcal{M}\mapsto Y $ and $k|_\mathcal{M}\mapsto X$.
Hence the theorem.
\qed

 \begin{remark}
 Therefore	$\pi:C(U_{q,\Theta}(3))\longrightarrow C(U_{q,\overline{\gamma}}(2))\otimes B(\mathbb{K}_{{\lambda}\overline{\varsigma}})$ is a representation such that $a\mapsto \alpha\otimes 1$, $b\mapsto (-q{\gamma}~\Gamma \beta^*)\otimes U_2^*$, $d\mapsto\beta\otimes U_2$, $e\mapsto \Gamma\alpha^*\otimes 1$,  $k\mapsto 1\otimes V_2$, $\Gamma_0\mapsto \Gamma\otimes 1$ and $D\mapsto \Gamma\otimes V_2$ where $U_2V_2={\lambda}\overline{\varsigma}V_2U_2$.
 \end{remark} 
 
\subsection{Case-3 : $g=0$, $d=0$, $Ker(h)=\{0\}$.} 
Analogous to the results in the previous case, here we have to the following results. The proofs are similiar.

\begin{theorem}
		Let $(\pi,\mathcal{H})$ be an irreducible representation satisfying the condition of case 3. Let $\mathcal{H}_0=Ker(h^*h-1)$. Then $(a|_{\mathcal{H}_0},h|_{\mathcal{H}_0},\Gamma_0|_{\mathcal{H}_0})$  are the representation of rotational algebra of $\mathbb{T}_{a_1,1,a_3}(3)$ on $\mathcal{H}_0$ where $a_1=\gamma \overline{\lambda}$, $a_3=\varsigma^2$ and $D=a\Gamma_0$.
\end{theorem}
	
\begin{theorem}
		Let $(\pi,\mathcal{H})$ be an representation satisfying the condition of case 3. Let $\mathcal{M}$ be the subspace of $Ker(h^*h-1)$ such that 
		\begin{align*}
			a(\mathcal{M})\subseteq \mathcal{M},&& h(\mathcal{M})\subseteq\mathcal{M},&&\Gamma_0(\mathcal{M})\subseteq \mathcal{M},\\
			a^*(\mathcal{M})\subseteq \mathcal{M},&& h^*(\mathcal{M})\subseteq\mathcal{M},&&\Gamma_0^*(\mathcal{M})\subseteq \mathcal{M},
		\end{align*}
		Then $\mathcal{H}_{\mathcal{M}}:=\{(e^*)^n\xi~:~n\in\mathbb{N},\xi\in\mathcal{M}\}$ is an invariant subspace for $\pi$.
		
		If $\mathcal{N}$ be the another subspace of $Ker(h^*h-1)$ satisfying the above conditions such that $\mathcal{M}\perp\mathcal{N}$, then $\mathcal{H}_{\mathcal{M}}$ and $\mathcal{H}_{\mathcal{N}}$ are orthogonal.
	\end{theorem}
	\begin{theorem}
	Let $(\pi,\mathcal{H})$ be an irreducible representation satisfying the condition of case 3. Then $\mathcal{H}$ is unitarily equivalent to  $  \ell^2(\mathbb{N})\otimes\mathbb{K}_{a_1,1,a_3}$ such that
	
	\begin{tabular}{lllll}
		$a\mapsto  I\otimes W,$	&&  $e\mapsto  S\sqrt{1-q^{2N}}\otimes   I$, && \\
		$f\mapsto  -q^{N+1}{\varsigma}^{N+1}\otimes X^*Y,$ && $h\mapsto  q^N\overline{\varsigma}^N\otimes X,$ &&  \\
		$k\mapsto  \sqrt{1-q^{2N}}S^*\otimes Y$	&& $D\mapsto I \otimes YW $ && $\Gamma_0=I\otimes Y$ 
	\end{tabular}
	\end{theorem}

\begin{remark}
	Therefore	$\pi:C(U_{q,\Theta}(3))\longrightarrow  B(\mathbb{K}_{\gamma\overline{\lambda}})\otimes C(U_{q,\overline{\varsigma}}(2))$ is a representation such that $a\mapsto U_1\otimes 1$, $e\mapsto 1\otimes \alpha $, $f\mapsto V_1^*\otimes (-q{\varsigma }~\Gamma \beta^*)$, $h\mapsto V_1\otimes \beta$, $k\mapsto 1\otimes \Gamma\alpha^*$, $\Gamma_0\mapsto 1\otimes \Gamma $ and $D\mapsto U_1\otimes \Gamma $.
\end{remark}

\subsection{Case-4 : $Ker(g)=\{0\}$, $c=0$}

\begin{theorem}
	Let $(\pi,\mathcal{H})$ be an irreducible representation satisfying the condition of case 4. Let $\mathcal{H}_0=Ker(a^*a)\cap Ker(kk^*)$. Then $(D|_{\mathcal{H}_0},\frac{1}{q}b|_{\mathcal{H}_0},\frac{1}{q}f|_{\mathcal{H}_0})$ are the representation of rotational algebra of $\mathbb{T}_{a_1,a_2,a_3}(3)$ on $\mathcal{H}_0$ where $a_1=\gamma^2\lambda\overline{\varsigma},$, $a_2=\overline{\gamma}\lambda\varsigma^2$ and $a_3=\varsigma\overline{\gamma}$. 
\end{theorem}
\proof Here we have $c=0$. So we have the following relations:
\begin{align}
	\begin{matrix}
		a b &= q\,\gamma\, b a,\\
		\overline{\lambda}\, a f &=\overline{\gamma}\, f a\\
		q~a^*b&=\overline{\gamma}ba^*\\
		\overline{\gamma}a^*f&=\overline{\lambda}fa^*
	\end{matrix}
	&& \begin{matrix}
		\overline{\varsigma}\, b k &= \overline{\lambda}\, k b\\
		f k &= q\,\overline{\varsigma}\, k f,\\
		\lambda~k^*b&=\varsigma~bk^*\\
		\varsigma~k^*f&=q~fk^*
	\end{matrix}\label{case4-th1-eq1}
\end{align}
Also we have $ak=ka$, $a^*k=ka^*$, $aa^*+bb^*=1$ and $f^*f+k^*k=1$.

Here $a$ and $k^*$ satisfy the relations for the generator of quantum disk algebra. If either $a^*a$ or $kk^*$ has point spectrum $1$, then either $Ker(a^*a-1)$ or $ Ker(kk^*-1)$ respectively will be invariant subspace. So by irreducibility, either $a$ or $k$ will be unitary. In any case $g=0$ and this is one of the previous cases. 

Let $\mathcal{H}_0=Ker(a^*a)\cap Ker(kk^*)\}$. Therefore we have $\mathcal{H}_0\neq\{0\}$. Therefore for any $\xi\in\mathcal{H}_0$, using above relations \eqref{case4-th1-eq1} we have $b\xi\in\mathcal{H}_0$ and $f\xi\in\mathcal{H}_0$. Also $b$ and $f$ are normal in this case. So using Fuglede-Putnam-Rosenblum theorem, we have $b^*\xi\in\mathcal{H}_0$ and $f^*\xi\in\mathcal{H}_0$. Also we have $aa^*\xi=(1-q^2)\xi$ and $k^*k\xi=(1-q^2)\xi$. Therefore $(\frac{1}{q}b)(\frac{1}{q}b)^*\xi=\xi$ and $(\frac{1}{q}f)^*(\frac{1}{q}f)\xi=\xi$.
Therefore we have $(\frac{1}{q}b)|_{\mathcal{H}_0}$, $(\frac{1}{q}f)|_{\mathcal{H}_0}$ and  $D|_{\mathcal{H}_0}$ are unitary. From the following relations; \begin{align*}
Db=\gamma^2\lambda\overline{\varsigma}~bD,&&Df= \overline{\gamma}\lambda\varsigma^2fD&&\text{ and }&&\overline{\varsigma}\, b f &= \overline{\gamma}\, f b
\end{align*} the statement follows.
\qed

\begin{theorem}
	Let $(\pi,\mathcal{H})$ be an irreducible representation satisfying the condition of case 4. Let $\mathcal{M}$ be the subspace of $Ker(a^*a)\cap Ker(kk^*)$ such that 
	\begin{align*}
		b(\mathcal{M})\subseteq \mathcal{M},&& f(\mathcal{M})\subseteq\mathcal{M},&&D(\mathcal{M})\subseteq \mathcal{M},\\
		b^*(\mathcal{M})\subseteq \mathcal{M},&& f^*(\mathcal{M})\subseteq\mathcal{M},&&D*(\mathcal{M})\subseteq \mathcal{M}.
	\end{align*}
	Then $\mathcal{H}_{\mathcal{M}}:=\{(a^*)^nk^m\xi~:~n,m\in\mathbb{N},\xi\in\mathcal{M}\}$ is an invariant subspace for $\pi$.
	
	If $\mathcal{N}$ be the another subspace of $Ker(a^*a)\cap Ker(kk^*)$ satisfying the above conditions such that $\mathcal{M}\perp\mathcal{N}$, then $\mathcal{H}_{\mathcal{M}}$ and $\mathcal{H}_{\mathcal{N}}$ are orthogonal.
\end{theorem}

\proof 
Now we have the following relations:
\begin{align*}
	ek^n&=k^ne+((q^{2n}-1)\overline{\varsigma}k^{n-1}h(\frac{1}{q}f)\text{ if }n\geq 1,\\
	dk^n&=({\lambda}\overline{\varsigma})^nk^nd+((q^2)^n-1){\lambda}^{n-2}\overline{\varsigma}^nk^{n-1}g(\frac{1}{q}f) \text{ if }n\geq 1,\\
	aa^{*n}&=q^{2n}a^{*n}a+(1-q^{2n})a^{*(n-1)} \text{ if }n\geq 1,
\end{align*}
\begin{align*}
	da^*=q\overline{\gamma} a^*d,&&
	ea^*= a^*e,\\
	ga^* = q\overline{\lambda} a^*g,&&
	gk = q{\lambda} kg,\\
	ha^* = \gamma\overline{\lambda} a^*h,&&hk=q\varsigma kh.
\end{align*}
 We also have $Ker(a^*a)=Ker(a)$ and $Ker(kk^*)=Ker(k^*)$. Therefore for any $\xi\in\mathcal{M}$, we can observe the followings:
 \begin{align*}
 	d\xi=(-1)\frac{1}{q}\overline{\gamma}b^*k^*D\xi=0,&&
 	e\xi=a^*k^*D\xi=0,\\
 	g\xi = (\overline{\varsigma}\overline{\lambda}) ~(\frac{1}{q}b^*)(\frac{1}{q}f^*)D\xi,&&
 	h\xi = (-\gamma\overline{\varsigma}\overline{\lambda})~ a^*(\frac{1}{q}f^*)D\xi,
 \end{align*}
\begin{align*}
	a(a^*)^nk^m\xi&=\left\{\begin{matrix}(1-q^{2n})(a^*)^{n-1}k^m\xi  &\text{ if }n\geq 1\text{ and }m\geq 0,\\0&\text{ if }n=0\end{matrix}\right.\\
	b(a^*)^nk^m\xi&=q^{n+1}\gamma^{n}(\varsigma \overline{\lambda})^m (a^*)^nk^m~(\frac{1}{q}b)\xi,\\
	d(a^*)^nk^m\xi&=\left\{\begin{matrix}q^n(q^{2m}-1)\overline{\gamma}^{n+1}\lambda^{m-2}\overline{\varsigma}^{m-1}(a^*)^nk^{m-1}~(\frac{1}{q}b^*)D\xi  &\text{ if }n\geq 0\text{ and }m\geq 1,\\0&\text{ if }m=0\end{matrix}\right.\\
	e(a^*)^nk^m\xi&=\left\{\begin{matrix}(1-q^{2m})(a^*)^{n+1}k^{m-1}~D\xi  &\text{ if }n\geq 0\text{ and }m\geq 1,\\0&\text{ if }m=0\end{matrix}\right.\\
	f(a^*)^nk^m\xi&=q^{m+1}\lambda^n\overline{\gamma}^n\overline{\varsigma}^m(a^*)^nk^{m}(\frac{1}{q}f)\xi,\\
	g(a^*)^nk^m\xi&=q^{n+m}\lambda^{m-n}(\overline{\varsigma}\overline{\lambda})(a^*)^nk^{m} ~(\frac{1}{q}b^*)(\frac{1}{q}f^*)D\xi,\\
	h(a^*)^nk^m\xi&=q^{m}\overline{\lambda}^n\gamma^n\varsigma^m(-\gamma\overline{\varsigma}\overline{\lambda})(a^*)^{n+1}k^{m}(\frac{1}{q}f^*)D\xi,\\
	k(a^*)^nk^m\xi&=(a^*)^nk^{m+1}\xi,
\end{align*}

So $\mathcal{H}_{\mathcal{M}}$ is invariant subspace under $\pi$. So $\mathcal{H}=\mathcal{H}_{\mathcal{M}}$.

For the 2nd part, take $\xi\in\mathcal{M}$, $\zeta\in\mathcal{N}$. Then for $n,m\in\mathbb{N}$, one has
\begin{align*}
	\langle (a^*)^nk^m\xi,(a^*)^nk^m\zeta\rangle &=\langle (a^*)^{n-1}k^m\xi,(aa^*)(a^*)^{n-1}k^m\zeta\rangle\\
	&={\prod_{s=1}^{n}(1-q^{2s})}\langle k^m\xi,k^m\zeta\rangle\\ 
	&=[{\prod_{s=1}^{n}(1-q^{2s})}][{\prod_{r=1}^{m}(1-q^{2r})}]\langle \xi,\zeta\rangle\\&=0
\end{align*}
Next, for $n,m\in\mathbb{N}$ with $m>n$, we have
\begin{align*}
	\langle (a^*)^n\xi,(a^*)^m\zeta\rangle &=\langle (a^*)^{n-1}\xi,(aa^*)(a^*)^{m-1}\zeta\rangle\\
	&={\prod_{s=m-n+1}^{m}(1-q^{2s})}\langle \xi,(a^*)^{m-n}\zeta\rangle\\ &={\prod_{s=m-n+1}^{m}(1-q^{2s})}\langle a^{m-n}\xi,\zeta\rangle=0
\end{align*}
A similar calculation gives $\langle (a^*)^n\xi,(a^*)^m\zeta\rangle=0$ for $m<n$. Similarly we have $\langle k^n\xi,k^m\zeta\rangle=0$ for all $m,n$. As $a^*$ and $a$ both commute with $k$ and $k^*$, we have $k^n\xi,k^m\zeta\in Ker(a)$ and $(a^*)^n\xi,(a^*)^m\zeta\in Ker(k^*)$. Therefore we get $\langle (a^*)^nk^m\xi,(a^*)^rk^s\zeta\rangle=0$. Thus $\mathcal{H}_\mathcal{M}$ and $\mathcal{H}_\mathcal{N}$ are orthogonal.

\qed

\begin{theorem}
	Let $(\pi,\mathcal{H})$ be an irreducible representation satisfying the condition of case 4. Then $\mathcal{H}$ is unitarily equivalent to  $\ell^2(\mathbb{N})\otimes\ell^2(\mathbb{N})\otimes \mathbb{K}_{a_1,a_2,a_3}$ such that
	\begin{align*}
		a \mapsto &\; I \otimes S\sqrt{1-q^{2N}}\otimes I, \\
		b \mapsto &\; (\varsigma\overline{\lambda})^N \otimes (-q^{N+1})\gamma^{N} \otimes X, \\
		d \mapsto &\; \gamma\overline{\lambda }\,S\sqrt{1-q^{2N}} \lambda^N\overline{\varsigma}^N\otimes q^{N}\overline{\gamma}^{N} \otimes WX^*, \\
		e \mapsto &\; S\sqrt{1-q^{2N}} \otimes \sqrt{1-q^{2N}}S^* \otimes W, \\
		f \mapsto &\; -q^{N+1}\overline{\varsigma}^N \otimes \lambda^N\overline{\gamma}^{N} \otimes Y, \\
		g \mapsto &\; \lambda\gamma\, q^N\lambda^N \otimes q^N\overline{\lambda}^{N} \otimes WX^*Y^*, \\
		h \mapsto &\; \varsigma\, q^{N}{\varsigma}^{N} \otimes \sqrt{1-q^{2N}}S^* \gamma^{N}\overline{\lambda}^{N} \otimes Y^*W, \\
		k \mapsto &\; \sqrt{1-q^{2N}}S^*\otimes I\otimes I, \\
		D \mapsto &\;I\otimes I\otimes W.
	\end{align*}
\end{theorem}
\proof 
Take $\mathcal{M}$ be a subspace of $Ker(a^*a)\cap Ker(kk^*)$ with conditions given  in previous theorem such that $(D|_{\mathcal{H}_0},\frac{1}{q}b|_{\mathcal{H}_0},\frac{1}{q}f|_{\mathcal{H}_0})$ is the representation of rotational algebra of $\mathbb{T}_{a_1,a_2,a_3}(3)$. Then $\mathcal{H}_\mathcal{M}$ is invariant subspace of $\mathcal{H}$ and hence by irreducibility of $\pi$, one has $\mathcal{H}_\mathcal{M}=\mathcal{H}$. The map \[U:e_m\otimes e_n\otimes  x\longmapsto \frac{1}{\sqrt{\prod_{s=1}^n(1-q^{2s})}\sqrt{\prod_{r=1}^m(1-q^{2r})}} (a^*)^nk^m(x)\] from $\ell^2(\mathbb{N})\otimes\ell^2(\mathbb{N})\otimes\mathcal{M}$ to $ \mathcal{H}$ now extends to a unitary and gives us the required unitary equivalence.

On $\mathcal{M}$, $D|_\mathcal{M}$, $\frac{1}{q}b|_\mathcal{M}$ and $\frac{1}{q}f|_\mathcal{M}$ are unitary such that \begin{align*}
	Db=\gamma^2\lambda\overline{\varsigma}~bD,&&Df= \overline{\gamma}\lambda\varsigma^2fD&&\text{ and }&&\overline{\varsigma}\, b f &= \overline{\gamma}\, f b
\end{align*}. Then $\mathcal{M}\cong \mathbb{K}_{a_1,a_2,a_3}$such that $D|_\mathcal{M}\mapsto W$, $\frac{1}{q}b|_\mathcal{M}\mapsto -X$ and $\frac{1}{q}f|_\mathcal{M}\mapsto -Y$ where $a_1=\gamma^2\lambda\overline{\varsigma},~a_3=\varsigma\overline{\gamma})$, and $a_2=a_1a_3^3$.
Therefore the statement follows.
\qed
\subsection{Case-5 : $Ker(c)=\{0\}$, $g=0$}

Analogous to the previous results, we have the following results in this case. The proofs are also similar.

\begin{theorem}
	Let $(\pi,\mathcal{H})$ be an irreducible representation satisfying the condition of case 5. Let $\mathcal{H}_0=Ker(a^*a)\cap Ker(kk^*)$. Then $(d|_{\mathcal{H}_0},h|_{\mathcal{H}_0},D|_{\mathcal{H}_0})$ are the representation of rotational algebra $\mathbb{T}_{a_1,a_2,a_3}(3)$ on $\mathcal{H}_0$ where $a_1=\gamma\overline{\varsigma},$, $a_2=\gamma^2\overline{\varsigma}\lambda$ and $a_3=\overline{\gamma}\varsigma^2\lambda $.
\end{theorem}
\begin{theorem}
	Let $(\pi,\mathcal{H})$ be an irreducible representation satisfying the condition of case 5. Let $\mathcal{M}$ be the subspace of $Ker(a^*a)\cap Ker(kk^*)$ such that 
	\begin{align*}
		d(\mathcal{M})\subseteq \mathcal{M},&& h(\mathcal{M})\subseteq\mathcal{M},&&D(\mathcal{M})\subseteq \mathcal{M},\\
		d^*(\mathcal{M})\subseteq \mathcal{M},&& h^*(\mathcal{M})\subseteq\mathcal{M},&&D^*(\mathcal{M})\subseteq \mathcal{M}.
	\end{align*}
	Then $\mathcal{H}_{\mathcal{M}}:=\{(a^*)^nk^m\xi~:~n,m\in\mathbb{N},\xi\in\mathcal{M}\}$ is an invariant subspace for $\pi$.
	
	If $\mathcal{N}$ be the another subspace of $Ker(a^*a)\cap Ker(kk^*)$ satisfying the above conditions such that $\mathcal{M}\perp\mathcal{N}$, then $\mathcal{H}_{\mathcal{M}}$ and $\mathcal{H}_{\mathcal{N}}$ are orthogonal.
\end{theorem}
\begin{theorem}
	Let $(\pi,\mathcal{H})$ be an irreducible representation satisfying the condition of case 5. Then $\mathcal{H}$ is unitarily equivalent to  $\ell^2(\mathbb{N})\otimes\ell^2(\mathbb{N})\otimes \mathbb{K}_{a_1,a_2,a_3}$ such that
	\begin{align*}
		a \mapsto &\; I\otimes S\sqrt{1-q^{2N}}\otimes I, \\
		b \mapsto &\; (\overline{\varsigma}\lambda\gamma )~S\sqrt{1-q^{2N}}(\overline{\lambda}\varsigma)^N \otimes (-q^{N+1}\gamma^{N} )\otimes W^*Y, \\
		c\mapsto &\lambda\varsigma~(-q^{N+1}\overline{\lambda}^N)\otimes (-q^{N+1}{\lambda}^N)\otimes W^*X^*Y\\
		d \mapsto &(\lambda \overline{\varsigma})^N\otimes q^N\overline{\gamma}^N\otimes W\\
		e \mapsto &\; S\sqrt{1-q^{2N}} \otimes \sqrt{1-q^{2N}}S^* \otimes Y, \\
		f \mapsto &\; \lambda\varsigma\overline{\gamma}~(-q^{N+1}\overline{\varsigma}^N )\otimes \sqrt{1-q^{2n}}S^*\lambda^N\overline{\gamma}^{N} \otimes X^*Y, \\
		h \mapsto &\; q^{N}{\varsigma}^{N} \otimes  \gamma^{N}\overline{\lambda}^{N} \otimes X, \\
		k \mapsto &\; \sqrt{1-q^{2N}}S^*\otimes I\otimes I, \\
		D \mapsto &\; I\otimes I\otimes Y.
	\end{align*}
\end{theorem}

\subsection{Case-6 : $Ker(c)=\{0\}=Ker(g)$}

	\begin{theorem}
		Let $(\pi,\mathcal{H})$ be an irreducible representation satisfying the condition of case 6. Let $\mathcal{H}_0=Ker(a^*a)\cap Ker(d^*d)\cap Ker(b^*b)$. Then $(D|_{\mathcal{H}_0},c|_{\mathcal{H}_0},g|_{\mathcal{H}_0})$ are the representation of rotational algebra of $\mathbb{T}_{a_1,a_2,a_3}(3)$ on $\mathcal{H}_0$ where $a_1=\gamma\lambda^2\varsigma$, $a_2=\overline{\gamma}\overline{\lambda}^2\overline{\varsigma}$ and $a_3=\overline{\lambda}^2$. 
	\end{theorem}
	,
	
	\proof 
	Here $a$ satisfies the relation for the generator of quantum disk algebra. If $a^*a$ has point spectrum $1$, then  $Ker(a^*a-1)$  will be invariant subspace. So by irreducibility, $a$ will be unitary. Then we have $g=0$ and this is one of the previous cases. 
	
	Let $\mathcal{H}_1=Ker(a^*a)$. Therefore we have $\mathcal{H}_1\neq\{0\}$. Therefore for any $\xi\in\mathcal{H}_1$,  we have $c\xi\in\mathcal{H}_1$ and $g\xi\in\mathcal{H}_1$. Also $c$ and $g$ are normal. So using Fuglede-Putnam-Rosenblum theorem, we have $c^*\xi\in\mathcal{H}_1$ and $g^*\xi\in\mathcal{H}_1$. Also we have $b\xi$, $b^*\xi$, $d\xi$ and $d^*\xi\in\mathcal{H}_0$. From $aa^*\xi=(1-q^2)\xi$, we have the following relations: \[(\frac{1}{q}b)(\frac{1}{q}b)^*\xi+(\frac{1}{q^2}c)(\frac{1}{q^2}c)^*\xi=\xi\text{ and }d^*d\xi+g^*g\xi=\xi.\]
	
	Here $(\frac{1}{q}b|_{\mathcal{H}_1})$ and $d|_{\mathcal{H}_1})$ satisfy the relation for the generator of quantum disk algebra. 
	
	Let $H_1'=Ker((\frac{1}{q}b|_{\mathcal{H}_1})^*(\frac{1}{q}b|_{\mathcal{H}_1})-1)$. If $(\frac{1}{q}b|_{\mathcal{H}_1})^*(\frac{1}{q}b|_{\mathcal{H}_1})$ has point spectrum $1$, then $\underset{n\geq 0}{\oplus}a^{*n}\mathcal{H}_1'$  will be invariant subspace. Then $c=0$ and this is one of the previous cases.
	
	Let $H_1''=Ker((d|_{\mathcal{H}_1})^*(d|_{\mathcal{H}_1})-1)$. If $(d|_{\mathcal{H}_1})^*(d|_{\mathcal{H}_1})$ has point spectrum $1$, then  $\underset{n\geq 0}{\oplus}a^{*n}\mathcal{H}_1''$  will be invariant subspace. Then $g=0$ and this is also one of the previous cases.
	
	The operators $(\frac{1}{q}b|_{\mathcal{H}_1})$ and $d|_{\mathcal{H}_1}$ also satisfy the relation for generator of quantum disc algebra and $\{(\frac{1}{q}b)(\frac{1}{q}b)^*,(\frac{1}{q}b)^*(\frac{1}{q}b)\}$ commute with $\{dd^*,d^*d\}$.
	
	Consider $\mathcal{H}_0=Ker(a^*a)\cap Ker(d^*d)\cap Ker(b^*b)$. Here $\mathcal{H}_0\neq\{0\}$. Also we have $c\mathcal{H}_0\subseteq \mathcal{H}_0$ and  $g\mathcal{H}_0\subseteq \mathcal{H}_0$. As $c,g$ are normal, we have  $c^*\mathcal{H}_0\subseteq \mathcal{H}_0$ and  $g^*\mathcal{H}_0\subseteq \mathcal{H}_0$ using Fuglede-Putnam-Rosenblum theorem.
	
	Therefore we have  $(\frac{1}{q^2}c)|_{\mathcal{H}_0}$, $g|_{\mathcal{H}_0}$ and  $D|_{\mathcal{H}_0}$ are unitary. From the following relations; 
\begin{align*}
Dc=\gamma\lambda^2\varsigma\,cD,&&Dg= \overline{\gamma}\overline{\lambda}^2\overline{\varsigma}\,gD&&\text{ and }&& cg&= \overline{\lambda}^2\, gc
\end{align*} the statement follows. \qed

	\begin{theorem}
		Let $(\pi,\mathcal{H})$ be an irreducible representation satisfying the condition of case 6. Let $\mathcal{M}$ be the subspace of $Ker(a^*a)\cap Ker(d^*d)\cap Ker(b^*b)$ such that 
		\begin{align*}
			c(\mathcal{M})\subseteq \mathcal{M},&& g(\mathcal{M})\subseteq\mathcal{M},&&D(\mathcal{M})\subseteq \mathcal{M},\\
			c^*(\mathcal{M})\subseteq \mathcal{M},&& g^*(\mathcal{M})\subseteq\mathcal{M},&&D^*(\mathcal{M})\subseteq \mathcal{M}.
		\end{align*}
		Then $\mathcal{H}_{\mathcal{M}}:=\{(a^*)^n(d^*)^m(b^*)^p\xi~:~n,m,p\in\mathbb{N},\xi\in\mathcal{M}\}$ is an invariant subspace for $\pi$.
		
		If $\mathcal{N}$ be the another subspace of  $Ker(a^*a)\cap Ker(d^*d)\cap Ker(b^*b)$  satisfying the above conditions such that $\mathcal{M}\perp\mathcal{N}$, then $\mathcal{H}_{\mathcal{M}}$ and $\mathcal{H}_{\mathcal{N}}$ are orthogonal.
	\end{theorem}
	\proof

	Now we have the following relations:
	\begin{align*}
		aa^{*n}&=q^{2n}a^{*n}a+(1-q^{2n})a^{*(n-1)} \text{ if }n\geq 1,\\
		bb^{*p}&=b^{*p}b+(1-q^{2p})b^{*(p-1)}(\frac{1}{q}c)(\frac{1}{q}c^*) \text{ if }p\geq 1,\\
		dd^{*m}&=d^{*m}d+(1-q^{2m})d^{*(m-1)}g^*g \text{ if }m\geq 1,\\
	\end{align*}
	We also have $Ker(a^*a)=Ker(a)$, $Ker(d^*d)=Ker(d)$ and $Ker(b^*b)=Ker(b)$. Therefore for any $\xi\in\mathcal{M}$, we can observe the followings: 
	
	\begin{align*}
		a(a^{*n}d^{*m}(\frac{1}{q^p}b^{*p}))\xi&=\left\{\begin{matrix}(1-q^{2n})(a^*)^{n-1}d^{*m}(\frac{1}{q^p}b^{*p})\xi  &\text{ if }n\geq 1\text{ and }m\geq 0,\\0&\text{ if }n=0 \end{matrix}\right.\\
		b(a^{*n}d^{*m}(\frac{1}{q^p}b^{*p}))\xi&=\left\{\begin{matrix}q^{n+1}\gamma^n(1-q^{2p})(a^{*n})d^{*m}(\frac{1}{q^{p-1}}b^{*{p-1}})(\frac{1}{q^2}c)(\frac{1}{q^2}c^*)\xi  &\text{ if }p\geq 1,\\0&\text{ if }p=0 \end{matrix}\right.\\
		c(a^{*n}d^{*m}(\frac{1}{q^p}b^{*p}))\xi&=q^{n+1}\lambda^n(\gamma\lambda)^m(q^{1+p}\varsigma^p)a^{*n}d^{*m}(\frac{1}{q^p}b^{*p})(\frac{1}{q^2}c)\xi,\\
		d(a^{*n}d^{*m}(\frac{1}{q^p}b^{*p}))\xi&=\left\{\begin{matrix}(q\overline{\gamma})^n(1-q^{2m})(a^{*n})d^{*(m-1)}(\frac{1}{q^p}b^{*p})g^*g\xi  &\text{ if }m\geq 1,\\0&\text{ if }m=0 \end{matrix}\right.\\
		k(a^{*n}d^{*m}(\frac{1}{q^p}b^{*p}))\xi&=(\lambda\overline{\varsigma})^m(\varsigma\overline{\lambda})^pa^{*n}d^{*m}(\frac{1}{q^p}b^{*p})k\xi \\
		h(a^{*n}d^{*m}(\frac{1}{q^p}b^{*p}))\xi&=(\gamma\overline{\lambda})^n(\gamma\overline{\varsigma})^ma^{*n}d^{*m}h(\frac{1}{q^p}b^{*p})\xi 
	\end{align*}

	So $\mathcal{H}_{\mathcal{M}}$ is invariant subspace under $\pi$. So $\mathcal{H}=\mathcal{H}_{\mathcal{M}}$.
	
	For the 2nd part, take $\xi\in\mathcal{M}$, $\zeta\in\mathcal{N}$. Then for $n,m,p\in\mathbb{N}$, one has
	\begin{align*}
		\langle (a^*)^n(d^*)^m(b^*)^p\xi,(a^*)^n(d^*)^m(b^*)^p\zeta\rangle &=\langle (a^*)^{n-1}(d^*)^m(b^*)^p\xi,(aa^*)(a^*)^{n-1}(d^*)^m(b^*)^p\zeta\rangle\\
		&={\prod_{s=1}^{n}(1-q^{2s})}\langle (d^*)^m(b^*)^p\xi,(d^*)^m(b^*)^p\zeta\rangle\\ &={\prod_{s=1}^{n}(1-q^{2s})}{\prod_{r=1}^{m}(1-q^{2s})}\langle (b^*)^p\xi,(b^*)^p\zeta\rangle\\ &={\prod_{s=1}^{n}(1-q^{2s})}{\prod_{r=1}^{m}(1-q^{2s})}{\prod_{y=1}^{p}(1-q^{2s})}\langle \xi,\zeta\rangle=0
	\end{align*}
	Next, for $n_1,n_2,m_1,m_2,p_1,p_2\in\mathbb{N}$ with $n_1>n_2$, we have $(d^*)^{m_1}(b^*)^{p_1}\xi\in Ker(a)$, and $(d^*)^{m_1}(b^*)^{p_1}\zeta\in Ker(a)$. So we can observe that
	\begin{align*}
		\langle (a^*)^{n_1}(d^*)^{m_1}(b^*)^{p_1}\xi,(a^*)^{n_2}(d^*)^{m_2}(b^*)^{p_2}\zeta\rangle &=\langle (aa^*)(a^*)^{n_1-1}(d^*)^{m_1}(b^*)^{p_1}\xi,(a^*)^{n_2-1}(d^*)^{m_2}(b^*)^{p_2}\zeta\rangle\\
		&={\prod_{s=n_1-n_2+1}^{n_1}(1-q^{2s})}\langle (a^*)^{n_1-n_2}(d^*)^{m_1}(b^*)^{p_1}\xi,(d^*)^{m_2}(b^*)^{p_2}\zeta\rangle\\ &=0. 
	\end{align*}
	
	If $n_1=n_2=n$, and $m_1>m_2$ we have $(b^*)^{p_1}\xi\in Ker(d)$, and $(b^*)^{p_1}\zeta\in Ker(d)$. So we can observe that
	\begin{align*}
		\langle (a^*)^{n}(d^*)^{m_1}(b^*)^{p_1}\xi,(a^*)^{n}(d^*)^{m_2}(b^*)^{p_2}\zeta\rangle &=\langle (aa^*)(a^*)^{n_1-1}(d^*)^{m_1}(b^*)^{p_1}\xi,(a^*)^{n_2-1}(d^*)^{m_2}(b^*)^{p_2}\zeta\rangle\\
		&={\prod_{s=1}^{n}(1-q^{2s})}{\prod_{r=m_1-m_2+1}^{m_1}(1-q^{2r})}\langle (d^*)^{m_1-m_2}(b^*)^{p_1}\xi,(b^*)^{p_2}\zeta\rangle\\ &=0. 
	\end{align*}

	If $n_1=n_2=n$, $m_1=m_2=m$, and $p_1>p_2$, we can observe that
	\begin{align*}
		\langle (a^*)^{n}(d^*)^{m}(b^*)^{p_1}\xi,(a^*)^{n}(d^*)^{m}(b^*)^{p_2}\zeta\rangle &=\langle (aa^*)(a^*)^{n_1-1}(d^*)^{m_1}(b^*)^{p_1}\xi,(a^*)^{n_2-1}(d^*)^{m_2}(b^*)^{p_2}\zeta\rangle\\
		&={\prod_{s=1}^{n}(1-q^{2s})}{\prod_{r=1}^{m}(1-q^{2r})}{\prod_{y=p_1-p_2+1}^{p_1}(1-q^{2y})}\langle (b^*)^{p_1-p_2}\xi,\zeta\rangle\\ &=0. 
	\end{align*}
	
	A similar calculation gives $\langle (a^*)^n\xi,(a^*)^m\zeta\rangle=0$ for $m<n$. Thus $\mathcal{H}_\mathcal{M}$ and $\mathcal{H}_\mathcal{N}$ are orthogonal.
	
\qed
\begin{theorem}
Let $(\pi,\mathcal{H})$ be an irreducible representation satisfying the condition of case 6. Then $\mathcal{H}$ is unitarily equivalent to  $\ell^2(\mathbb{N})\otimes\ell^2(\mathbb{N})\otimes\ell^2(\mathbb{N})\otimes \mathbb{K}_{a_1,a_2,a_3}$ where $a_1=\gamma\lambda^2\varsigma$, $a_2=\overline{\gamma}\overline{\lambda}^2\overline{\varsigma}$ and $a_3=\overline{\lambda}^2$ such that
\begin{align*}
a \mapsto &\; I\otimes S\sqrt{1-q^{2N}}\otimes I\otimes I, \\
b \mapsto &\; \gamma^{2N}\otimes q^{N+1}\gamma^N\otimes ~S\sqrt{1-q^{2N}}\otimes I, \\
c\mapsto &\;(\lambda\gamma)^N\otimes q^{N+1}\lambda^N\otimes q^{N+1}\varsigma^N\otimes X,\\
d \mapsto &S\sqrt{1-q^{2N}}\otimes q^N\overline{\gamma}^N\otimes I\otimes I,\\
g \mapsto &\; (q\overline{\varsigma})^N \otimes (q\overline{\lambda})^N \otimes (\overline{\lambda}\overline{\gamma})^N\otimes Y, \\
D \mapsto &\; (\gamma^2\lambda\overline{\varsigma})^N\otimes1\otimes (\overline{\gamma}^2\overline{\lambda}\varsigma)^N \otimes W\\
e\mapsto &\; - (q\gamma)^N\otimes q\overline{\lambda} \otimes (q\overline{\gamma})^N\otimes  Y^*X^*W\\
\,&\;- \gamma^{2N} S\sqrt{1-q^{2N}}\otimes \gamma\sqrt{1-q^{2N}}S^*\otimes S\sqrt{1-q^{2N}}\otimes I\\
f\mapsto &\;  (q\lambda)^N\otimes q\overline{\gamma}\varsigma(\overline{\gamma}\lambda)^N \otimes (\gamma\lambda)^N \sqrt{1-q^{2N}}S^*(\overline{\gamma}^2\overline{\lambda}\varsigma)^N\otimes Y^*W \\
\,&\;-(\gamma\lambda)^N S\sqrt{1-q^{2N}}\otimes q\lambda\sqrt{1-q^{2N}}S^*(\overline{\gamma}\lambda)^N\otimes (q\varsigma)^N\otimes X\\
h\mapsto &\; \sqrt{1-q^{2N}}S^*(\gamma\overline{\varsigma})^N\otimes\overline{\lambda}\overline{\varsigma} (\gamma\overline{\lambda})^N\otimes (q\overline{\gamma}^2\overline{\lambda})^N \otimes X^*W \\
\,&\;-(q\gamma^2\overline{\varsigma})^N\otimes \gamma\sqrt{1-q^{2N}}S^*(\gamma\overline{\lambda})^N\otimes S\sqrt{1-q^{2N}}(\overline{\gamma}\overline{\lambda})^N\otimes Y,\\
k\mapsto&\; -(q\gamma\lambda\overline{\varsigma})^N\otimes q\lambda\sqrt{1-q^{2N}}S^* \otimes  (q\overline{\gamma}\overline{\lambda}\varsigma)^N \otimes XY\\
\,&\; -\sqrt{1-q^{2N}}S^*(\lambda\overline{\varsigma})^N\otimes \overline{\gamma}I \otimes \sqrt{1-q^{2N}}S^*(\overline{\gamma}^2\overline{\lambda}\varsigma)^N \otimes W,
\end{align*}
\end{theorem}

\proof 
Take $\mathcal{M}$ be a subspace of $Ker(a^*a)\cap Ker(d^*d)\cap Ker(b^*b)$ with conditions given  in previous theorem such that $(D|_{\mathcal{H}_0},\frac{1}{q^2}c|_{\mathcal{H}_0},g|_{\mathcal{H}_0})$ is the representation of Non-commutative $3$-torus $\mathbb{T}_{a_1,a_2,a_3}(3)$. Then $\mathcal{H}_\mathcal{M}$ is invariant subspace of $\mathcal{H}$ and hence by irreducibility of $\pi$, one has $\mathcal{H}_\mathcal{M}=\mathcal{H}$. The map \[U:e_m\otimes e_n\otimes e_p\otimes  x\longmapsto \frac{1}{q^p\sqrt{\prod_{s=1}^n(1-q^{2s})}\sqrt{\prod_{r=1}^m(1-q^{2r})}\sqrt{\prod_{t=1}^p(1-q^{2t})}} a^{*n}d^{*m}b^{*p}(x)\] from $\ell^2(\mathbb{N})\otimes \ell^2(\mathbb{N})\otimes\ell^2(\mathbb{N})\otimes\mathcal{M}$ to $ \mathcal{H}$ extends to a unitary and gives the required unitary equivalence.

On $\mathcal{M}$, $D|_\mathcal{M}$, $\frac{1}{q^2}c|_\mathcal{M}$ and $g|_\mathcal{M}$ are unitary such that \begin{align*}
	Dc=\gamma\lambda^2\varsigma\,cD,&&Dg= \overline{\gamma}\overline{\lambda}^2\overline{\varsigma}\,gD&&\text{ and }&& cg&= \overline{\lambda}^2\, gc.
\end{align*} Then $\mathcal{M}\cong \mathbb{K}_{a_1,a_2,a_3}$such that $D|_\mathcal{M}\mapsto W$, $\frac{1}{q}c|_\mathcal{M}\mapsto X$ and $g|_\mathcal{M}\mapsto Y$  where $a_1=\gamma\lambda^2\varsigma$, $a_2=\overline{\gamma}\overline{\lambda}^2\overline{\varsigma}$ and $a_3=\overline{\lambda}^2$.
Therefore the statement follows.\\

Similarly, one can find the irreducible representations for $q>1$. The case $q=1$ is addressed in a separate article.
\section*{Acknowledgments}
The first author would like to thank the Indian Institute of Technology Bombay for supporting this work through a Post Doc position.

\appendix

\section{Some facts}

\subsection{Non-commutative $n$-torus $\mathbb{T}_\Theta(n)$:}\label{subsec:nctorus}
The $n$-dimensional noncommutative torus is defined in \cite{phillips2006every}, a universal $C^*$-algebra generated by $n$-unitaries $U_1,~U_2,\cdots,U_n$ such that
\begin{align*}
	U_iU_j=\theta_{j,i} U_jU_i.
\end{align*}
The $n$-dimensional noncommutative torus $\mathbb{T}_\Theta(n)$ is simple if and only if $\Theta$ is
nondegenerate.

For $n=2$, $\mathbb{T}_\mu(2)$ denotes the noncommutative 2-torus which is generated by two unitaries $U,V$ such that $UV=\mu VU$ and $\mathbb{K}_\mu$ denotes the corresponding Hilbert space of an irreducible representation of noncommutative $2$-torus $\mathbb{T}_\mu(2)$.

For $n=3$, $\mathbb{T}_{\gamma,\lambda,\varsigma}(3)$ denotes noncommutative $3$-torus which is generated by three unitaries $W,X,Y$ such that 
\[ 
WX=\gamma XW,~~WY=\lambda YW,~~XY=\varsigma YX
\]
and $\mathbb{K}_{\gamma,\lambda,\varsigma}$  denotes the corresponding Hilbert space of an irreducible representation of noncommutative $3$-torus $\mathbb{T}_{\gamma,\lambda,\varsigma}(3)$.

\subsection{The $C^*$-algebra $C(U_{q,\mathbbm{1}}(n))$:}\label{subsec: uq1}

Let $0<q<1$ and $\theta_{i,j}=1$, We recall here the complete classsification of its irreducible representation from \cite{koelink91}.
The $C^*$-algebra $C(U_{q,\mathbbm{1}}(n))$ has the following representations:
\begin{itemize}
	\item For $\mu=(\lambda_1,\lambda_2,\cdots \lambda_n)$, let $\chi_{\mu }(V_{i,j})=\left\{\begin{matrix}
		\lambda_1\delta_{i,j}&\text{ if }i=1\\\frac{\lambda_i}{\lambda_{i-1}}\delta_{i,j}&\text{ if }i>1
	\end{matrix}\right.$ and $\chi(\mathscr{D}^{-1})=\frac{1}{\lambda_n}$. 
	\item $\psi_{s_k}(V_{i,j})=\left\{\begin{matrix}
		S\sqrt{1-q^N}&\text{ if }&i=j=k\\
		\sqrt{1-q^N}S^*&\text{ if }&i=j=k+1\\
		-q^{N+1}&\text{ if }&i=k,~j=k+1\\
		q^{N}&\text{ if }&i=k+1,~j=k\\
		I&\text{ if }&i=j\neq k,~k+1\\
		\delta_{i,j}&\text{ otherwise }
	\end{matrix}\right.$  and $\psi_{s_k}(\mathscr{D}^{-1})=1$.
	\item For any two representation $\pi_1$, $\pi_2$ we have $\pi_1*\pi_2=(\pi_1\otimes \pi_2)\Delta$.
	
	\item $\psi_{s_{[b,a]}}=\left\{\begin{matrix}
		\psi_{s_{b-1}}*\psi_{s_{b-1}}*\cdots *\psi_{s_{a}}&\text{ if }b>a\\\text{ it is omitted }&\text{ if }b=a
	\end{matrix}\right.$
	\item Let $r=(r_1,r_2,\cdots r_{n-1})$, define $\psi_{r}= \psi_{s_{[2,r_1]}}*\psi_{s_{[3,r_2]}}*\cdots *\psi_{s_{[n,r_{n-1}]}}$, where $1\leq r_i\leq i+1$
	\item Therefore all  inequivalent irreducible representations are of the form $\chi_\mu\ast\psi_r$, where $r=(r_1,r_2,\cdots r_{n-1})$ and $\mu=(\lambda_1,\lambda_2,\cdots \lambda_n)$ such that $1\leq r_i\leq i+1$ and $|\lambda_i|=1$.
\end{itemize}

Let $\varphi(V_{i,j})=\left\{\begin{matrix}
	S\otimes I^{\otimes(n-1)}&\text{ if }i=j=1\\I^{\otimes(i-2)}\otimes S^*\otimes S\otimes I^{\otimes(n-i)}&\text{ if }j=i>1\\
	\delta_{i,j}&\text{ otherwise }
\end{matrix}\right.$ and $\varphi(\mathscr{D}^{-1})=I^{\otimes(n-1)}\otimes S^*$. 

and $r_0=(1,1,\cdots,1)$, then $\varphi*\psi_{r_0}$ gives a faithful representation of $C(U_{q,\mathbbm{1}}(n))$ 

\subsection{The $C^*$-algebra $C(U_{1,\Theta}(n))$:}
We recall here the $C^*$ algebra from \cite{convio02} which is given as follows
\begin{align}
	V_{i,j}V_{k,l}&=\theta_{i,k}\theta_{l,j}V_{k,l}V_{i,j}&\forall~~i,j,k,l,\\
	V_{i,j}V_{k,l}^*&=\theta_{k,i}\theta_{j,l}V_{k,l}^*V_{i,j}&\forall~~i,j,k,l,\\
	\sum_{k=1}^n V^*_{k,i}V_{k,j}&=\delta_{i,j},&\forall~~i,j,\\
	\sum_{k=1}^n V_{i,k}V^*_{j,k}&=\delta_{i,j},&\forall~~i,j.
\end{align}
Then $V_{i,j}$'s are normal operators. Also $ V_{i,j}^*V_{i,j}$ commute with all generators. For irreducible representations, $V_{i,j}^*V_{i,j}=r_{i,j}^2I$ for all $i,j$. Let $V_{i,j}=r_{i,j}U_{i,j}$ where $U_{i,j}$ is unitary. Then it satisfy the following relations:
\begin{align}
	U_{i,j}U_{k,l}&=\theta_{i,k}\theta_{l,j}U_{k,l}U_{i,j}&\forall~~i,j,k,l,\\
	U_{i,j}U_{k,l}^*&=\theta_{k,i}\theta_{j,l}U_{k,l}^*U_{i,j}&\forall~~i,j,k,l,\\
	\sum_{k=1}^n r_{k,i}r_{k,j}U^*_{k,i}U_{k,j}&=\delta_{i,j},&\forall~~i,j,\\
	\sum_{k=1}^n r_{i,k}r_{j,k}U_{i,k}U^*_{j,k}&=\delta_{i,j},&\forall~~i,j.
\end{align}

\subsection{The $C^*$-algebra $C(U_{q,\Theta}(2))$:}
For $n=2$ and $\Theta=\begin{bmatrix} 1&\overline{\mu}\\\mu&1\end{bmatrix}$, we denote $C(U_{q,\Theta}(2))$ by $C(U_{q\mu}(2))$. From \cite{zhangzhao02}, we have $C(U_{q\mu}(2))$ is generated by three generators $a,b,\Gamma$ satisfying the following relations:
\begin{align}
      ba=q\mu ab\,,\qquad a^*b=q\mu ba^*\,,\qquad bb^*=b^*b,\qquad \quad aa^*+bb^*=1\,,\nonumber\\
     a^*a+q^2b^*b=1,\quad a\Gamma=\Gamma a,\quad b\Gamma=\mu^2\Gamma b,\quad \Gamma\Gamma^*=\Gamma^*\Gamma=1. 
\end{align}
Then from \cite{guinsaurabh21}, for $0<q<1$ and $q>1$ case, we have the following faithful representations   $\Upsilon_{\mu}:C(U_{q\mu}(2))\longrightarrow B(\ell^2(\mathbb{N})\otimes\ell^2(\mathbb{Z})\otimes\ell^2(\mathbb{Z}))$ such that
\begin{align}
\Upsilon_{\mu}(a) = \sqrt{1-q^{2N}}\,S^*\otimes 1\otimes 1\,,\quad\Upsilon_{\mu}(b)=q^N\mu^n \otimes S^* \otimes 1,\qquad
\Upsilon_{\mu}(D)= 1\otimes \overline{\mu}^{2N}\otimes S^*,\\\nonumber
\Upsilon_{\mu}(a) = S\sqrt{1-q^{-2N}}\otimes 1\otimes 1\,,\quad\Upsilon_{\mu}(b)=q^{-n-1}\mu^{-N} \otimes S^* \otimes 1,\quad
\Upsilon_{\mu}(D)= 1\otimes \overline{\mu}^{2N}\otimes S^*.
\end{align}
and for $q=1$, we have representation $\Upsilon_{\mu}:C(U_{q\mu}(2))\longrightarrow B(\ell^2(\mathbb{Z})\otimes\ell^2(\mathbb{Z})\otimes\ell^2(\mathbb{Z}))$ such that
\begin{align}
\Upsilon_{\mu,t}(a) = \sqrt{1-t^2}\,S^*\otimes 1\otimes 1\,,\qquad\Upsilon_{\mu}(b)=t\mu^N \otimes S \otimes 1,\quad
\Upsilon_{\mu}(D)= 1\otimes \mu^{2N}\otimes S^*.
\end{align}

\subsection{Quantum Disk Algebra $C(D^2_q)$:}\label{subsec:disk}
Let $0<q<1$ and $C(D^2_q)$ denotes the universal $C^*$ algebra generated by $a$ such that it satisfies the following equation
\begin{align*}
    q^2a^*a-aa^*=q^2-1.
\end{align*}
Let $\pi$ be any admissible representation. Then either $\pi(a)$ is unitary or one has 
\begin{align*}
    \sigma(\pi(a^*a))=\{0,1-q^2,1-q^4,\cdots\}\cup\{1\}&&\sigma(\pi(aa^*))=\{1-q^2,1-q^4,\cdots\}\cup\{1\}
\end{align*}

and all inequivalent irreducible representation of $C(D^2_q)$ is given by 
\begin{align*}
\rho: a\mapsto S\sqrt{1-q^{2n}} \text{ on } \ell^2{(\mathbb{N})},\qquad \sigma_t: a\mapsto t \text{ on } \mathbb{C}, \, t\in \mathbb{S}.
\end{align*}
Let $(\pi,\mathcal{H})$ be any representation of the Quantum Disk Algebra such that $1$ is not a point spectrum of $\pi(aa^*)$ and $\pi(a^*a)$. Therefore $\mathcal{H}\equiv\ell^2(\mathbb{N})\otimes \mathcal{H}_0\oplus\mathcal{H}_1$ such that $\mathcal{H}_0=ker(\pi(a^*a))$ and $e_i\otimes h\equiv \frac{1}{\sqrt{\prod_{s=1}^i(1-q^{2s})}} \pi(a^*)^i h$ for all $h\in\mathcal{H}_0$. So Here $\pi(a)\equiv S\sqrt{1-q^{2n}}\otimes Id\oplus W$ and $\pi(a^*)\equiv \sqrt{1-q^{2n}}S^*\otimes Id\oplus W$, where $W$ is unitary.

\section{FRT construction}\label{sec:FRT}

In this construction, $R$-matrix $R$ can be realised as a $n\times n$ block matrix where each block is also a matrix of order n, i.e. $R=\sum_{k,l,i,j}R_{kl,ij}(E_{k,i}\otimes E_{l,j}) \in M_{n^2}(A(R))$. Consider $V=\sum_{i,j}V_{i,j}E_{i,j}\in M_n(A(R))$. Also we have $V_2=I\otimes V=\sum_{p,r,s} V_{r,s}E_{p,p}\otimes E_{r,s}$ and $V_1=V\otimes I=\sum_{x,y,z} V_{x,y}(E_{x,y}\otimes E_{z,z})$.

Therefore generating commutation relation of the FRT construction is given by $RV_1V_2=V_2V_1R$. i.e.,
\begin{align*}
	RV_1V_2&=\sum_{k,l,i,j}\sum_{p,r,s}\sum_{x,y,z} R_{kl,ij}V_{x,y}V_{r,s}(E_{k,i}\otimes E_{l,j})(E_{x,y}\otimes E_{z,z})( E_{p,p}\otimes E_{r,s})\\
	&=\sum_{k,l,i,j}\sum_{p,r,s}\sum_{x,y,z}\delta_{i,x}\delta_{y,p}\delta_{j,z}\delta_{z,r} R_{kl,ij}V_{x,y}V_{r,s}(E_{k,p}\otimes E_{l,s})\\
	&=\sum_{k,l,i,j}\sum_{p,s} R_{kl,ij}V_{i,p}V_{j,s}(E_{k,p}\otimes E_{l,s})\\
	&=\sum_{i,j,r,s}[\sum_{k,l} R_{ji,kl}V_{k,r}V_{l,s}](E_{j,r}\otimes E_{i,s})
\end{align*}

\begin{align*}
	V_2V_1R&=\sum_{k,l,i,j}\sum_{p,r,s}\sum_{x,y,z} R_{kl,ij}V_{r,s}V_{x,y}( E_{p,p}\otimes E_{r,s})(E_{x,y}\otimes E_{z,z})(E_{k,i}\otimes E_{l,j})\\
	&=\sum_{k,l,i,j}\sum_{p,r,s}\sum_{x,y,z} \delta_{p,x}\delta_{y,k}\delta_{s,z}\delta_{z,l}R_{kl,ij}V_{r,s}V_{x,y}(E_{p,i}\otimes E_{r,j})\\
	&=\sum_{p,r,i,j}[\sum_{k,l} R_{kl,ij}V_{r,l}V_{p,k}](E_{p,i}\otimes E_{r,j})\\
	&=\sum_{r,s,i,j}[\sum_{k,l} R_{lk,rs}V_{i,k}V_{j,l}](E_{j,r}\otimes E_{i,s})\\
\end{align*}
Therefore we have $\sum_{k,l} R_{lk,rs}V_{i,k}V_{j,l}=\sum_{k,l} R_{ji,kl}V_{k,r}V_{l,s}$ for all $i,j,r,s$.
\begin{theorem}\label{Bialgebra}(\cite{bookklisch97}, chapter 1, Proposition 8) Let $\mathcal{S}$ be a subset of an algebra $\mathcal{A}$ which generates $\mathcal{A}$ as
	an algebra. Let $\Delta:A\rightarrow A\otimes A$ and $\epsilon:A\rightarrow\mathbb{C}$ be homomorphisms and $S:A\rightarrow A$ be anti-homomorphism of the corresponding algebras. If the coassociativity condition and the counit condition (and the antipode condition ) are satisfied for elements in $\mathcal{S}$ , then they are valid on the whole of $\mathcal{A}$ and hence $\mathcal{A}$ is a bialgebra (resp. a Hopf algebra).
\end{theorem}
\begin{theorem}\label{FRT}(\cite{bookklisch97}, chapter 9, Proposition 1)
	There is a unique bialgebra structure on the algebra $A(R)$ such that
	\begin{align}
		\Delta(V_{i,j})=\sum_k V_{i,k}\otimes V_{k,j}&&
		\epsilon(V_{i,j})=\delta_{i,j}.
	\end{align}
\end{theorem}
\section{Commutation relations among generators and its $^*$:}\label{sec:star}
If we consider $V=\sum_{i,j}V_{i,j}E_{i,j}\in M_n(\mathbb{C}[U_{q,\Theta}(n)])$, $V_2=I\otimes V=\sum_{i,r,s} V_{r,s}E_{i,i}\otimes E_{r,s}$ and $V_1=V\otimes I=\sum_{x,y,z} V_{x,y}(E_{x,y}\otimes E_{z,z})$, then we have $V^*=\sum_{i,j}V_{j,i}^*E_{i,j}$,  $(V^*)_2=I\otimes V^*=\sum_{p,r,s} V_{s,r}^*(E_{p,p}\otimes E_{r,s})$ such that $V_2(V^*)_2=(V^*)_2V_2=I$. The generating commutation relation of the FRT construction is given by $RV_1V_2=V_2V_1R$.

Multiplying both side by $(V^*)_2$ we have $(V^*)_2RV_1=V_1R(V^*)_2$.
Therefore we have 
\begin{align*}
	(V^*)_2RV_1&=[\sum_{p,r,s} V_{s,r}^*(E_{p,p}\otimes E_{r,s})][\sum_{k,l,i,j} R_{kl,ij}(E_{k,i}\otimes E_{l,j})][\sum_{x,y,z} V_{x,y}(E_{x,y}\otimes E_{z,z})]\\
	&= \sum_{p,r,s}\sum_{k,l,i,j}\sum_{x,y,z}\delta_{p,k}\delta_{i,x}\delta_{s,l}\delta_{j,z}  R_{kl,ij}V_{s,r}^*V_{x,y}(E_{p,y}\otimes E_{r,z})\\
	&=\sum_{k,y,r,j}\sum_{l,i}  R_{kl,ij}V_{l,r}^*V_{i,y}(E_{k,y}\otimes E_{r,j})\\
	&= \sum_{i,j,r,s}[\sum_{k,l} R_{ik,ls} V^*_{kj}V_{lr}](E_{i,r}\otimes E_{j,s})
\end{align*}

\begin{align*}
	V_1R(V^*)_2&=[\sum_{x,y,z} V_{x,y}(E_{x,y}\otimes E_{z,z})][\sum_{k,l,i,j} R_{kl,ij}(E_{k,i}\otimes E_{l,j})][\sum_{p,r,s} V_{s,r}^*(E_{p,p}\otimes E_{r,s})]\\
	&= \sum_{p,r,s}\sum_{k,l,i,j}\sum_{x,y,z}\delta_{y,k}\delta_{i,p}\delta_{z,l}\delta_{j,r}  R_{kl,ij}V_{x,y}V_{s,r}^*(E_{x,p}\otimes E_{z,s})\\
	&=\sum_{x,i,l,s}[\sum_{k,j} R_{kl,ij}V_{x,k}V_{s,j}^*](E_{x,i}\otimes E_{l,s})\\
	&=\sum_{i,j,r,s}[\sum_{k,l} R_{kj,rl}V_{i,k}V_{s,l}^*](E_{i,r}\otimes E_{j,s})
\end{align*}

Therefore we have $[\sum_{k,l} R_{ik,ls} V^*_{kj}V_{lr}]= [\sum_{k,l} R_{kj,rl}V_{i,k}V_{s,l}^*]$ for all $i,j,r,s$.


\begin{thebibliography}{Asd}
	\bibitem[Bra89]{bragiel1989twisted}
	Kazimierz Bragiel. 
	\newblock{The twisted $SU(3)$ group. Irreducible $*$-representations of the $C^*$-algebra $C(S_{\mu}U(3))$.}
	\newblock{ \em Lett Math Phys 17, 37–44 (1989).}
	
	\bibitem[CDV02]{convio02}
	Alian Connes and Michel Dubois-Violette.
	\newblock Noncommutative finite-dimensional manifolds i. spherical manifolds and
	related examples.
	\newblock {\em Communications in mathematical physics, 230(3):539–579, 2002.}
	
	\bibitem[FRT88]{faddeev1988quantization}
	Ludwig~D Faddeev, N~Yu Reshetikhin, and Leon~A Takhtajan.
	\newblock Quantization of lie groups and lie algebras.
	\newblock In {\em Algebraic analysis}, pages 129--139. Elsevier, 1988.
	
	\bibitem[GS21]{guinsaurabh21}
	Satyajit Guin and Bipul Saurabh.
	\newblock Representations and classification of the compact quantum groups
	$U_q(2)$ for complex deformation parameters.
	\newblock {\em International Journal of Mathematics, 32, 04, 2021}
	
	\bibitem[GS23]{guin2023equivariant}
	Satyajit Guin and Bipul Saurabh.
	\newblock Equivariant spectral triple for the quantum group {${ U}_q(2)$} for complex deformation parameters
	\newblock {\em Journal of Geometry and Physics, 185, 104748, 2023.}
	
	\bibitem[Hay90]{hayashi90}
	Takahiro Hayashi.
	\newblock Quantum groups and quantum determinants.
	\newblock {\em Journal of algebra, 152, pages 146-165(1992)}, 1990.
	
	\bibitem[Jan24]{jana2024some}
	Debabrata Jana
	\newblock{On some algebraic and geometric aspects of the quantum unitary group},
	\newblock{\em Proceedings-Mathematical Sciences, 134, 2, 39, 2024.}
 	
	\bibitem[KMRW16]{kasprzak2016braided}
	Pawe{\l} Kasprzak, Ralf Meyer, Sutanu Roy, Stanis{\l}aw~Lech Woronowicz.
	\newblock Braided quantum {${ SU}(2)$} groups.
	\newblock {\em J. Noncommut. Geom}, 10(4):1611--1625, 2016.
	
	\bibitem[Koe91]{koelink91}
	H.~Tjerk Koelink.
	\newblock On $ast$-representations of the hopf star-algebra associated with the
	quantum group {${ U}_q(n)$}.
	\newblock {\em {\em Compositio Math.}}, 77(2):(199--231), 1991.
	
	\bibitem[KS97]{bookklisch97}
	Anatoli Klimyk and Konrad Schm{\"u}dgen.
	\newblock {\em Quantum {G}roups and {T}heir {R}epresentations}.
	\newblock Texts and Monographs in Physics. Springer Berlin, Heidelberg, first
	edition, 1997.
	
	\bibitem[Kul15]{kula2015woronowicz}
	Anna Kula
	\newblock{Woronowicz construction of compact quantum groups for functions on permutations. Classification result for N= 3.}
	\newblock{Journal of Mathematical Analysis and Applications, 421(2):1673–1712, 2015.}
	
	\bibitem[Lin04]{lining2004multi}
	Jiang Lining.
	\newblock{Multi-parameter compact matrix quantum group with generators of norm one},
	\newblock{\em Boletim da Sociedade Paranaense de Matematica, 22(2), 35-42, 2004.}
	
	\bibitem[GJZ02]{guo02}
	Maozheng Guo, Lining Jiang, and Ervin Yunwei Zhao.
	\newblock{The construction of braid Hopf algebra}.
	\newblock{\em communications in algebra, 30(4), pages 1725-1750(2002)}
	
	\bibitem[MRW16]{meyer2016quantum}
	Ralf Meyer, Sutanu Roy, and Stanis{\l}aw~Lech Woronowicz.
	\newblock Quantum group-twisted tensor products of $C^*$-algebras. ii.
	\newblock {\em Journal of Noncommutative Geometry}, 10(3):859--888, 2016.
	
	\bibitem[Phi06]{phillips2006every}
	N~Christopher Phillips.
	\newblock Every simple higher dimensional noncommutative torus is an at algebra.
	\newblock {\em arXiv preprint math/0609783}, 2006.

	
	\bibitem[QQG92]{qian1992new}
	Zhao-hui Qian, Min Qian, and Maozheng Guo.
	\newblock{A new type of Hopf algebra which is neither commutative nor cocommutative},
	\newblock{Journal of Physics A: Mathematical and General, 5(5):1237–1245, 1992.}
	
	\bibitem[Wor87a]{wor87}
	Stanis{\l}aw~L Woronowicz.
	\newblock Twisted {${ SU}(2)$} group. {A}n example of a noncommutative
	differential calculus.
	\newblock {\em {\em Publ. Res. Inst. Math. Sci.}}, 23:(117--181), 1987.
	
	\bibitem[Wor87b]{wor87cmp}
	Stanis{\l}aw~L Woronowicz.
	\newblock Compact matrix pseudogroups.
	\newblock {\em Communications in Mathematical Physics}, 111(4):613--665, 1987.
	
	\bibitem[ZZ05]{zhangzhao02}
	Ervin Yunwei  Zhao and Xiao Xia~Zhang.
	\newblock The compact quantum group $U_q(2)$ (i).
	\newblock {\em Linear Algebra and its Applications 408(2005), pages 244-258},
	2005.
	
\end{thebibliography}
\end{document}